\documentclass{amsart}
\usepackage{amssymb}
\usepackage{latexsym}
\usepackage{amsmath,}

\def\wt{\widetilde}
\def\wh{\widehat}
\def\ov{\overline}
\def \im{{\rm Im}}
 \def\up{\upharpoonright}
\def\cH{\mathcal H} \def\cB{\mathcal B}

\def\cK{\mathcal K} \def\cL{\mathcal L}
 \def\cN{\mathcal N}  
 \def\cT{\mathcal T} \def\cI{\mathcal I}

\def \gH{\mathfrak H}   \def \gN{\mathfrak N}

\def \bC{\mathbb C}    \def\bR{\mathbb R}
\def\bH{\mathbb H} 
\def \l{\lambda}
\def \a{\alpha} \def \b{\beta}  \def \L{\Lambda}  \def \s{\sigma}
 \def \t{\theta} \def\g {\gamma}
\def\d {\delta}   
\def \f{\varphi}  \def \G{\Gamma} \def\D {\Delta} \def\Si{\Sigma}

\def \C{\widetilde {\mathcal C}}
\def \CA{\C(\cH_0,\cH_1)}

\def \cd {\cdot}

\def\HH {\cH_0\oplus\cH_1}
\def\AC {AC(\cI; \bH)}   \def\LI {L_\Delta^2(\cI)}
\def\lI {\cL_\Delta^2(\cI)}
\def\LSH {L^2(\Sigma ;\cH)} \def\LS {L^2(\Sigma_\tau ; H_0)}
\def\St{\Sigma_\tau(\cdot)}
\def\tma{\cT_{\max}} \def\tmi{\cT_{\min}} \def\Tma{T_{\max}} \def\Tmi{T_{\min}}
\def\Tt {\wt T^\tau}

\def \dom {{\rm dom}\,}  \def \ran {{\rm ran}\,}  \def \ker{{\rm ker\,}}
 \def \mul {{\rm mul}\,}

\def \exa { {Ext}_A}
  \def\tm{\times}

\def  \RH {\wt R (\cH_0,\cH_1)}
\def  \RP {\wt R (\wt\cH_b,\cH_b)}

 \def \RZ {\wt R^0 (\cH_0,\cH_1)}
\def  \Rh {\wt R (\cH)} 
\def \pair {\tau=\{\tau_+,\tau_-\}}
\def \CR {\bC\setminus\bR}

\newcommand {\lo}[1] {\cL_\D^2[#1,\bH ]}

\def\bt{\{\cH,\G_0,\G_1\}}
\def\bta{\{\cH_0\oplus \cH_1,\Gamma _0,\Gamma _1\}}

\newtheorem{theorem}{Theorem}[section]
\newtheorem{proposition}[theorem]{Proposition}
\newtheorem{corollary}[theorem]{Corollary}
\newtheorem{lemma}[theorem]{Lemma}

\theoremstyle{definition}

\theoremstyle{definition}
\newtheorem {definition} [theorem]{Definition}
\theoremstyle{remark}
\newtheorem{remark}[theorem]{Remark}
\numberwithin{equation}{section}
\begin{document}
\title[On Titchmarsh-Weyl functions ]
{On Titchmarsh-Weyl functions and eigenfunction expansions of
first-order symmetric systems }

\author{Sergio Albeverio}
\address{Institut f\"ur Angevandte Mathematic,HCM, IZKS, SFB611,
 Universit\"at Bonn, Endenicherallee 60, D - 53115 Bonn, Germany;  }
\address{CERFIM, Locarno, Switzerland}
\email{albeverio@uni-bonn.de}
\author{Mark Malamud}
\address{Institute of Applied Mathematics and Mechanics, NAS of
Ukraine,  R. Luxemburg Str. 74,  83050 Donetsk,    Ukraine}
\email{mmm@telenet.dn.ua}
\author {Vadim Mogilevskii}
\address{Department of Mathematical Analysis, Lugans'k National University,
 Oboronna Str. 2, 91011  Lugans'k,   Ukraine}
\email{vim@mail.dsip.net}




\begin{abstract}
We study general (not necessarily Hamiltonian) first-order
symmetric  systems $J y'(t)-B(t)y(t)=\D(t) f(t)$ on an interval
$\cI=[a,b\rangle $ with the regular endpoint $a$. It is assumed
that the deficiency indices $n_\pm(\Tmi)$ of the   minimal
relation $\Tmi$ in $\LI$ satisfy $n_-(\Tmi)\leq n_+(\Tmi)$. By
using a Nevanlinna boundary parameter $\tau=\tau(\l)$ at the
singular endpoint $b$ we define self-adjoint and $\l$-depending
Nevanlinna boundary conditions which are analogs  of separated
self-adjoint boundary conditions for Hamiltonian systems. With a
boundary value problem involving such conditions we associate the
$m$-function $m(\cd)$, which is an analog of the Titchmarsh-Weyl
coefficient for the Hamiltonian system. By using $m$-function we
obtain the Fourier transform $V:\LI\to L^2(\Si)$ with the spectral
function $\Si(\cd)$ of the minimally possible dimension. If $V$ is
an isometry, then the (exit space) self-adjoint extension $\wt T$
of $\Tmi$ induced by the boundary problem is unitarily equivalent
to the multiplication operator in $L^2(\Si)$; hence the spectrum
of $\wt T$ is defined by the spectral function $\Si(\cd)$. We show
that all the objects of the boundary problem are determined by the
parameter $\tau$, which  enables us to parametrize all spectral
function $\Si(\cd) $ immediately in terms of $\tau$. Similar
results for various classes of boundary problems were obtained by
Kac and Krein, Fulton , Hinton and Shaw  and other authors.
\end{abstract}
\maketitle
\section{Introduction}
Let  $H$ and $\wh H$ be finite dimensional Hilbert spaces and let
\begin{gather}\label{1.1}
H_0:=H\oplus\wh H, \qquad \bH:=H_0\oplus H=H\oplus\wh H \oplus H.
\end{gather}

The main object of the paper is  first-order symmetric  system of
differential equations defined on an interval $\cI=[a,b\rangle,
-\infty<a <b\leq\infty,$ with the regular endpoint $a$ and regular
or singular endpoint $b$. Such a system is of the form
\cite{Atk,GK}
\begin {equation}\label{1.2}
J y'(t)-B(t)y(t)=\D(t) f(t), \quad t\in\cI,
\end{equation}
where $B(t)=B^*(t)$ and $\D(t)\geq 0$ are the $[\bH]$-valued
functions on $\cI$ and
\begin {equation} \label{1.3}
J=\begin{pmatrix} 0 & 0&-I_H \cr 0& i I_{\wh H}&0\cr I_H&
0&0\end{pmatrix}:H\oplus\wh H\oplus H \to H\oplus\wh H\oplus H.
\end{equation}
Throughout the paper we  assume  that the system \eqref{1.2} is
definite. The latter means  that   for any $\l\in\bC$ each common
solution of the equations
\begin {equation}\label{1.5}
J y'(t)-B(t)y(t)=\l \D(t) y(t)
\end{equation}
and $\D(t)y(t)=0$  (a.e. on $\cI$) is trivial, i.e.,
$y(t)=0,\;t\in\cI$.

System \eqref{1.2} is called Hamiltonian system if $\wh H=\{0\}$.
In this case one has
\begin {equation}\label{1.5a}
J=\begin{pmatrix}  0&-I_H \cr  I_H& 0\end{pmatrix}:H\oplus H \to
H\oplus H.
\end{equation}

In what follows  we denote by $\LI$ the Hilbert space of
$\bH$-valued Borel measurable functions  $f(\cd)$ (in fact,
equivalence classes) on $\cI$ satisfying $||f||_\D^2:=
\int\limits_\cI (\D(t)f(t),f(t))_\bH\,dt<\infty$.

Investigations of symmetric systems  is motivated by several
reasons.  For instance, systems \eqref{1.5} form more general
objet than formally self-adjoint differential equation of
arbitrary order with matrix coefficients. Such  equation is
reduced  to a system of the form \eqref{1.5} with $J$ given by
\eqref{1.3} (see \cite{KogRof75}). Emphasize that presence  of the
term $i I_{\wh H}$ in \eqref{1.3} under this reduction
characterizes  odd order equations, although  even order equations
are reduced to Hamiltonian systems (with  $J$ given by
\eqref{1.5a}). Moreover,  the Krein-Feller string equation is also
reduced to  Hamiltonian  system \eqref{1.5} (\cite[Chapter 6, \S
8]{GK}).

As it is known,  the extension theory of symmetric linear
relations gives a natural framework   for investigation of the
boundary value problems for symmetric systems (see
\cite{BHSW10,DLS88,DLS93,HSW00,Kac03,LT82,LesMal03,Orc} and
references therein). According to \cite{Kac03,LesMal03, Orc}  the
system \eqref{1.2} generates the minimal linear relation $\Tmi$
and the maximal linear relation $\Tma$ in $\LI$. It turns out that
$\Tmi$ is a closed symmetric relation with not necessarily equal
deficiency indices $n_\pm(\Tmi)$. Since system \eqref{1.2} is
assumed to be  definite,  $n_\pm(\Tmi)$ can be defined as a
number of $L_\D^2$-solutions of  \eqref{1.5} for $\l\in\bC_\pm$.
Moreover, $\Tma=\Tmi^*$ and the equality
\begin {equation}\label{1.6}
[y,z]_b=\lim _{t\uparrow b} (J y(t), z(t)), \quad y,z\in\dom\Tma,
\end{equation}
defines a skew-Hermitian bilinear form on  the domain of $\Tma$.

A description of various classes of extensions of $\Tmi$
(self-adjoint, $m$-dissipative, etc.) in terms of boundary
conditions is an important problem in the spectral theory of
symmetric systems. Assume that the system \eqref{1.2} is
Hamiltonian and $n_+(\Tmi)=n_-(\Tmi)$. Let $
y(t)=\{y_0(t),y_1(t)\}(\in H\oplus H)$ be  the representation of a
function $y\in\dom\tma$. Then according to \cite{HinSch93} the
general form of self-adjoint separated boundary conditions is
\begin {equation}\label{1.7}
\cos B_1 \, y_0(a)+\sin B_1 \, y_1(a)=0, \;\; [y,\chi_j]_b=0, \;\;
j\in \{1,\,...\,,\,\nu_b\},\;\; y\in\dom\tma,
\end{equation}
where $B_1$ is a self-adjoint operator on $H$ and
$\{\chi_j\}_1^{\nu_b}, \nu_b=n_\pm(\Tmi) -\dim H,$ is a certain
system of functions from $\dom \Tma$. The  vector
$y_b:=\{[y,\chi_j]_b\}_1^{\nu_b}\in \bC^{\nu_b}$ is called a
singular boundary value   of a function $y\in\dom\tma$. Observe
that for ordinary differential operators description  \eqref{1.7}
goes back to I.M. Glazman (see \cite[Appendix 2,\S 5] {AkhGla}),
while the form of the boundary conditions (at regular endpoints)
goes back to F.S. Rofe-Beketov \cite{Rof69}. Note also that the
notion of a singular boundary value can be found in the book
\cite[Ch.13.2] {DunSch}).

Boundary conditions \eqref{1.7} generate a self-adjoint extension
$\wt A$ of $\Tmi$ given by $\wt A=\{\{y,f\}\in\Tma:\, y \;\;
\text{satisfies \eqref{1.7}}\}$. The resolvent $(\wt A-\l)^{-1}$
of $\wt A$ is defined as follows:  for any
$f\in\LI$ vector $y=(\wt A-\l)^{-1}f$ is the $L_\D^2$-solution of the equation  
\begin {equation}\label{1.7a}
J y'(t)-B(t)y(t)=\l\D (t)y+\D(t) f(t), \quad f\in\LI, \quad \l\in
\CR,
\end{equation}
subject to the boundary conditions \eqref{1.7}. Moreover,
according to \cite{HinSch93} the Titchmarsh - Weyl coefficient
$M_{TW}(\l)(\in [H])$ of  the boundary problem \eqref{1.7a},
\eqref{1.7} is defined by the relations
\begin {gather*}
v(t,\l):=\f(t,\l)M_{TW}(\l)+\psi(t,\l)\in \lo{H},\\
[v(\cd,\l)h,\chi_j]_b=0,\quad h\in H, \quad j\in
\{1,\;\dots,\;\nu_b\}.
\end{gather*}
Here $\f(\cd,\l)$ and $\psi(\cd,\l)$ are the
$[H,\bH\oplus\bH]$-valued operator solutions of Eq. \eqref{1.5}
with the initial data
$$
\f(a,\l)= (\sin B_1,\, -\cos B_1)^\top \quad \text{and} \quad \psi
(a,\l)= (-\cos B_1,\, \sin B_1)^\top.
$$
Note also the papers \cite{Khr06,Kra89}, where   the Titchmarsh -
Weyl coefficient for Hamiltonian systems is defined in another
way. By using $M_{TW}(\cd)$ one obtains the Fourier transform with
the spectral function $\Si(\cd)$ of the minimally possible
dimension $N_\Si=\dim H$ (see \cite{DLS88,DLS93,HinSch98,Kac03}).

It turns out that for general (not necessarily Hamiltonian)
symmetric systems the situation is more complicated. In
particular, it was shown in \cite{Mog11} that
\emph{non-Hamiltonian system \eqref{1.2} does not admit separated
self-adjoint boundary conditions}. Moreover,  the inequality
$n_+(\Tmi)\neq n_-(\Tmi)$, and hence absence  of self-adjoint
boundary conditions is a typical situation for such systems.  For
instance, in the limit point case at $b$ one has $n_+(\Tmi)=\dim
H$ and $n_-(\Tmi)=\dim H+\dim \wh H$. Such circumstances make it
natural to investigate  the following problems:

$\bullet$\; To find (might be $\l$-depending) analogs of
self-adjoint separated boundary conditions for general systems
\eqref{1.2} and describe such type conditions;


$\bullet$\; To describe in terms of boundary conditions all
spectral matrix functions that have the minimally possible
dimension and  investigate the corresponding Fourier transforms.


In the  paper we solve these problems for symmetric systems
\eqref{1.2} assuming that $n_-(\Tmi)\leq n_+(\Tmi)$.  However to
simplify presentation we assume within this section that
$n_-(\Tmi)= n_+(\Tmi)$ (the case $n_+(\Tmi)< n_-(\Tmi)$ will be
treated elsewhere). We first show that there exists a
finite-dimensional Hilbert space $\cH_b$  and a surjective linear
mapping
\begin{gather*}
\G_b=(\G_{0b},\,  \wh\G_b,\,  \G_{1b})^\top:\dom\Tma\to
\cH_b\oplus\wh H\oplus \cH_b
\end{gather*}
such that the bilinear form \eqref{1.6} admits the representation
\begin {equation*}
[y,z]_b=(\G_{0b}y,\G_{1b}z)-(\G_{1b}y,\G_{0b}z)+ i (\wh\G_b y,
\wh\G_b z), \quad y,z\in\dom\Tma.
\end{equation*}
It turns out that $\G_b y$ can be chosen in the form of a singular
boundary value of  $y\in\dom\Tma$ (see Remark \ref{rem3.2a}).
Moreover, each proper extension of $\Tmi$ can be defined by means
of boundary conditions imposed on vectors $y(a)=\{y_0(a), \, \wh
y(a), \, y_1(a)\}(\in H\oplus\wh H\oplus H)$ and $\G_b
y=\{\G_{0b}y,\,  \wh\G_by,\,  \G_{1b}y\}(\in \cH_b\oplus\wh
H\oplus \cH_b)$. In particular, a linear relation $T$ given by
\begin {equation*}
T:=\{\{ y,  f\}\in\Tma: \, y_1(a)=0, \;\wh y(a)=\wh\G_b y,\;
\G_{0b}y =\G_{1b}y=0 \},
\end{equation*}
is a symmetric extension of  $\Tmi$  and  plays a crucial  role in
our considerations.

Recall that a generalized resolvent of $T$ is an operator-valued
function given by
\begin {equation*}
R(\l)=P_{\LI}(\wt T -\l)^{-1}\up \LI, \quad \l\in\CR,
\end{equation*}
where $\wt T$ is an exit space  self-adjoint extension of $T$
acting in a wider Hilbert space $\wt \gH \supset \LI$.  Moreover,
the spectral function of $T$ is defined by
\begin {equation*}
F(t)=P_{\LI}E(t)\up \LI,\quad t\in\bR,
\end{equation*}
where $E(\cdot)$ is the orthogonal spectral function (resolution
of identity) of $\wt T$. We show that each generalized resolvent
$y=R(\l)f, \; f\in\LI,$ is given as the $\LI$-solution of the
following boundary-value  problem with $\l$-depending boundary
conditions:
\begin{gather}
J y'-B(t)y=\l \D(t)y+\D(t)f(t), \quad t\in\cI,\label{1.10}\\
y_1(a)=0, \qquad \wh  y (a)= \wh\G_b y,\label{1.11}\\
C_0(\l)\G_{0b}y+C_1(\l)\G_{1b}y=0, \quad \l\in\CR. \label{1.12}
\end{gather}
Here $C_0(\cdot)$ and  $C_1(\cdot)$ are the components of a
Nevanlinna operator pair $\tau
(\cdot)=\{(C_0(\cdot),C_1(\cdot))\}$ with values in
$[\cH_b]\oplus[\cH_b]$, so that formula \eqref{1.12} defines a
Nevanlinna boundary condition at the singular endpoint $b$. One
may consider a pair $\tau=\tau(\cdot)$ as a boundary parameter,
since $R(\l)$ runs over the set of all generalized resolvents of
$T$ when $\tau$ runs over the set of all Nevanlinna operator
pairs. To indicate this fact explicitly we write
$R(\l)=R_\tau(\l)$ and $F(t)=F_\tau(t)$ for the generalized
resolvents and spectral functions of $T$ respectively. Moreover,
we denote by $\wt T=\wt T^\tau$ the exit space self-adjoint
extension of $T$ generating $R_\tau(\cd)$ and $F_\tau(\cd)$.

The boundary value problem \eqref{1.10}-\eqref{1.12} defines a
canonical resolvent $R_\tau(\l)$ if and only if $\tau$ is a
self-adjoint operator pair $\tau=\{(\cos B,\sin B)\}$ with some
$B=B^*\in [\cH_b]$. In this case $ R_\tau(\l)=(\wt T^\tau -
\l)^{-1}, \; \l\in\CR,$ where  $\wt T^\tau$ is a self-adjoint
extension of $T$ in $\LI$ defined by the following mixed boundary
conditions :
\begin {equation}\label{1.14}
\wt T^\tau =\{\{y,f\}\in\Tma : y_1(a)=0, \quad \wh  y(a)= \wh\G_b
y,\quad \cos B\cd\G_{0b}y+\sin B\cd\G_{1b}y=0\}.
\end{equation}
For  Hamiltonian systems  the equalities in the right-hand side of
\eqref{1.14} take the form of self-adjoint separated boundary
conditions
\begin {gather}\label{1.15}
y_1(a)=0, \quad \cos B\cd\G_{0b}y+\sin B\cd\G_{1b}y=0.
\end{gather}
Formula \eqref{1.15} seems to be more convenient than \eqref{1.7},
because it enables one to parametrize  singular self-adjoint
boundary conditions (at the endpoint $b$) by means of a
self-adjoint boundary parameter $B$.

Next assume that $\f(\cd,\l)$ and $\psi (\cd,\l)$ are
$[H_0,\bH]$-valued operator solutions of  equation  \eqref{1.5}
satisfying  the initial conditions
\begin {equation*}
\f(a,\l)=\begin{pmatrix} I_{H_0}\cr 0\end{pmatrix}(\in [H_0,
H_0\oplus H]),\quad \psi(a,\l)=\begin{pmatrix} -\tfrac i 2 P_{\wh
H}\cr -P_H\end{pmatrix}(\in [H_0, H_0\oplus H]).
\end{equation*}
 We show that,  for each Nevanlinna boundary parameter $\tau =
\{(C_0(\l),C_1(\l))\}$, there exists a unique operator function
$m_\tau(\l)(\in [H_0])$ such that the operator solution
\begin {equation*}
v_\tau(t,\l):=\f(t,\l)m_\tau(\l) +\psi (t,\l), \quad \l\in\CR,
\end{equation*}
of Eq. \eqref{1.5} has the following property: for every $h_0\in
H_0$ the function $y=v_\tau (t,\l)h_0$ belongs to $\LI$ and
satisfies the boundary conditions
\begin{equation*}
i(\wh y(a) - \wh\G_b y )=P_{\wh H} h_0 ,\qquad
C_0(\l)\G_{0b}y+C_1(\l)\G_{1b}y=0.
\end{equation*}
We call $m_\tau(\cd)$  the $m$-function corresponding to the
boundary problem \eqref{1.10}-\eqref{1.12}. It turns out that
$m_\tau(\cd) $ is a Nevanlinna operator function satisfying the
inequality
\begin {equation*}
(\im \,\l)^{-1}\cd \im\, m_\tau(\l)\geq \int_\cI
v_\tau^*(t,\l)\D(t) v_\tau(t,\l)\, dt, \quad\l\in\CR.
\end{equation*}
Moreover, in the case of the Hamiltonian system the  $m$-function
of the ''canonical'' boundary problem \eqref{1.10}, \eqref{1.15}
coincides with the Titchmarsh - Weyl coefficient $M_{TW}(\cd)$ in
the sense of \cite{HinSch93,Kra89,Khr06}. Note also that a concept
of the Titchmarsh - Weyl function for the general system
\eqref{1.2} with separated $\l$-depending boundary conditions was
proposed  in \cite{Khr06}. This function is no longer a Nevanlinna
function, that does not allow one to define  the spectral function
of the corresponding boundary value problem (cf. \eqref{1.21}
below).

In the final part of the paper we study eigenfunction expansions
of the boundary value  problems for symmetric systems. Namely, let
$\tau  = \{(C_0(\cdot),C_1(\cdot))\}$ be a boundary parameter and
let $F_\tau(\cd)$ be the spectral function of $T$ generated by the
boundary value problem \eqref{1.10}--\eqref{1.12}. A nondecreasing
left-continuous operator-valued function $\Si_\tau(\cd):\bR\to
[H_0]$ is called a spectral function of this problem  if, for each
function $f\in\LI$ with compact support, the Fourier transform
\begin {equation}\label{1.14A}
\wh f(s)=\int _\cI \f^*(t,s)\, \D(t)\,f(t)\, dt
\end{equation}
satisfies
\begin {equation}\label{1.20}
((F_\tau(\b)-F_\tau(\a))f,f)_{\LI}=\int_{[\a,\b)}
(d\Si_{\tau}(s)\wh f(s), \wh f(s))
\end{equation}
for any compact interval $[\a,\b)\subset\bR$. We show that for
each boundary parameter $\tau$ there exists unique spectral
function $\Si_\tau(\cd)$ and it is recovered  from the
$m$-function $m_\tau(\cd)$ by means of the Stieltjes inversion
formula
   \begin {equation}\label{1.21}
\Si_\tau(s)=\lim\limits_{\delta \to +0}\lim\limits_{\varepsilon
\to +0}\frac 1 \pi \int_{-\delta}^{s-\delta} \im\, m_\tau(\sigma
+i\varepsilon)\, d\sigma.
  \end{equation}
Below (within this section) we assume  for simplicity that $T$ is
a  (not necessarily densely defined)
operator, i.e., $\mul T=\{0\}$. 

It follows from \eqref{1.20} that,  the mapping  $V f=\wh f$,
originally defined by \eqref{1.14A} for functions with compact
supports, admits a continuous extension to a contractive map $V:
\LI \to \LS$ (for the strict definition of the Hilbert space $\LS$
see \cite{DunSch, Kac50, MalMal03} and also Section
\ref{sect6.2}).
In the following theorem we characterize the most interesting case
when the mapping $V$ is isometric.
\begin{theorem}\label{th1.1}
For each boundary parameter $\tau$ the following statements are
equivalent:

{\rm (i)} The Fourier transform $V$ is an isometry from $\LI$ to
$\LS$ or, equivalently, the Parseval equality $||\wh f||_{\LS}=||
f||_{\LI}$ holds for every  $f\in\LI$.

{\rm (ii)} The exit space self-adjoint extension $\wt T^\tau$ (in
$\wt\gH$) is the operator, that is $\mul \wt T^\tau=\{0\}$.

If $(i)$ (hence  $(ii)$) is valid, then:

{\rm (1)} For each $f\in\LI$ the inverse Fourier transform is
given by
\begin {equation*}
f(t)=\int _\bR \f(t,s) d\Si_\tau(s)\wh f(s)
\end{equation*}
where the integral is understood in an appropriate sense.

{\rm (2)} There exists a unitary extension  $U$ of the operator
$V$ that maps  $\wt\gH$ onto $\LS$ and such that the operator
$\Tt$ is unitarily equivalent to  the multiplication operator $\L$
on  $\LS$,  $\Tt =U^*\L U$. Hence, the operators $\Tt$ and $\L$
have the same spectral properties; for instance, the  multiplicity
of spectrum of $\Tt$ does not exceed $\dim H_0(=\dim H+\dim \wh
H)$.
\end{theorem}

It follows from Theorem \ref{th1.1} that $V$ is a unitary operator
from $\LI$ onto $\LS$ if and only if $\tau =\{\cos B, \sin B\}$ is
a selfadjoint operator pair and the self-adjoint extension
\eqref{1.14} of $T$  is the operator. Observe also that the
statements (i) and  (ii) hold for any boundary parameter $\tau $
if and only if $T$ is a densely defined operator.

Next, we show that all spectral functions $\Si_\tau(\cd)$ can  be
parametrized immediately in terms of the  boundary parameter
$\tau$. More precisely the following theorem holds.
\begin{theorem}\label{th1.2}
There exists a Nevanlinna operator function
\begin {equation}\label{1.23}
M(\l)=\begin{pmatrix} m_0(\l) & M_2(\l) \cr M_3(\l) & M_4(\l)
\end{pmatrix}:H_0\oplus\cH_b\to H_0\oplus\cH_b, \quad \l\in\CR,
\end{equation}
such that for each Nevanlinna boundary parameter $\tau
=\{(C_0(\cdot), C_1(\cdot))\}$ the corresponding $m$-function
$m_\tau(\cd)$ is  given by
\begin {equation}\label{1.24}
m_\tau(\l)=m_0(\l)+M_2(\l)(C_0(\l)-C_1(\l)M_4(\l))^{-1}C_1(\l)M_3(\l),
\quad\l\in\CR .
\end{equation}
Thus, formula \eqref{1.24} together with the Stieltjes inversion
formula \eqref{1.21} defines  (unique) spectral function
$\Si_\tau(\cd)$ of the boundary problem
\eqref{1.10}--\eqref{1.12}. Moreover, the Fourier transform $V$ is
an isometry if and only if the following two conditions are
fulfilled:
\begin{gather}
\lim\limits_{y\to \infty} \tfrac 1 y (C_0(i y)-C_1(i y)M_4(i
y))^{-1}C_1(i
y)=0, \label{1.25} \\
\lim\limits_{y\to \infty} \tfrac 1 y M_4(i y) (C_0(i y)-C_1(i
y)M_4(i y))^{-1}C_0(i y)=0.\label{1.26}
\end{gather}
\end{theorem}
Note that a description of spectral functions for various classes
of boundary problems in the form close to \eqref{1.24},
\eqref{1.21} can be found in
\cite{Ful77,Gor66,HinSha82,Hol85,KacKre,Mog07}.

The above results are obtained in the framework of the new
approach to the extension theory of symmetric operators developed
during three last decades (see \cite{DM00, DM91, DM95, GorGor,
Mal92,MalMog02,MalNei12, Mog06.2} and references therein). This
approach is based on concepts of boundary triplets and the
corresponding Weyl functions. To apply this method to boundary
value problems for system \eqref{1.2} we construct  an appropriate
boundary triplet for the relation $\Tma$ (see  Proposition
\ref{pr3.3}). Moreover, in Proposition \ref{pr4.1a} and Corollary
\ref{cor4.1b}  we express the corresponding Weyl function
$M(\cdot)$ in the sense of \cite{DM91,Mal92,Mog06.2} in terms of
the boundary values of respective matrix solutions of \eqref{1.5}.
It is worth to mention that the operator-valued function
\eqref{1.23} coincides with the Weyl function $M(\cdot)$  computed
in Corollary \ref{cor4.1b}. Note also that  conditions
\eqref{1.25}, \eqref{1.26} are implied by general result on
$\Pi$-admissibility from \cite{DM00, DM09}.

We complete the paper by explicit  example illustrating the main
results.

Some results of the paper have been  published as a  preprint
\cite{AlbMalMog12}.

\section{Preliminaries}
\subsection{Notations}
The following notations will be used throughout the paper: $\gH$,
$\cH$ denote Hilbert spaces; $[\cH_1,\cH_2]$  is the set of all
bounded linear operators defined on the Hilbert space $\cH_1$ with
values in the Hilbert space $\cH_2$; $[\cH]:=[\cH,\cH]$; $A\up
\cL$ is the restriction of an operator $A$ onto the linear
manifold $\cL$; $P_\cL$ is the orthogonal projector in $\gH$ onto
the subspace $\cL\subset\gH$; $\bC_+\,(\bC_-)$ is the upper
(lower) half-plane  of the complex plane.

Recall that a closed linear   relation   from $\cH_0$ to $\cH_1$
is a closed linear subspace in $\cH_0\oplus\cH_1$. The set of all
closed linear relations from $\cH_0$ to $\cH_1$ (in $\cH$) will be
denoted by $\C (\cH_0,\cH_1)$ ($\C(\cH)$). A closed linear
operator $T$ from $\cH_0$ to $\cH_1$  is identified  with its
graph $\text {gr}\, T\in\CA$.

For a linear relation $T\in\C (\cH_0,\cH_1)$  we denote by $\dom
T,\,\ran T, \,\ker T$ and $\mul T$  the domain, range, kernel and
the multivalued part of $T$ respectively. Recall also that the
inverse and adjoint linear relations of $T$ are the relations
$T^{-1}\in\C (\cH_1,\cH_0)$ and $T^*\in\C (\cH_1,\cH_0)$ defined
by
\begin{gather}
T^{-1}=\{\{h_1,h_0\}\in\cH_1\oplus\cH_0:\{h_0,h_1\}\in T\}\nonumber\\
T^* = \{\{k_1,k_0\}\in \cH_1\oplus\cH_0:\, (k_0,h_0)-(k_1,h_1)=0,
\; \{h_0,h_1\}\in T\}\label{2.0}.
\end{gather}

  In the
case $T\in\CA$ we write $0\in \rho (T)$ if $\ker T=\{0\}$\ and\
$\ran T=\cH_1$, or equivalently if $T^{-1}\in [\cH_1,\cH_0]$;
$0\in \wh\rho (T)$\ \ if\ \ $\ker T=\{0\}$\ and\   $\ran T$ is a
closed subspace in $\cH_1$. For a linear relation $T\in \C(\cH)$
we denote by $\rho (T):=\{\l \in \bC:\ 0\in \rho (T-\l)\}$ and
$\wh\rho (T)=\{\l \in \bC:\ 0\in \wh\rho (T-\l)\}$ the resolvent
set and the set of regular type points  of $T$ respectively.

A linear relation $T\in\C (\cH)$ is called symmetric
(self-adjoint) if $T\subset T^*$ (resp. $T=T^*$). For each
$T=T^*\in \C (\cH)$ the following decompositions hold
\begin {equation}\label{2.0.1}
\cH=\cH'\oplus \mul T, \qquad T= T'\oplus\wh {\mul} T,
\end{equation}
where $\wh {\mul} T =\{0\}\oplus \mul T$ and $T'$ is the
self-adjoint operator in $\cH'$ (the operator part of $T$).

Let $T=T^*\in \C (\cH)$, let $\cB$ be the Borel $\s$-algebra of
$\bR$ and let $E'(\cd):\cB\to [\cH']$ be the orthogonal spectral
measure of $T'$. Then the spectral measure $E(\cd)$ of $T$ is
defined as $E(\d)=E'(\d)P_{\cH'}, \; \d\in\cB$.

Recall also the following definition.
\begin{definition}\label{def2.0}
A holomorphic operator function $\Phi (\cd):\bC\setminus\bR\to
[\cH]$ is called a Nevanlinna function  if $\im\, \l\cd \im \Phi
(\l)\geq 0 $ and $\Phi ^*(\l)= \Phi (\ov \l), \;
\l\in\bC\setminus\bR$.
\end{definition}
\subsection{Holomorphic operator pairs }
Let  $\Lambda$ be an open set in $\bC$, let $\cK,\cH_0,\cH_1$  be
Hilbert spaces and let
\begin {equation*}
(C_0(\l), \,C_1(\l)):\cH_0\oplus\cH_1\to\cK, \quad \l\in\Lambda,
\end{equation*}
be a pair of holomorphic operator functions $C_j(\cd):\Lambda \to
[\cH_j,\cK], \; j\in \{0,1\}$  (in short a holomorphic pair). Two
such pairs $C_j(\cd):\Lambda\to [\cH_j,\cK]$ and
$C_j'(\cd):\Lambda\to [\cH_j,\cK']$ are said to be equivalent if
there exists a holomorphic isomorphism $\f(\cd): \Lambda\to [\cK,
\cK']$ such that $C_j'(\l)=\f (\l)C_j(\l), \l\in\Lambda,\;j\in
\{0,1\}$. Clearly, the set of all holomorphic pairs splits into
disjoint equivalence classes; moreover, the equality
\begin {equation}\label{2.1}
\tau (\l)=\{(C_0(\l), C_1(\l));\cK\}:=\{\{h_0,h_1\}\in
\cH_0\oplus\cH_1: C_0(\l)h_0+C_1(\l)h_1=0\}
\end{equation}
allows us to identify such a class with the $\CA$-valued function
$\tau(\l), \; \l\in\Lambda $.

In what follows, unless otherwise stated, $\cH_0$ is a Hilbert
space, $\cH_1$ is a subspace in $\cH_0$,
$\cH_2:=\cH_0\ominus\cH_1$ and $P_j$ is the orthoprojector in
$\cH_0$ onto $\cH_j,\; j\in\{1,2\}$.

 With each linear relation $\t\in\CA$ we associate the
$\tm$-adjoint linear relation $\t^\times\in\CA$ given by
\begin {equation*}
\t^\times=\{\{k_0,k_1\}\in \cH_0\oplus\cH_1:(k_1,h_0)-(k_0,h_1)+i
(P_2 k_0,P_2 h_0)=0,\; \{h_0,h_1\}\in\t\}.
\end{equation*}
It follows from \eqref{2.0} that  in the case $\cH_0=\cH_1=:\cH$
one has $\t^\times=\t^*$.

Next assume that
\begin {equation}\label{2.2}
\begin{array}{c}
\tau_+(\l)=\{(C_0(\l),C_1(\l));\cH_0\}, \;\;\l\in\bC_+; \\
\tau_-(\l)=\{(D_0(\l),D_1(\l));\cH_1\}, \;\;\l\in\bC_-
\end{array}
\end{equation}
are  equivalence classes of the holomorphic pairs
\begin {gather}
(C_0(\l), C_1(\l)):\HH\to\cH_0,\;\;\;\;\l\in\bC_+ \label{2.3}\\
(D_0(\l), D_1(\l)):\HH\to\cH_1,\;\;\;\;\l\in\bC_-.\label{2.3a}
\end{gather}
Assume also that
\begin {gather*}
C_0(\l)=(C_{01}(\l),C_{02}(\l)):\cH_1\oplus\cH_2\to\cH_0\\
D_0(\l)=(D_{01}(\l),D_{02}(\l)):\cH_1\oplus\cH_2\to\cH_1
\end{gather*}
are the block representations of $C_0(\l)$ and $D_0(\l)$.
\begin{definition}\label{def2.1}
 A collection $\pair$ of two holomorphic pairs
\eqref{2.2} (more precisely, of the equivalence classes of the
corresponding pairs) belongs to the class $\RH$ if it satisfies
the following relations:
\begin{gather}
2\,\im(C_{1}(\l)C_{01}^*(\l))+ C_{02}(\l)C_{02}^*(\l)\geq 0,\label{2.6}\\
\quad 2\,\im(D_1(\l)D_{01}^*(\l))+ D_{02}(\l)D_{02}^*(\l)\leq
0,  \label{2.6a}\\
C_1(\l)D_{01}^*(\ov\l)-C_{01}(\l)D_{1}^*(\ov\l)+i
C_{02}(\l)D_{02}^* (\ov\l)=0,
\;\;\; \l\in\bC_+\label{2.7}\\
0\in \rho (C_0(\l)-iC_1(\l)P_1), \;\l\in\bC_+;\;\;\;0\in\rho
(D_{01}(\l)+iD_{1}(\l)),\;\l\in\bC_-.\label{2.8}
\end{gather}

A collection $\pair\in\RH$ belongs to the class $\RZ$ if for some
(and hence for any) $\l\in\bC_+$ one has
\begin{equation*}
2\,\im(C_{1}(\l)C_{01}^*(\l))+ C_{02}(\l)C_{02}^*(\l)= 0
\;\;\;\;\text{and}\;\;\;\;
 0\in\rho(C_{01}(\l)+iC_{1}(\l)).
\end{equation*}
\end{definition}
The following proposition is immediate from Definition
\ref{def2.1} and the results of \cite{Mog06.1}.
\begin{proposition}\label{pr2.2}
{\rm (1)} If $\pair\in\RH$, then $(-\tau_\pm
(\ov\l))^\tm=-\tau_\mp(\l),$ $\;\l\in\bC_\mp,$ and the following
equality holds
\begin {equation}\label{2.8b}
\tau_\mp(\l)=\{\{-h_1-i P_2 h_0, -P_1h_0\}:\{h_1,h_0\}\in
(\tau_\pm (\ov\l))^*\}.
\end{equation}

{\rm (2)} The set $\RZ$ is not empty if and only if $\dim\cH_0=
\dim\cH_1$. This implies that in the case $\dim\cH_1<\infty$ the
set $\RZ$ is not empty if and only if $\cH_0= \cH_1=:\cH$.

{\rm (3)} Each collection $\pair\in\RZ$ can be represented as a
constant
\begin {equation}\label{2.10}
\tau_\pm(\l)\equiv \{(C_0,C_1);\cH_0\}=\t(\in\CA),\quad
\l\in\bC_\pm,
\end{equation}
where $C_j\in [\cH_j,\cH_0],\; j\in\{0,1\},$ and   $(-\t)^\tm
=-\t$.
\end{proposition}
Moreover, one can easily prove the following proposition.
\begin{proposition}\label{pr2.3}
If $\dim\cH_0<\infty$, then a collection $\pair$ of two
holomorphic pairs \eqref{2.2} belongs to the class $\RH$ if and
only if \eqref{2.6}--\eqref{2.7} holds and
\begin {equation}\label{2.11}
\ran (C_0(\l), C_1(\l))=\cH_0,\;\;\l\in\bC_+; \quad \ran (D_0(\l),
D_1(\l))= \cH_1,\;\;\l\in\bC_-.
\end{equation}
\end{proposition}
\begin{remark}\label{rem2.4}
 If $\cH_1=\cH_0=:\cH$, then the class $\Rh:=\wt R (\cH,\cH)$ coincides with
the well-known class of Nevanlinna functions $\tau (\cd)$ with
values in $\C (\cH)$ (see, for instance, \cite{DM00}). In this
case the collection \eqref{2.2} turns into the Nevanlinna pair
\begin {equation}\label{2.19}
\tau(\l)=\{(C_0(\l),C_1(\l));\cH\}, \quad \l\in\CR,
\end{equation}
with $C_0(\l), \; C_1(\l)\in [\cH]$. In view of
\eqref{2.6}--\eqref{2.8} such a pair is characterized by the
relations (cf. \cite[Definition 2.2]{DM00})
\begin{gather}
\im \l\cd \im (C_1(\l)C_0^*(\l))\geq 0,\quad C_1(\l)C_0^*(\ov\l)-
C_0(\l)C_1^*(\ov\l)=0,  \;\; \l\in\CR,\label{2.20}\\
0\in\rho (C_0(\l)-iC_1(\l)), \;\; \l\in\bC_+; \qquad 0\in\rho
(C_0(\l)+ iC_1(\l)), \;\; \l\in\bC_-.
\end{gather}
Moreover, the function $\tau(\cd)$ belongs to the class $\wt
R^0(\cH):=\wt R^0(\cH,\cH)$ if and only if it admits the
representation in the form of the constant (cf. \eqref{2.10})
\begin {equation}\label{2.22}
\tau(\l)\equiv \{(C_0,C_1);\cH\}=\t(\in\C (\cH)),\quad \l\in\CR
\end{equation}
with the operators $C_j\in [\cH]$ such that $\im (C_1C_0^*)=0 $
and $0\in\rho (C_0\pm i C_1)$ (this means that $\t=\t^*$). Observe
also that according to \cite{Rof69} each $\tau\in\wt R^0(\cH)$
admits the normalized representation \eqref{2.22} with
\begin {equation}\label{2.23}
C_0=\cos B, \qquad C_1=\sin B, \qquad B=B^*\in [\cH].
\end{equation}

Assume now that $n:=\dim \cH<\infty$, $e=\{e_j\}_1^n$ is an
orthonormal basis in $\cH$,  $\tau (\l)=\{(C_0 (\l),C_1(\l));
\cH\}$ is a pair of holomorphic operator-functions
$C_l(\cd):\CR\to [\cH]$ and $C_l(\l)= (c_{kj,l} (\l) )_{k,j=1}^n $
is the matrix representations of the operator $C_l(\l),\;
l\in\{0,1\},$ in the basis $e$. Then by Proposition \ref{pr2.3}
$\tau$ belongs to the class $\wt R(\cH)$ if and only if the
matrices $C_0(\l)$ and $C_1(\l)$ satisfy \eqref{2.20} and the
following equality:
\begin {equation*}
\text {rank} \,(C_0(\l),C_1(\l))=n, \quad \l\in\CR.
\end{equation*}
Moreover, the operator pair $\t=\{(C_0,C_1);\cH\}$ belongs to the
class $\wt R^0(\cH)$ if and only if  $\im (C_1C_0^*)=0$ and
$\text{rank}\, (C_0,C_1)=n$ (here $C_l=(c_{kj,l} )_{k,j=1}^n$ is
the matrix representation of the operator $C_l, \; l\in \{0,1\},$
in the basis $e$). Note that such a ''matrix'' definition of the
classes $\wt R(\cH)$ and $\wt R^0(\cH)$ in the case $\dim
\cH<\infty$ can be found, e.g. in \cite{DLS93,Kov83}
\end{remark}

\subsection{Boundary triplets and Weyl functions}
Here we recall  definitions of boundary triplets, the
corresponding Weyl functions, and $\gamma$-fields following
\cite{DM91, DM95, Mal92, Mog06.2}.

 Let $A$ be a closed  symmetric linear relation in the Hilbert space $\gH$,
let $\gN_\l(A)=\ker (A^*-\l)\; (\l\in\wh\rho (A))$ be a defect
subspace of $A$, let $\wh\gN_\l(A)=\{\{f,\l f\}:\, f\in
\gN_\l(A)\}$ and let $n_\pm (A):=\dim \gN_\l(A)\leq\infty, \;
\l\in\bC_\pm,$ be deficiency indices of $A$. Denote by $\exa$ the
set of all proper extensions of $A$, i.e., the set of all
relations $\wt A\in \C (\gH)$ such that $A\subset\wt A\subset
A^*$.

 Next assume that $\cH_0$ is a Hilbert space,  $\cH_1$ is a subspace
in $\cH_0$ and   $\cH_2:=\cH_0\ominus\cH_1$, so that
$\cH_0=\cH_1\oplus\cH_2$. Denote by $P_j$ the orthoprojector  in
$\cH_0$ onto $\cH_j,\; j\in\{1,2\} $.

\begin{definition}\label{def2.5}
 A collection $\Pi=\bta$, where
$\G_j: A^*\to \cH_j, \; j\in\{0,1\},$ are linear mappings, is
called a boundary triplet for $A^*$, if the mapping $\G :\wh f\to
\{\G_0 \wh f, \G_1 \wh f\}, \wh f\in A^*,$ from $A^*$ into
$\cH_0\oplus\cH_1$ is surjective and the following Green's
identity
\begin {equation}\label{2.28}
(f',g)-(f,g')=(\G_1  \wh f,\G_0 \wh g)_{\cH_0}- (\G_0 \wh f,\G_1
\wh g)_{\cH_0}+i (P_2\G_0 \wh f,P_2\G_0 \wh g)_{\cH_2}
\end{equation}
 holds for all $\wh
f=\{f,f'\}, \; \wh g=\{g,g'\}\in A^*$.
\end{definition}
\begin{proposition}\label{pr2.7}
Let  $\Pi=\bta$ be a boundary triplet for $A^*$. Then:
\begin{enumerate}
\item
$\;\dim \cH_1=n_-(A)\leq n_+(A)=\dim \cH_0$.
\item
$\ker \G_0\cap\ker\G_1=A$ and $\G_j$ is a bounded operator from
$A^*$ into $\cH_j, \;  j\in\{0,1\}$.
\item
The equality
\begin {equation}\label{2.31}
A_0:=\ker \G_0=\{\wh f\in A^*:\G_0 \wh f=0\}
 \end{equation}
defines the maximal symmetric extension $A_0\in\exa$ such that
$\bC_+\subset \rho (A_0) $.
\end{enumerate}
\end{proposition}
\begin{proposition}\label{pr2.8}$\,$ \cite{Mog06.2}
Let  $\Pi=\bta$ be a boundary triplet for $A^*$. Denote also by
$\pi_1$ the orthoprojector  in $\gH\oplus\gH$ onto $\gH\oplus
\{0\}$. Then the operators $\G_0\up \wh \gN_\l (A), \;\l\in\bC_+,$
and $P_1\G_0\up \wh \gN_z (A),\; z\in\bC_-,$ isomorphically map
$\wh\gN_\l (A)$ onto $\cH_0$ and $\wh\gN_z(A)$ onto $\cH_1$
respectively. Therefore the equalities
\begin {equation}\label{2.32}
\begin{array}{c}
\g_{+} (\l)=\pi_1(\G_0\up\wh \gN_\l (A))^{-1}, \;\;\l\in\Bbb C_+,\\
 \g_{-} (z)=\pi_1(P_1\G_0\up\wh\gN_z (A))^{-1}, \;\; z\in\Bbb C_-,
\end{array}
\end{equation}
\begin{gather}
M_{+}(\l)h_0=\G_1\{\g_+(\l)h_0, \l\g_+(\l)h_0\}, \quad
h_0\in\cH_0, \quad
\l\in\bC_+\label{2.32a}\\
M_{-}(z)h_1=(\G_1+iP_2\G_0)\{\g_-(z)h_1, z\g_-(z)h_1\}, \quad
h_1\in\cH_1, \quad z\in\bC_- \label{2.32b}
\end{gather}
correctly define the operator functions $\g_{+}(\cdot):\Bbb
C_+\to[\cH_0,\gH], \; \; \g_{-}(\cdot):\Bbb C_-\to[\cH_1,\gH]$ and
$M_{+}(\cdot):\bC_+\to [\cH_0,\cH_1], \;\; M_{-}(\cdot):\bC_-\to
[\cH_1,\cH_0]$, which are holomorphic on their domains. Moreover,
the equality $M_+^*(\ov\l)=M_-(\l), \;\l\in\bC_-,$ is valid.
\end{proposition}
It follows from \eqref{2.32} that for each $h_0\in\cH_0$ and
$h_1\in\cH_1$ the following equalities hold
\begin {equation}\label{2.33}
\G_0\{\g_+(\l)h_0,\l \g_+(\l)h_0\}=h_0,  \qquad
P_1\G_0\{\g_-(z)h_1, z \g_-(z)h_1\}=h_1.
\end{equation}

\begin{definition}\label{def2.10}$\,$\cite{Mog06.2}
The operator functions $\g_\pm(\cd)$ and $M_\pm(\cd)$ defined in
Proposition \ref{pr2.8}  are called the $\g$-fields and the Weyl
functions, respectively, corresponding to the boundary triplet
$\Pi$.
\end{definition}
\begin{proposition}\label{pr2.10a}
Let $\Pi=\bta$ be a boundary triplet for $A^*$ and let
$\g_\pm(\cd)$ and $M_\pm(\cd)$ be the corresponding $\g$-fields
and Weyl functions respectively. Moreover, let the spaces $\cH_0$
and $\cH_1$ be decomposed as
\begin {equation*}
\cH_1=\wh\cH\oplus\dot\cH_1, \qquad \cH_0=\wh\cH\oplus\dot\cH_0
\end{equation*}
(so  that $\dot\cH_0=\dot\cH_1\oplus\cH_2$) and let
\begin {equation*}
\G_0=( \wh\G_0 , \dot\G_0)^\top : A^* \to \wh\cH\oplus\dot\cH_0,
\qquad \G_1=( \wh\G_1 , \dot\G_1 )^\top: A^* \to
\wh\cH\oplus\dot\cH_1
\end{equation*}
be the block representations of the operators $\G_0$ and $\G_1$.
Then:

{\rm (1)} The equality
\begin {equation*}
\wt A=\{\wh f\in A^*: \wh\G_0 \wh f=\dot\G_0\wh f= \dot\G_1\wh
f=0\}
\end{equation*}
defines a closed symmetric extension $\wt A\in\exa$ and the
adjoint relation $\wt A^*$ of $\wt A$ is
\begin {equation*}
\wt A^*=\{\wh f\in A^*: \wh\G_0 \wh f=0\}.
\end{equation*}

If in addition $n_\pm (A)<\infty$, then the deficiency indices of
$\wt A$ are $n_\pm (\wt A)=n_\pm (A)-\dim \wh \cH$.

{\rm  (2)} The collection $\dot\Pi=\{\dot\cH_0\oplus\dot\cH_1,
\dot \G_0\up \wt A^*, \dot \G_1\up \wt A^*\}$  is a boundary
triplet for $\wt A^*$.

{\rm (3)} The $\g$-fields $\dot\g_\pm(\cd)$ and the Weyl functions
$\dot M_\pm (\cd)$ corresponding to $\dot\Pi$ are given by
\begin{gather*}
\dot\g_+(\l)=\g_+(\l)\up \dot\cH_0, \qquad \dot
M_+(\l)=P_{\dot\cH_1}M_+(\l)\up
\dot\cH_0, \quad \l\in\bC_+\\
\dot\g_-(\l)=\g_-(\l)\up \dot\cH_1, \qquad \dot
M_-(\l)=P_{\dot\cH_0}M_-(\l)\up \dot\cH_1, \quad \l\in\bC_-.
\end{gather*}
\end{proposition}
We omit the proof of Proposition \ref{pr2.10a}, since it is
similar to that of Proposition 4.1 in \cite{DM00} (see also remark
\ref{rem2.10b} below).
\begin{remark}\label{rem2.10b}
If $\cH_0=\cH_1:=\cH$, then   the boundary triplet in the sense of
Definition \ref{def2.5} turns into the boundary triplet
$\Pi=\{\cH,\G_0,\G_1\}$ for $A^*$ in the sense of
\cite{GorGor,Mal92}.In this case $n_+(A)=n_-(A)=\dim \cH$,
$\,A_0(=\ker \G_0 )$ is a self-adjoint extension of $A$ and
according
 to \cite{DM91,Mal92,DM95} the relations
\begin {equation}\label{2.34a}
\begin{array}{c}
\g(\l)=\pi_1(\G_0\up\wh\gN_\l(A))^{-1}, \\
 \G_1\{\g(\l)h, \l \g(\l)h\}=
M(\l)h,\quad h\in\cH, \quad \l\in\rho (A_0)
\end{array}
\end{equation}
define the $\g$-field $\g(\cd):\rho (A_0)\to [\cH,\gH]$ and the
Weyl function $M(\cd):\rho (A_0)\to [\cH]$ corresponding to the
triplet $\Pi$. It follows from \eqref{2.34a} that $\g(\cd)$ and
$M(\cd)$ are associated
 with the operator functions  $\g_\pm(\cd)$ and $M_\pm(\cd)$ from Definition
 \ref{def2.10}  via $\g(\l)=\g_{\pm}(\l)$ and $M(\l)=M_{\pm}(\l),
\;\l\in\bC_\pm $. Moreover, for such a triplet the identity
\begin {equation}\label{2.34b}
M(\mu)-M^*(\l)=(\mu-\ov\l)\g^*(\l)\g(\mu), \quad \mu,\l\in\CR.
\end{equation}
holds, which implies that $M(\cd)$ is a Nevanlinna operator
function. Observe also that for the triplet $\Pi=\bt$ all the
results in this subsection were obtained in
\cite{DM91,Mal92,DM95,DM00}.

In what follows a boundary triplet $\Pi=\bt$ in the sense  of
\cite{GorGor,Mal92} will be sometimes called an ordinary boundary
triplet for $A^*$.
\end{remark}
\subsection{Generalized resolvents and spectral functions}
Let $\gH$ be a subspace in  a Hilbert space $\wt\gH$, let $\wt
A=\wt A^*\in\C (\wt \gH)$ and let $E(\cd)$ be the spectral measure
of $\wt A$.
\begin{definition}\label{def2.10c}
The relation $\wt A$ is called $\gH$-minimal if it satisfies at
least one of the following equivalent conditions:

(1) $\text{span} \{\gH,(\wt A-\l)^{-1}\gH: \l\in\CR\}=\wt\gH$;

(2) there is not a nontrivial subspace $\gH'\subset
\wt\gH\ominus\gH$ such that $E([\a,\b))\gH'\subset\gH'$ for each
bounded interval $[\a,\b)\subset\bR$.
\end{definition}

\begin{definition}\label{def2.11}
The relations $T_j\in \C (\gH_j), \; j\in\{1,2\},$ are said to be
unitarily equivalent (by means of a unitary operator $U\in
[\gH_1,\gH_2]$) if $T_2=\wt U T_1$ with $\wt U=U\oplus U \in
[\gH_1^2, \gH_2^2]$.
\end{definition}
\begin{proposition}\label{pr2.11.1}
Let $\gH_j$ be a subspace in a Hilbert space $\wt\gH_j$  and let
$\wt A_j=\wt A_j^*\in\C (\wt\gH_j)$  be a $\gH_j$-minimal
relation, $j\in \{1,2\}$. Assume also that $V\in [\gH_1,\gH_2]$ is
a unitary operator such that
\begin {equation*}
P_{\gH_1}(\wt A_1-\l)^{-1}\up\gH_1 = V^*(P_{\gH_2}(\wt
A_2-\l)^{-1}\up\gH_2) V.
\end{equation*}
Then there exists a unitary operator $U\in [\wt\gH_1, \wt\gH_2]$
such that $U\up\gH_1=V$ and the relations $\wt A_1$ and $\wt A_2$
are unitarily equivalent by means of $U$.
\end{proposition}
In the case $\gH_1=\gH_2=:\gH$ and $V=I_{\gH}$ the proof of this
proposition can be found in \cite{LanTex77}. In general case the
proof is similar.

Recall further the following definition.
\begin{definition}\label{def2.11.2}
Let $A$ be a symmetric relation in a Hilbert space $\gH$.  The
operator functions $R(\cd):\CR\to [\gH]$ and $F(\cd):\bR\to [\gH]$
are called the generalized resolvent and the spectral function of
$A$ respectively if there exist a Hilbert space $\wt
\gH\supset\gH$ and a self-adjoint  relation $\wt A\in \C (\wt\gH)$
such that $A\subset \wt A$ and the following equalities hold:
\begin {gather}
R(\l) =P_\gH (\wt A- \l)^{-1}\up \gH, \quad \l \in \CR\label{2.36}\\
F(t)=P_{\gH}E((-\infty,t))\up\gH, \quad  t\in\bR\label{2.37}
\end{gather}
(in formula \eqref{2.37} $E(\cd)$ is the spectral measure of $\wt
A$).

The relation $\wt A$ in \eqref{2.36} is called an exit space
extension of $A$.
\end{definition}
It follows from \eqref{2.36} and \eqref{2.37} that the generalized
resolvent $R(\cd)$ and the spectral function $F(\cd)$ generated by
the same extension $\wt A$ of $A$ are connected by
\begin {equation}\label{2.38}
R(\l)=\int_{\bR}\frac {d F(t)} {t-\l}, \quad \l\in\bR.
\end{equation}
Moreover, \eqref{2.37} yields
\begin {equation}\label{2.39}
F(\infty)(:=s-\lim\limits_{t\to +\infty}F(t) )=P_\gH
P_{\wt\gH_0}\up\gH,
\end{equation}
where $\wt\gH_0=\wt\gH\ominus \mul \wt A$.

According to \cite{LanTex77} each generalized resolvent  of $A$ is
generated by some $\gH$-minimal exit space extension  $\wt A$ of
A. Moreover,  if the $\gH$-minimal exit space  extensions $\wt
A_1\in\C (\wt\gH_1)$ and $\wt A_2\in\C (\wt\gH_2)$ of $A$ induce
the same generalized resolvent $R(\l)$, then in view of
Proposition \ref{pr2.11.1} there exists a unitary operator $V'\in
[\wt\gH_1\ominus \gH, \wt\gH_2\ominus \gH]$ such that  $\wt A_1$
and $\wt A_2$ are unitarily equivalent by means of
$U=I_{\gH}\oplus V'$.  By using this fact we suppose in the
following that the exit space extension $\wt A$ in \eqref{2.36} is
$\gH$-minimal, so that  $\wt A$  is defined by $R(\cd)$ uniquely
up to the unitary equivalence.
\begin{definition}\label{def2.11.3}
The generalized resolvent \eqref{2.36} and the spectral function
\eqref{2.37} are called canonical if $\wt \gH=\gH$, i.e., if
$R(\l)=(\wt A-\l)^{-1}, \; \l\in\CR,$ is the resolvent of the
extension $\wt A=\wt A^* \in \C (\gH)$ of $A$ and
$F(t)=E((-\infty,t)), \; t\in\bR,$ is the spectral function  of
$\wt A $.
\end{definition}
Clearly, canonical resolvents and spectral functions  exist if and
only if $n_+(A)=n_-(A)$.

A description of all generalized resolvents of $A$ in terms of
boundary triplets for $A^*$ is given in the following theorem (see
\cite{Bru76,Mal92} for the case $n_+(A)=n_-(A)$ and \cite{Mog06.2}
for the case of arbitrary deficiency indices $n_\pm(A)$).
\begin{theorem}\label{th2.12}
Let $\Pi=\bta$ be a boundary triplet for $A^*$. If $\pair\in\RH$
is a collection of holomorphic pairs \eqref{2.2}, then for every
$g\in\gH$ and $\l\in\CR$ the abstract boundary value problem
\begin{gather}
\{f,\l f+g\}\in A^*\label{2.45}\\
C_0(\l)\G_0\{f,\l f+g\}-C_1(\l)\G_1\{f,\l f+g\}=0, \quad
\l\in\bC_+
\label{2.46}\\
D_0(\l)\G_0\{f,\l f+g\}-D_1(\l)\G_1\{f,\l f+g\}=0, \quad
\l\in\bC_-\label{2.47}
\end{gather}
has a unique solution $f=f(g,\l)$ and the equality
$R(\l)g:=f(g,\l)$ defines a generalized resolvent $R(\l)=R_\tau
(\l)$ of the relation $A$. Conversely, for each generalized
resolvent $R(\l)$ of $A$ there exists a unique $\tau\in\RH$ such
that $R(\l)=R_\tau (\l)$. Moreover, $R_\tau(\l)$ is a canonical
resolvent if and only if $\tau\in\RZ$.


\end{theorem}

\section{Boundary triplets for symmetric systems}
\subsection{Notations}
Let $\cI=[ a,b\rangle\; (-\infty < a< b\leq\infty)$ be an interval
of the real line (the symbol $\rangle$ means that the endpoint
$b<\infty$  might be either included  to $\cI$ or not). Further,
let $\bH$ be a finite-dimensional Hilbert space, let $\AC$ be the
set of functions $f(\cd):\cI\to \bH$ which are absolutely
continuous on each segment $[a,\b]\subset \cI$ and let $AC
(\cI):=AC(\cI;\bC)$. Denote also by $\cL_{loc}^1(\cI; [\bH])$ the
set of  Borel operator-valued  functions $F(\cd)$ defined almost
everywhere on $\cI$ with values in $[\bH]$ and such that
$\int\limits_{[a,\b]}||F(t)||\,dt<\infty$ for each $\b\in \cI$.

Next assume that $\D(\cd)\in \cL_{loc}^1(\cI;[\bH])$ is an
operator function such that $\D(t)\geq 0 $ a.e. on $\cI$. Denote
by $\lI$  the linear space of all Borel-measurable
vector-functions $f(\cd): \cI\to \bH$ satisfying
$$
\int\limits_{\cI}(\D (t)f(t),f(t))_\bH \,dt =
\int\limits_{\cI}||\D^{\frac 1 2} (t)f(t)||^2\,dt <\infty.
$$
Moreover, for a given finite-dimensional Hilbert space $\cK$
denote by $\lo{\cK}$ the set of all Borel operator-functions
$F(\cd): \cI\to [\cK,\bH]$ such that $F(t)h\in \lI$ for each
$h\in\cK$. It is clear that the latter condition is equivalent to
$\int\limits_{\cI}||\D^{\frac 1 2} (t)F(t)||^2\,dt <\infty$.

It  is known \cite{Kac50, DunSch, MalMal03} that $\lI$ is a
semi-Hilbert space with the semi-definite inner product
$(\cd,\cd)_\D$ and the semi-norm $||\cd||_\D$ given by
\begin {equation}\label{3.0}
(f,g)_\D=\int_{\cI}(\D (t)f(t),g(t))_\bH \,dt, \;\;
||f||_\D=((f,f)_\D)^{\frac1 2 }, \;\; f,g\in \lI.
\end{equation}
The semi-Hilbert space $\lI$ gives rise to the quotient Hilbert
space $\LI=\lI /
\{f\in\lI: ||f||_\D=0\}$. 
The inner product and the norm in $\LI$ are defined by
\begin {equation*}
(\wt f, \wt g)=(f,g)_\D, \quad ||\wt f||=(\wt f, \wt f)^{\frac 1
2}=||f||_\D, \qquad \wt f, \wt g\in\LI,
\end{equation*}
respectively, where $f\in\wt f \; (g\in\wt g)$ is any
representative of the class $\wt f$ (resp. $\wt g$).

In the sequel we systematically use the quotient map $\pi$ from
$\lI$ onto $\LI$ given by $\pi f=\wt f(\ni f), \; f\in \lI$.
Moreover, we let $\wt\pi=\pi\oplus\pi: (\lI)^2 \to (\LI)^2$, so
that $\wt \pi\{f,g\}=\{\wt f, \wt g\}, \;\; f,g \in \lI$.

\subsection{Symmetric systems}
In this subsection we provide some known results on symmetric
systems of differential equations following  \cite{GK, Kac03,
KogRof75, LesMal03, Orc}.

Let as above $\cI=[ a,b\rangle \;(-\infty < a <b\leq\infty )$ be
an interval and let $\bH$ be a Hilbert space with $n:=\dim
\bH<\infty$. Moreover, let $B(\cd), \D (\cd)\in \cL_{loc}^1(\cI;
[\bH])$ be operator functions such that $B(t)=B^*(t)$ and
$\D(t)\geq 0$ a.e. on $\cI$ and let $J\in [\bH]$ be a signature
operator ( this means that $J^*=J^{-1}=-J$).

A first-order symmetric  system on an interval $\cI$ (with the
regular endpoint $a$) is a system of differential equations of the
form
\begin {equation}\label{3.1}
J y'(t)-B(t)y(t)=\D(t) f(t), \quad t\in\cI,
\end{equation}
where $f(\cd)\in \lI$. Together with \eqref{3.1} we consider also
the homogeneous  system
\begin {equation}\label{3.2}
J y'(t)-B(t)y(t)=\l \D(t) y(t), \quad t\in\cI, \quad \l\in\bC.
\end{equation}
A function $y\in\AC$ is a solution of \eqref{3.1} (resp.
\eqref{3.2}) if the equality \eqref{3.1} (resp. \eqref{3.2} holds
a.e. on $\cI$. Moreover, a function $Y(\cd,\l):\cI\to [\cK,\bH]$
is an operator solution of the equation \eqref{3.2} if
$y(t)=Y(t,\l)h$ is a (vector) solution of this equation for each
$h\in\cK$ (here $\cK$ is a Hilbert space with $\dim\cK<\infty$).

In what follows  we always assume  that  system \eqref{3.1} is
definite in the sense of the following definition.
 \begin{definition}\label{def3.1}$\,$\cite{GK,KogRof75}
The symmetric system \eqref{3.1} is called definite if for each
$\l\in\bC$ and each solution $y$ of \eqref{3.2} the equality
$\D(t)y(t)=0$ (a.e. on $\cI$) implies $y(t)=0, \; t\in\cI$.
\end{definition}
As it is known \cite{Orc, Kac03, LesMal03} symmetric system
\eqref{3.1} gives rise to the \emph{maximal linear relations}
$\tma$ and $\Tma$  in  $\lI$ and $\LI$, respectively. They are
given by
\begin {equation}\label{3.4}
\begin{array}{c}
\tma=\{\{y,f\}\in(\lI)^2 :y\in\AC \;\;\text{and}\;\; \qquad\qquad\qquad\qquad \\
\qquad\qquad\qquad\qquad\qquad\quad  J y'(t)-B(t)y(t)=\D(t)
f(t)\;\;\text{a.e. on}\;\; \cI \}
\end{array}
\end{equation}
%
and $\Tma=\wt\pi \tma$. Moreover the Lagrange's identity
\begin {equation}\label{3.6}
(f,z)_\D-(y,g)_\D=[y,z]_b - (J y(a),z(a)),\quad \{y,f\}, \;
\{z,g\} \in\tma.
\end{equation}
holds with
\begin {equation}\label{3.7}
[y,z]_b:=\lim_{t \uparrow b}(J y(t),z(t)), \quad y,z \in\dom\tma.
\end{equation}
Formula \eqref{3.7} defines the boundary bilinear form
$[\cd,\cd]_b $ on $\dom \tma$, which plays a crucial  role in our
considerations. By using this form we define the \emph{minimal
relations} $\tmi$ in $\lI$ and $\Tmi$ in $\LI$ via
\begin {equation*}
\tmi=\{\{y,f\}\in\tma: y(a)=0 \;\; \text{and}\;\;
[y,z]_b=0\;\;\text{for each}\;\; z\in \dom \tma \}.
\end{equation*}
and $\Tmi= \wt\pi \tmi$. According to \cite{Orc, LesMal03} $\Tmi$
is a closed symmetric linear relation in $\LI$ and $\Tmi^*=\Tma$.
\begin{remark}\label{rem3.1a}
It  is known (see e.g. \cite{LesMal03}) that the maximal relation
$\Tma$ induced by the definite symmetric system \eqref{3.1}
possesses the following \emph{regularity property}: for each
$\{\wt y, \wt f \}\in \Tma $ there exists unique function $y\in
\AC \cap \lI $ such that $y\in \wt y$ and $\{y,f\}\in \tma$ for
each $f\in\wt f$. Below we associate such a function $ y\in \AC
\cap \lI$ with each pair $\{\wt y, \wt f\}\in\Tma$.
\end{remark}
For any  $\l\in\bC$ denote by $\cN_\l$ the linear space of
solutions of the homogeneous system \eqref{3.2} belonging to
$\lI$. Definition \eqref{3.4} of $\tma$ implies
\begin{equation*}
\cN_\l=\ker (\tma-\l)=\{y\in\lI:\; \{y,\l y\}\in\tma\},
\quad\l\in\bC,
\end{equation*}
and hence $\cN_\l\subset \dom\tma$.

As usual, denote by
\begin{gather*}
n_\pm (\Tmi )=\dim \gN_\l (\Tmi), \quad \l\in\bC_\pm,
\end{gather*}
the  deficiency indices of  $\Tmi$.  Since the system \eqref{3.1}
is definite, $ \pi\cN_\l=\gN_\l (\Tmi)$ and
$\ker(\pi\up\cN_\l)=\{0\},\;\; \l\in\bC$.  This implies that $\dim
\cN_\l=n_\pm (\Tmi), \; \l\in\bC_\pm$.

The following lemma will be useful in the sequel.
\begin{lemma}\label{lem3.2}
{\rm (1)} If $Y(\cd,\l)\in \lo{\cK}$ is an operator solution  of
Eq. \eqref{3.2}, then the relation
\begin {equation}\label{3.30}
\cK\ni h\to (Y(\l) h)(t)=Y(t,\l)h \in\cN_\l.
\end{equation}
defines the linear mapping $Y(\l):\cK\to \cN_\l$ and, conversely,
for each such a mapping $Y(\l)$ there exists  unique
operator-valued  solution $Y(\cd,\l)\in \lo{\cK}$ of equation
\eqref{3.2} such that \eqref{3.30} holds.

\rm{(2)} Let $Y(\cd,\l)\in \lo{\cK}$ be an operator solution of
Eq. \eqref{3.2} and let $F(\l)=\pi Y(\l)(\in [\cK, \LI])$. Then
for each $\wt f\in \LI$
\begin {equation}\label{3.31}
F^*(\l)\wt f=\int_\cI Y^*(t,\l)\D(t)f(t)\, dt, \quad f\in\wt f.
\end{equation}
\end{lemma}
The first statement of this lemma is obvious, while the second one
can be proved in the same way as formula (3.70) in \cite{Mog09.1}
(see also formula (2.40) in \cite{LesMal03}).

Let $J\in [\bH]$ be the signature operator  in \eqref{3.1} and let
\begin{gather*}
\nu_+=\dim\ker (i J-I) \;\;\; \text{and} \;\;\; \nu_-=\dim\ker (i
J+I).
\end{gather*}
In what follows we suppose that
\begin {equation}\label{3.12}
\wh\nu:=\nu_- - \nu_+\geq 0 .
\end{equation}
In this case one can assume without loss of generality that the
following statements hold:

(i) the Hilbert space $\bH$ is of the form
\begin{gather}\label{3.16}
\bH=H\oplus\wh H \oplus H,
\end{gather}
where $H$ and $\wh H$ are finite dimensional Hilbert spaces with
\begin {equation}\label{3.16a}
 \dim H=\nu_+, \qquad  \dim \wh H=\wh\nu;
\end{equation}

 (ii) the operator $J$ is of the form \eqref{1.3}.

Introducing the Hilbert space
\begin {equation} \label{3.17a}
H_0=H\oplus\wh H
\end{equation}
one can represent the equality \eqref{3.16} as
\begin {equation} \label{3.17b}
\bH=(H\oplus\wh H) \oplus H=H_0\oplus H.
\end{equation}

Let $\nu_{b+} $ and $\nu_{b-}$ be inertia indices of  the
skew-Hermitian bilinear form \eqref{3.7}. Then $\nu_{b\pm}<\infty$
and the following equalities hold \cite{BHSW10,Mog11}
\begin {equation} \label{3.17c}
n_+(\Tmi)=\nu_+ +\nu_{b+}, \qquad n_-(\Tmi)=\nu_- +\nu_{b-}.
\end{equation}
This yields the equivalence
\begin {equation} \label{3.17d}
n_+(\Tmi)=n_-(\Tmi)\iff\wh\nu=\nu_{b+}-\nu_{b-}.
\end{equation}

Next assume that
\begin {equation} \label{3.17.1}
U=\begin{pmatrix} u_1 & u_2 & u_3  \cr u_4 & u_5 & u_6
\end{pmatrix}: H\oplus\wh H\oplus H\to \wh H\oplus H
\end{equation}
is the operator  satisfying the relations
\begin{gather}
\ran U=\wh H\oplus H \label{3.17.2}\\
iu_2u_2^* - u_1u_3^*+u_3u_1^*=iI_{\wh H}, \qquad iu_5u_2^* -
u_4u_3^*+u_6u_1^*=0\label{3.17.3}\\
iu_5u_5^* + u_6u_4^*-u_4u_6^*=0\label{3.17.4}
\end{gather}
One can prove that the operator \eqref{3.17.1} admits an extension
to the $J$-unitary operator
\begin {equation} \label{3.17.5}
\wt U=\begin{pmatrix} u_7 & u_8 & u_9 \cr \hline  u_1 & u_2 & u_3
\cr u_4 & u_5 & u_6
\end{pmatrix}: H\oplus\wh H\oplus H\to H\oplus \wh H\oplus H,
\end{equation}
i.e. the operator satisfying  $\wt U^* J\wt U=J$. The operator
\eqref{3.17.5} induces the linear mapping $\G_a:\AC\to\bH$ given
by
\begin {equation} \label{3.21}
\G_ay=\wt U y(a), \quad y\in\AC.
\end{equation}
In accordance with the decomposition \eqref{3.16} $\G_a$ admits
the block representation
\begin {equation} \label{3.22}
\G_a=\left(\G_{0a},\, \wh\G_a,\, \G_{1a}\right)^\top :\AC\to
H\oplus \wh H\oplus H.
\end{equation}
If a function $y\in\AC$ is decomposed as
\begin {equation*}
y(t)=\{y_0(t),\,\wh y(t), \, y_1(t) \}(\in H\oplus\wh H \oplus H),
\quad t\in\cI,
\end{equation*}
then the mappings $\G_{ja}:\AC\to H, \; j\in \{0,1\},$ and
$\wh\G_{a}:\AC\to \wh H$ in \eqref{3.22} can be represented as
\begin{gather}
\G_{0a}y=u_7 y_0(a)+ u_8 \wh y(a)+ u_9 y_1(a),\quad y\in\AC\label{3.23}\\
\wh\G_{a}y=u_1 y_0(a)+ u_2 \wh y(a)+ u_3 y_1(a),\quad \G_{1a}y=u_4
y_0(a)+ u_5 \wh y(a)+ u_6 y_1(a). \label{3.24}
\end{gather}
This implies that $\wh\G_a$ and $\G_{1a}$ are determined by the
operator $U$, while $\G_{0a}$ is determined by the extension $\wt
U$.

Let $\l\in\bC$ and  $\cK$ be a finite-dimensional  Hilbert space.
By using the operator \eqref{3.17.5} we associate with each
operator solution $Y(\cd,\l):\cI\to [\cK,\bH]$ of  equation
\eqref{3.2} the operator $Y_a (\l)\in [\cK,\bH]$ given by
\begin {equation}\label{3.26}
Y_a(\l)=\wt U Y(a,\l).
\end{equation}
If in addition $Y(\cd,\l) \in \lo{\cK}$, then  the operator
\eqref{3.26} admits the representation
\begin {equation}\label{3.32}
Y_{a}(\l)=\G_{a} Y(\l),
\end{equation}
where $Y(\l)$ is defined in Lemma \ref{lem3.2}.

In what follows we associate with each operator $U$ (see
\eqref{3.17.1}) the operator solution $\f
(\cd,\l)=\f_U(\cd,\l)(\in [H_0,\bH]),\;\l\in\bC,$ of Eq.
\eqref{3.2} with the initial data
\begin {equation}\label{3.32.1}
\f_U(a,\l)= \begin{pmatrix} u_6^*  & iu_3^* \cr -iu_5^* & u_2^*
\cr -u_4^*  & -iu_1^*\end{pmatrix}:\underbrace { H\oplus\wh
H}_{H_0}\to \underbrace{H\oplus \wh H\oplus H}_{\bH}.
\end{equation}
One can easily verify that for each $J$-unitary extension $\wt U$
of $U$ the following equality holds
\begin {equation}\label{3.32.2}
\f_{U,a}(\l)(=\wt U \f_U(a,\l))=\begin{pmatrix}I_{H_0} \cr 0
\end{pmatrix}:H_0\to H_0\oplus H.
\end{equation}
The particular case of the operator $U$ and its $J$-unitary
extension $\wt U$ is (cf. \cite{HinSch06})
\begin {equation*}
U= \begin{pmatrix}  0& I_{\wh H} & 0 \cr  \cos B & 0 & \sin
B\end{pmatrix}, \quad  \wt U=\begin{pmatrix} \sin B & 0 & -\cos B
\cr 0& I_{\wh H} & 0 \cr \cos B & 0 & \sin B\end{pmatrix},
\end{equation*}
where $B=B^*\in [H]$. For such $U$ the  solution $\f_U(\cd,\l)$
is defined by the initial data
\begin {equation*}
\f_U(a,\l)=\begin{pmatrix} \sin B & 0  \cr 0& I_{\wh H}  \cr -\cos
B & 0
\end{pmatrix}: H\oplus\wh H\to H\oplus
\wh H\oplus H.
\end{equation*}

\subsection{Decomposing boundary triplets}\label{sub2.3}
We start with  the following lemma.
\begin{lemma}\label{lem3.2.1}
If $n_-(\Tmi)\leq n_+(\Tmi)$, then there exist a finite
dimensional  Hilbert space $\wt \cH_b$, a subspace $\cH_b\subset
\wt \cH_b$  and a surjective linear mapping
\begin{gather}\label{3.32.4}
\G_b=\begin{pmatrix}\G_{0b}\cr  \wh\G_b \cr  \G_{1b}\end{pmatrix}
:\dom\tma\to \wt \cH_b\oplus\wh H \oplus \cH_b
\end{gather}
such that for all $y,z \in \dom\tma$ the following identity is
valid
\begin {equation} \label{3.32.5}
\begin{array}{r}
 [y,z]_b=(\G_{0b}y,\G_{1b}z)_{\wt\cH_b}-(\G_{1b}y,\G_{0b}z)_{\wt\cH_b}+
 \qquad\qquad\qquad\qquad \\
\qquad\qquad\qquad\qquad +i(P_{\cH_b^\perp}\G_{0b}y,
P_{\cH_b^\perp}\G_{0b}z)_{\wt\cH_b}+ i (\wh\G_b y, \wh\G_b z)_{\wh
H}
\end{array}
\end{equation}
Moreover, for each such a mapping $\G_b$ one has
\begin {equation}\label{3.32.6}
\dim\cH_b=\nu_{b-}, \qquad \dim \wt\cH_b=\nu_{b+}- \wh\nu
\end{equation}
and the following equivalence holds
\begin {equation}\label{3.32.7}
n_+(\Tmi)= n_-(\Tmi)\iff\wt \cH_b=\cH_b.
\end{equation}
Therefore in the case of equal deficiency indices
$n_+(\Tmi)=n_-(\Tmi)$  the identity \eqref{3.32.5} takes the form
\begin {equation*}
[y,z]_b=(\G_{0b}y,\G_{1b}z)_{\cH_b}-(\G_{1b}y,\G_{0b}z)_{\cH_b}+
i (\wh\G_b y, \wh\G_b z)_{\wh H}
\end{equation*}
\end{lemma}
\begin{proof}
In view of \eqref{3.17c} and \eqref{3.12} one has $
\nu_{b+}-\nu_{b-}\geq \wh\nu$. Therefore by \cite [Lemma
5.1]{Mog11} there exist Hilbert spaces $\cH_b$ and $\wh\cH_b$ and
a surjective linear mapping
\begin {equation}\label{3.33.3}
\G_b=(\G_{0b}',\,  \wh\G_b',\,  \G_{1b})^\top:\dom\tma\to
\cH_b\oplus\wh\cH_b\oplus \cH_b
\end{equation}
such that
\begin {equation*}
[y,z]_b=(\G_{0b}'y,\G_{1b}z)_{\cH_b}-(\G_{1b}y,\G_{0b}'z)_{\cH_b}+
i (\wh\G_b y, \wh\G_b z)_{\wh\cH_b}, \quad y,z \in \dom\tma.
\end{equation*}
Moreover, for such a mapping $\G_b$ one has
\begin {equation}\label{3.34}
\dim\cH_b=\nu_{b-}, \qquad \dim \wh\cH_b=\nu_{b+}- \nu_{b-},
\end{equation}
which in view of \eqref{3.16a}   yields $\dim\wh\cH_b\geq \dim \wh
H $. Therefore without loss of generality one may assume that $\wh
H\subset \wh \cH_b$ and hence
\begin {gather}\label{3.34.0}
\wh \cH_b= \cH_2'\oplus\wh H
\end{gather}
with $\cH_2'=\wh \cH_b\ominus\wh H$. Let
$\wt\cH_b=\cH_b\oplus\cH_2'$ (so that $\cH_b\subset\wt\cH_b$) and
let $\G_{0b}:\dom\tma\to \wt\cH_b $ and $\wh\G_{b}:\dom\tma\to \wh
H $ be the linear mappings given by
\begin {gather*}
\G_{0b}=\G_{0b}'+P_{\cH_2'}\wh\G_b', \qquad \wh \G_b=P_{\wh H}\wh
\G_b'.
\end{gather*}
Then \eqref{3.33.3} can be written in the form \eqref{3.32.4} and
the direct calculation gives the identity \eqref{3.32.5}.
Moreover,
\begin {gather*}
\dim \wt\cH_b=\dim \cH_b+(\dim \wh\cH_b-\dim \wh H),
\end{gather*}
which together with \eqref{3.34} and the second equality in
\eqref{3.16a} yields \eqref{3.32.6}. Finally, the equivalence
\eqref{3.32.7} is implied by \eqref{3.17d} and \eqref{3.32.6}.
\end{proof}
\begin{remark}\label{rem3.2a}
(1) Since the mapping $\G_b$ is surjective, it follows from
\eqref{3.32.5} that $\G_b y =0$ for each function $y\in\dom\tma$
such that $y(t)=0$ on some interval $(\b, b)\subset \cI$.
Therefore, if $y_1, y_2\in\dom\tma$ and $y_1(t)=y_2(t)$ on some
interval $(\b, b)$, then $\G_b y_1=\G_b y_2$.

(2) In the case of the regular system \eqref{3.1}  (i.e., when
$\cI=[a,b]$ is a compact interval and both integrals $\int_\cI
||B(t)||\, dt$ and $\int_\cI ||\D(t)||\, dt$ are finite) one can
put in \eqref{3.32.4} $\wt \cH_b=\cH_b=H$ and $\G_b y= X_b y(b),
\; y\in \dom\tma,$ where $X_b\in [\bH]$ and $X_b^*J X_b=J $.

In  general case Remark 5.2 in \cite{Mog11} implies that   the
mapping \eqref{3.32.4} can be constructed  with the aid of the
following assertion:

--- there exist systems  of functions
$\{\t_j^{(1)}\}_1^{\nu_{b+}-\wh \nu},\;\{\t_j^{(2)} \}_1^{\wh\nu}
$ and $\{\t_j^{(3)}\}_1^{\nu_{b-}}$ in $\dom\tma$  such that  the
operators
\begin {equation*}
\G_{0b}y =\{[y, \t_j^{(1)}]_b\}_1^{\nu_{b+}-\wh \nu}, \quad
\wh\G_b y=\Sigma_1^{\wh\nu}\, [y,\t_j^{(2)} ]_b\, e_j, \quad
\G_{1b}y =\{[y,\t_j^{(3)}]_b\}_1^{\nu_{b-}}
\end{equation*}
($y\in\dom\tma$) form the surjective linear mapping
$\G_b=(\G_{0b},\, \hat\G_b,\, \G_{1b})^\top:\dom\tma\to
\bC^{\nu_{b+}-\wh\nu}\oplus \wh H\oplus \bC^{\nu_{b-}}$ satisfying
the identity \eqref{3.32.5} (here $\{e_j\}_1^{\wh\nu}$ is an
orthonormal basis in $\wh H$).

In the case $\nu_{b-}=0, \; \nu_{b+}=\wh \nu$ and $\dim \wh H=1$
one has $\wt \cH_b=\cH_b=\{0\}$. In this case one can put
$$
\wh \G_b y=[y,\t]_b \,e,
$$
where $e$ is an ort in $\wh H$ and $\t$ is a function in
$\dom\tma$ such that $[\t,\t]_b=i$.

These assertions  show that one may consider $\G_b y$  as a
singular boundary value of a function $y\in\dom\tma$ (cf.
\cite[Ch. 13.2]{DunSch}).
\end{remark}

 The following
proposition is immediate from \cite[Theorem 5.8]{Mog11} and
\eqref{3.32.7}.
\begin{proposition}\label{pr3.3}
Assume that $n_-(\Tmi)\leq n_+(\Tmi)$,  $\wt U$ is the $J$-unitary
operator \eqref{3.17.5},  $\G_a$ is the linear mapping
\eqref{3.21} with the block representation \eqref{3.22} and $\G_b$
is the surjective linear mapping \eqref{3.32.4} satisfying the
identity \eqref{3.32.5}. Moreover, let $\cH_0$ and  $\cH_1(\subset
\cH_0)$ be finite dimensional Hilbert spaces defined by
\begin {equation}\label{3.34b}
\cH_0=H_0 \oplus\wt\cH_b, \qquad \cH_1=H_0\oplus \cH_b
\end{equation}
 and let $\G_j:\Tma\to\cH_j, \; j\in
\{0,1\},$ be the operators given by
\begin {gather}
\G_0\{\wt y, \wt f\}=\{ - \G_{1a}y +i(\wh\G_a-\wh\G_b)y,\;
\G_{0b}y\}(\in  H_0\oplus\wt\cH_b),\qquad\qquad\qquad\quad\label{3.43.1}\\
\G_1\{\wt y, \wt f\}=\{\G_{0a}y + \tfrac 1 2(\wh\G_a+\wh\G_b)y, \;
-\G_{1b}y\}(\in  H_0\oplus \cH_b), \quad \{\wt y, \wt
f\}\in\Tma.\label{3.43.2}
\end{gather}
(here $y\in\dom\tma$ is the function corresponding to $\{\wt y,
\wt f\}\in\Tma$ according to Remark \ref{rem3.1a}). Then the
collection $\Pi=\bta$ is a boundary triplet for $\Tma$.

If in addition $n_+(\Tmi)=n_-(\Tmi)$, then $\Pi$ turns into an
ordinary boundary triplet $\Pi=\bt$ for $\Tma$, where
$\cH=H_0\oplus\cH_b$  and $\G_j:\Tma\to\cH, \; j\in\{0,1\},$ are
the operators given by \eqref{3.43.1} and \eqref{3.43.2} with
$\wt\cH_b=\cH_b$.
\end{proposition}
\begin{definition}\label{def3.4}
The boundary triplet $\Pi=\bta$  constructed in Proposition
\ref{pr3.3} will be called a decomposing  boundary triplet for
$\Tma$.
\end{definition}
\begin{proposition}\label{pr3.6}
Let $n_-(\Tmi)\leq n_+(\Tmi)$, let $U$ be the operator
\eqref{3.17.1}, let $\wh \G_a$ and $\G_{1a}$ be the linear
mappings \eqref{3.24} and let $\G_b$ be the linear mapping
\eqref{3.32.4}. Then:

{\rm (1)} The equalities
\begin{gather}
T=\{\{\wt y, \wt f\}\in\Tma: \, \G_{1a}y=0, \;\wh\G_a y=\wh\G_b
y,\;
\G_{0b}y =\G_{1b}y=0 \}\label {3.45}\\
T^*=\{\{\wt y, \wt f\}\in\Tma: \, \G_{1a}y=0, \;\wh\G_a y=\wh\G_b
y \}\label {3.45a}
\end{gather}
define a symmetric extension $T$ of $\Tmi$ and its adjoint $T^*$.
Moreover, the deficiency indices of $T$ are $
n_+(T)=\nu_{b+}-\wh\nu$ and $ n_-(T)=\nu_{b-}$.

{\rm (2)} The collection $\dot\Pi=\{\wt\cH_b\oplus \cH_b,
\dot\G_0,\dot\G_1\}$ with the operators
\begin {equation}\label {3.46}
\dot\G_0 \{\wt y,\wt f\}=\G_{0b}y, \qquad \dot\G_1 \{\wt y,\wt
f\}=-\G_{1b}y, \qquad \{\wt y,\wt f\}\in T^*,
\end{equation}
is a boundary triplet for $T^*$ and the (maximal symmetric)
relation $A_0(=\ker\dot\G_0)$ is of the form
\begin {equation}\label {3.47}
A_0=\{\{\wt y, \wt f\}\in\Tma: \, \G_{1a}y=0, \;\wh\G_a y=\wh\G_b
y,\; \G_{0b}y =0 \}.
\end{equation}

If in addition $n_+(\Tmi)=n_-(\Tmi)$, then  $
n_+(T)=n_-(T)=\nu_{b-}$, $\dot\Pi=\{\cH_b, \dot\G_0,\dot\G_1\}$ is
an ordinary boundary triplet for $T^*$ and $A_0=A_0^*$.
\end{proposition}
\begin{proof}
Let $\wt U$ be the $J$-unitary extension \eqref{3.17.5} of $U$,
let $\G_{0a}$ be the operator \eqref{3.23} and let $\Pi=\bta$ be
the decomposing boundary triplet \eqref{3.43.1}, \eqref{3.43.2}
for $\Tma$. Applying to this triplet Proposition \ref{pr2.10a} one
obtains the desired statements.
\end{proof}
\begin{remark}\label{rem3.7}
Clearly, $\mul T=\mul T^*$ if and only if the following condition
is fulfilled:

(C1) For each function $y\in\dom\tma$ the equalities
\begin {equation}\label {3.48}
\G_{1a}y=0, \quad \wh\G_a y=\wh\G_b y\;\;\;\text{ and} \;\;\;
\D(t)y(t)=0 \;\;\text{(a.e. on}\;\; \cI)
\end{equation}
yield $\G_{0b}y=\G_{1b}y=0 $.

 Moreover, $\mul T^*=0 $ (i.e., $T$ is a densely
defined operator) if and only if the following condition is
satisfied:

(C2) For each $y\in\dom\tma$ the equalities \eqref{3.48} yield
$y=0$.

\end{remark}
\section{$\cL_\D^2$-solutions of boundary value problems}\label{sect4}
In what follows we suppose that the symmetric system \eqref{3.1}
satisfies the condition $n_-(\Tmi)\leq n_+(\Tmi)$. Our
considerations will be also based on the following assumptions:
\begin{itemize}
\item[(A1)]
 $U$ is the operator \eqref{3.17.1} satisfying the relations \eqref{3.17.2} -
\eqref{3.17.4} and $\wh\G_a$ and $\G_{1a}$ are the linear mappings
\eqref{3.24}.

\item[(A2)]
 $\wt\cH_b$ and $\cH_b(\subset \wt\cH_b)$ are finite dimensional
Hilbert spaces and $\G_b$ is the surjective linear mapping
\eqref{3.32.4} such that \eqref{3.32.5} holds.
\end{itemize}
In addition to (A1)--(A2) we will sometimes use the following
assumption:
\begin{itemize}
\item[(A3)]
$\wt U$ is a $J$-unitary extension \eqref{3.17.5} of $U$ and
$\G_{0a}$ is the mapping \eqref{3.23}.
\end{itemize}

Let  (A1)--(A2) be satisfied and  let $\pair\in\RP$ be a
collection of holomorphic operator pairs \eqref{2.2} with
$C_0(\l)\in [\wt\cH_b], \; C_1(\l)\in [\cH_b,\wt\cH_b], \;
\l\in\bC_+,$ and $D_0(\l)\in [\wt\cH_b,\cH_b], \; D_1(\l)\in
[\cH_b], \; \l\in\bC_-$. For a given function $f\in\lI$ consider
the following boundary value problem:
\begin{gather}
J y'-B(t)y=\l \D(t)y+\D(t)f(t), \quad t\in\cI,\label{4.0.1}\\
\G_{1a}y=0, \quad \wh \G_a y= \wh\G_b y,\quad \l\in\CR,\label{4.0.2}\\
C_0(\l)\G_{0b}y +C_1(\l)\G_{1b}y=0, \quad \l\in\bC_+ \label{4.0.3.1}\\
 D_0(\l)\G_{0b}y +D_1(\l)\G_{1b}y=0, \quad \l\in\bC_-.\label{4.0.3.2}
\end{gather}
A function $y(\cd,\cd):\cI\tm (\CR)\to\bH$ is called a solution of
this problem if for each $\l\in\CR$ the function $y(\cd,\l)$
belongs to $\AC\cap\lI$ and satisfies the equation \eqref{4.0.1}
a.e. on $\cI$ (so that $y\in\dom\tma$) and the boundary conditions
\eqref{4.0.2} -- \eqref{4.0.3.2}.

If $n_+(\Tmi)=n_-(\Tmi)$, then in view of \eqref{3.32.7}
$\wt\cH_b=\cH_b$ and the collection $\tau$ turns into a Nevanlinna
operator pair $\tau\in\wt R(\cH_b)$ defined by \eqref{2.19} with
$C_j(\l)\in [\cH_b],\;\l\in\CR, \; j\in\{0,1\}$. In this case the
boundary conditions \eqref{4.0.3.1}--\eqref{4.0.3.2} takes the
form
\begin {equation}\label{4.0.3a}
 C_0(\l)\G_{0b}y +C_1(\l)\G_{1b}y=0, \quad \l\in\CR.
\end{equation}
If in addition $\tau\in\wt R^0(\cH_b)$ is a selfadjoint operator
pair \eqref{2.22} with $C_j\in [\cH_b], \;j\in\{0,1\},$ then
\eqref{4.0.3a} becomes a self-adjoint boundary condition (at the
endpoint $b$):
\begin {equation}\label{4.0.3b}
 C_0\G_{0b}y +C_1\G_{1b}y=0.
\end{equation}
\begin{theorem}\label{th4.0.1}
Let  $T$ be a symmetric relation in $\LI$ defined by \eqref{3.45}.
If $\pair\in\RP$ is a collection \eqref{2.2}, then for every
$f\in\lI$ the boundary problem \eqref{4.0.1} - \eqref{4.0.3.2} has
a unique solution $y(t,\l)=y_f(t,\l) $ and the equality
\begin {equation}\label {4.0.4}
R(\l)\wt f = \pi(y_f(\cd,\l)), \quad \wt f\in \LI, \quad f\in\wt
f, \quad \l\in\CR,
\end{equation}
defines a generalized resolvent $R(\l)=:R_\tau(\l)$ of $T$.
Conversely, for each generalized resolvent $R(\l)$ of $T$ there
exists a unique $\tau\in\RP$ such that $R(\l)=R_\tau(\l)$.

If in addition $n_+(\Tmi)=n_-(\Tmi)$, then  $n_+(T)=n_-(T)$ and
the above statements hold with   Nevanlinna operator pairs
$\tau\in\wt R(\cH_b)$ of the form \eqref{2.19} and the boundary
condition \eqref{4.0.3a} in place of \eqref{4.0.3.1} and
\eqref{4.0.3.2}. Moreover, $R_\tau(\l)$ is a canonical resolvent
of $T$ if and only if $\tau\in\wt R^0(\cH_b)$ is a self-adjoint
operator pair \eqref{2.22}, in which case $R_\tau(\l)=(\wt T^\tau
- \l)^{-1}$ with
\begin {equation}\label {4.0.5}
\wt T^\tau=\{\{\wt y,\wt f\}\in\Tma : \G_{1a}y=0, \;\wh \G_a y=
\wh\G_b y,\; C_0\G_{0b}y +C_1\G_{1b}y=0\}.
\end{equation}
\end{theorem}
\begin{proof}
Let $\dot \Pi=\{\wt\cH_b\oplus\cH_b,\dot\G_0, \dot\G_1\}$ be the
boundary triplet \eqref{3.46} for $T^*$. Applying to this triplet
Theorem \ref{th2.12} we obtain the required statements.
\end{proof}
\begin{remark}\label{rem4.0.3}
Let in Theorem \ref{th4.0.1} $\tau_0=\{\tau_{+},\tau_{-} \}\in\wt
R(\wt\cH_b,\cH_b)$ be defined by \eqref{2.2} with
\begin {equation}\label{4.0.6}
C_0(\l) \equiv I_{\wt\cH_b}, \quad C_1(\l)\equiv 0 , \quad D_0(\l)
\equiv P_{\cH_b}(\in [\wt\cH_b,\cH_b]), \quad D_1(\l)\equiv 0
\end{equation}
and let $R_0(\l):=R_{\tau_0}(\l)$ be the corresponding generalized
resolvent of $T$. Then
\begin {equation*}
R_0(\l)=(A_0-\l)^{-1},\;\l\in\bC_+ \;\;\text{and}\;\;
R_0(\l)=(A_0^*-\l)^{-1},\;\l\in\bC_-,
\end{equation*}
where $A_0$ is given by \eqref{3.47}.

If in addition $n_+(\Tmi)=n_-(\Tmi)$, then $\tau_0$ turns into a
self-adjoint operator pair $\tau_0=\{(I_{\cH_b},0);\cH_b\}\in \wt
R^0(\cH_b)$ and $R_0(\l)=(A_0-\l)^{-1}$   is a canonical resolvent
of $T$.
\end{remark}
\begin{proposition}\label{pr4.1}
Let the assumptions {\rm (A1)} and {\rm (A2)} be satisfied. Then:

{\rm (1)} For every $\l\in\CR$ there exists a unique operator
solution $v_0(\cd,\l)\in\lo{H_0}$ of Eq. \eqref{3.2} such that
\begin{gather}
\G_{1a} v_0(\l)=-P_H, \qquad i(\wh\G_a-\wh\G_b)v_0(\l)=P_{\wh  H},
\;\;\;\; \l\in\CR \label{4.3}\\
\G_{0b} v_0(\l)=0, \;\;\;\;\l\in\bC_+;\qquad P_{\cH_b}\G_{0b}
v_0(\l)=0, \;\;\;\;\l\in\bC_- \label{4.4}
\end{gather}

{\rm(2)} For every $\l\in\bC_+ \; (\l\in\bC_-)$ there exists a
unique operator solution $u_+(\cd,\l)\in\lo{\wt\cH_b}$ (resp.
$u_-(\cd,\l)\in\lo{\cH_b}$) of Eq. \eqref{3.2} such that
\begin{gather}
\G_{1a}u_\pm(\l)=0, \qquad i(\wh\G_a-\wh\G_b)u_\pm(\l)=0,\quad
\l\in\bC_\pm,\label{4.5}\\
\G_{0b}u_+(\l)=I_{\wt\cH_b},\;\; \l\in\bC_+; \qquad
P_{\cH_b}\G_{0b}u_-(\l)=I_{\cH_b},\;\; \l\in\bC_-.\label{4.6}
\end{gather}
In formulas \eqref{4.3}-- \eqref{4.6} $v_0(\l)$ and $u_\pm(\l)$
are linear mappings from Lemma \ref{lem3.2}, (1) corresponding to
the solutions $v_0(\cd,\l)$ and $u_{\pm}(\cd,\l)$ respectively.
\end{proposition}
\begin{proof}
Let $\wt U$ be the $J$-unitary extension \eqref{3.17.5} of $U$,
let $\G_{0a}$ be the operator \eqref{3.23} and let $\Pi=\bta$ be
the decomposing boundary triplet \eqref{3.43.1}, \eqref{3.43.2}
for  $\Tma$. Assume also that $\g_\pm(\cd)$ are the $\g$-fields
of  $\Pi$. Since the quotient mapping $\pi$ isomorphically maps
$\cN_\l$ onto $\gN_\l (\Tmi)$, it follows that for every
$\l\in\bC_+\;(\l\in\bC_-)$ there exists an isomorphism
$Z_+(\l):\cH_0\to \cN_\l$ (resp. $Z_-(\l):\cH_1\to \cN_\l$) such
that
\begin {equation}\label{4.11}
\g_+(\l)=\pi Z_+(\l), \;\;\;\l\in\bC_+; \qquad \g_-(\l)=\pi
Z_-(\l), \;\;\;\l\in\bC_-.
\end{equation}
Let $\G_0'$ and $\G_1'$ be the linear mappings given by
\begin {equation}\label{4.11.1}
\begin{array}{l}
\G_0'=\begin{pmatrix} - \G_{1a} +i(\wh\G_a-\wh\G_b)\cr  \G_{0b}
\end{pmatrix}:\dom\tma \to H_0\oplus\wt\cH_b,\\
\G_1'=\begin{pmatrix} \G_{0a} + \tfrac 1 2(\wh\G_a+\wh\G_b)\cr
-\G_{1b}
\end{pmatrix}:\dom\tma \to H_0\oplus\cH_b.
\end{array}
\end{equation}
Then by \eqref{3.43.1} and \eqref{3.43.2} one has $ \G_j\{\pi y,
\l \pi y\}=\G_j' y, \; y\in\cN_\l, \; j \in \{0,1\}. $ Combining
of this equality with \eqref{4.11} and \eqref{2.33} gives
\begin{gather}
\G_0' Z_+(\l)=I_{\cH_0},\quad \l\in\bC_+;\qquad
P_{H_0\oplus\cH_b}\G_0' Z_-(\l)= I_{\cH_1}, \quad
\l\in\bC_-,\label{4.12}
\end{gather}
which in view of \eqref{4.11.1} can be written as
\begin{gather}
\begin{pmatrix}-\G_{1a}+i (\wh\G_a- \wh\G_b) \cr \G_{0b}
\end{pmatrix} Z_+(\l)=\begin{pmatrix} I_{H_0} & 0 \cr 0 & I_{\wt\cH_b}
\end{pmatrix}, \;\; \l\in\bC_+ \label{4.14}\\
\begin{pmatrix}-\G_{1a}+i (\wh\G_a- \wh\G_b) \cr P_{\cH_b}\G_{0b}
\end{pmatrix} Z_-(\l)=\begin{pmatrix} I_{H_0} & 0 \cr 0 & I_{\cH_b}
\end{pmatrix}, \;\; \l\in\bC_- \label{4.16}
\end{gather}
It follows from \eqref{4.14} and \eqref{4.16} that
\begin{gather}
\G_{1a}Z_\pm(\l)=(-P_H,\; 0), \quad \tfrac 1 2 (\wh\G_a- \wh\G_b)
Z_\pm(\l)=(-\tfrac i 2 P_{\wh H},\; 0), \;\;\;\l\in\bC_\pm \label{4.18}\\
\G_{0b} Z_+(\l)= (0,\, I_{\wt \cH_b}), \;\; \l\in\bC_+; \quad
P_{\cH_b}\G_{0b} Z_-(\l)= (0,\, I_{ \cH_b}), \;\;
\l\in\bC_-.\label{4.22}
\end{gather}

Assume now that the block representations of $Z_\pm(\l)$ are
\begin {gather}
Z_+(\l)=(v_0(\l),\,u_+(\l)):H_0\oplus\wt\cH_b\to
\cN_\l,\;\;\l\in\bC_+\label{4.26}\\
 Z_-(\l)=(v_0(\l),\, u_-(\l)):H_0\oplus \cH_b\to
\cN_\l,\;\;\l\in\bC_-\label{4.26a}
\end{gather}
and let $v_0(\cd,\l)\in \lo{H_0}, \; u_+(\cd,\l)\in \lo{\wt\cH_b}
$ and $u_-(\cd,\l)\in \lo{\cH_b} $ be the operator solutions of
Eq. \eqref{3.2} corresponding to $v_0(\l), \; u_+(\l)$ and
$u_-(\l)$ respectively (see Lemma \ref{lem3.2}). Then  the
representations \eqref{4.26} and \eqref{4.26a} together with
\eqref{4.18} and \eqref{4.22} yield the relations
\eqref{4.3}-\eqref{4.6} for $v_0(\cd,\l)$ and $u_\pm (\cd,\l)$.

To prove uniqueness of $v_0(\cd,\l)$ and $u_\pm(\cd,\l)$ assume
that $\wt v_0(\cd,\l)\in\lo{H_0}, \wt u_+(\cd,\l)\in\lo{\wt \cH_b}
$ and $ \wt u_-(\cd,\l)\in\lo{\cH_b}$ are other solutions of Eq.
\eqref{3.2} satisfying \eqref{4.3}--\eqref{4.6}. Then for each
$h_0\in H_0, \; \wt h_b\in\wt\cH_b$ and $h_b\in \cH_b$ the
functions $y_1=(v_0(t,\l)- \wt v_0(t,\l))h_0, \; y_2=(u_+(t,\l)-
\wt u_+(t,\l)) \wt h_b$ and $y_3=(u_-(t,\l)- \wt u_-(t,\l)) h_b$
are solutions of the homogeneous boundary problem \eqref{4.0.1} --
\eqref{4.0.3.2} with $f=0$ and $C_j(\l), \; D_j(\l),\;
j\in\{0,1\},$ defined by \eqref{4.0.6}. Since by Theorem
\ref{th4.0.1} such a problem has a unique solution $y=0$, it
follows that $y_1=y_2=y_3=0$ and, consequently, $v_0=\wt v_0$ and
$u_\pm=\wt u_\pm$.
\end{proof}
\begin{proposition}\label{pr4.1a}
Assume the hypothesis  {\rm (A1)--(A3)}. Let  $\Pi=\bta$ be the
decomposing
boundary triplet for $\Tma$ defined in Proposition \ref{pr3.3}, 
let $\g_\pm(\cd)$ and  $M_\pm(\cd)$ be the corresponding
$\g$-field and the Weyl function, respectively. Then $\g_\pm(\cd)$
is connected with the solutions $v_0(\cd,\l)$ and $u_\pm(\cd,\l)$
from Proposition \ref{pr4.1} by
\begin {gather}
\g_\pm(\l)\up H_0=\pi v_0(\l),\;\;\;\;\l\in\bC_\pm;\label{4.26.1}\\
\quad \g_+(\l)\up \wt\cH_b=\pi u_+(\l), \;\;\;\l\in\bC_+;\quad
\g_-(\l)\up \cH_b=\pi u_-(\l), \;\;\;\l\in\bC_-.\label{4.26.2}
\end{gather}
and the block representations
\begin{gather}
M_+(\l)=\begin{pmatrix} m_0(\l )& M_{2+}(\l) \cr M_{3+}(\l) &
M_{4+}(\l)
\end{pmatrix}: H_0\oplus\wt\cH_b\to H_0\oplus\cH_b,
\;\;\;\l\in\bC_+\label{4.26.3}\\
M_-(\l)=\begin{pmatrix} m_0(\l )& M_{2-}(\l) \cr M_{3-}(\l) &
M_{4-}(\l)
\end{pmatrix}: H_0\oplus\cH_b\to H_0\oplus \wt\cH_b,
\;\;\;\l\in\bC_-\label{4.26.4}
\end{gather}
hold with
\begin{gather}
m_0(\l)=(\G_{0a}+\wh\G_a)v_0(\l)+\tfrac i 2 P_{\wh H}, \quad
\l\in\CR \label{4.26.5}\\
M_{2\pm}(\l)=(\G_{0a}+\wh\G_a)u_\pm(\l),\quad \l\in\bC_\pm \label{4.26.6}\\
M_{3+}(\l)=-\G_{1b}v_0(\l), \qquad M_{4+}(\l)=-\G_{1b}u_+(\l),
\quad \l\in\bC_+\label{4.26.7}\\
M_{3-}(\l)=(-\G_{1b}+i P_{\cH_b^\perp}\G_{0b})v_0(\l), \label{4.26.8a}\\
M_{4-}(\l)=(-\G_{1b}+i P_{\cH_b^\perp}\G_{0b})u_-(\l),\;\;
\l\in\bC_-.\label{4.26.8b}
\end{gather}
\end{proposition}
\begin{proof}
The equalities \eqref{4.26.1} and \eqref{4.26.2} are immediate
from \eqref{4.11} and \eqref{4.26}, \eqref{4.26a}.

Next assume that $\G_0'$ and $\G_1'$ are the linear mappings
\eqref{4.11.1} and let $M_\pm(\cd)$ have the block representations
\eqref{4.26.3}, \eqref{4.26.4}. Then by using \eqref{4.11} and
\eqref{2.32a}, \eqref{2.32b} one obtains
\begin {equation*}
\G_1' Z_+(\l)=M_+(\l), \;\;\l\in\bC_+; \qquad   (\G_1'+i
P_{\cH_b^\perp}\G_0') Z_-(\l)=M_-(\l), \;\; \l\in\bC_-,
\end{equation*}
which  can be represented as
\begin{gather}
\begin{pmatrix}\G_{0a}+\tfrac 1 2  (\wh\G_a+ \wh\G_b) \cr -\G_{1b}
\end{pmatrix} Z_+(\l)=\begin{pmatrix} m_0(\l )& M_{2+}(\l) \cr M_{3+}(\l)
 & M_{4+}(\l)\end{pmatrix}, \;\; \l\in\bC_+ \label{4.26.9}\\
\begin{pmatrix}\G_{0a}+\tfrac 1 2  (\wh\G_a+ \wh\G_b) \cr
-\G_{1b}+iP_{\cH_b^\perp}\G_{0b}
\end{pmatrix} Z_-(\l)=\begin{pmatrix} m_0(\l )& M_{2-}(\l) \cr M_{3-}(\l) & M_{4-}(\l)
\end{pmatrix}, \;\; \l\in\bC_-. \label{4.26.10}
\end{gather}
Hence
\begin{gather}
\G_{0a}Z_\pm(\l)=(P_H m_0(\l),\, P_H M_{2\pm}(\l)),
\;\;\;\l\in\bC_\pm
\label{4.26.11} \\
\tfrac 1 2 (\wh\G_a+ \wh\G_b) Z_\pm(\l)=(P_{\wh H} m_0(\l),\,
P_{\wh H}
M_{2\pm}(\l)),\; \;\;\l\in\bC_\pm. \label{4.26.12}\\
\G_{1b} Z_+(\l)=(-M_{3+}(\l),\, -M_{4+}(\l)), \;\; \l\in\bC_+,\label{4.26.13}\\
(-\G_{1b}+i P_{\cH_b^\perp}\G_{0b})Z_-(\l)=(M_{3-}(\l), \,
M_{4-}(\l)), \;\; \l\in\bC_-. \label{4.26.14}
\end{gather}
Summing up the second equality in \eqref{4.18} with
\eqref{4.26.11} and \eqref{4.26.12} one obtains
\begin {equation}\label{4.26.15}
(\G_{0a}+\hat\G_a)Z_{\pm}(\l)=(m_0(\l)-\tfrac i 2 P_{\wh
H},\,M_{2\pm}(\l)), \;\;\;\l\in\bC_\pm.
\end{equation}
Combining now \eqref{4.26.13}--\eqref{4.26.15} with the block
representations \eqref{4.26} and \eqref{4.26a} of $Z_\pm(\l)$ we
arrive at the equalities \eqref{4.26.5}--\eqref{4.26.8b}.
\end{proof}
In the case of equal deficiency indices the statements of
Propositions \ref{pr4.1} and \ref{pr4.1a} can be rather
simplified.  Namely the following corollary is obvious.
\begin{corollary}\label{cor4.1b}
Let the assumptions {\rm (A1)} and {\rm (A2)} be satisfied,
$n_+(\Tmi)=n_-(\Tmi)$, and let $A_0$ be the selfadjoint extension
of $\Tmi$ given by \eqref{3.47}. Then for every $\l\in\rho (A_0)$
there exists a unique pair of operator-valued solutions
$v_0(\cd,\l)\in\lo{H_0}$ and $u(\cd,\l)\in\lo{\cH_b}$ of Eq.
\eqref{3.2} satisfying the following boundary conditions:
\begin{gather*}
\G_{1a} v_0(\l)=-P_H, \quad i(\wh\G_a-\wh\G_b)v_0(\l)=P_{\wh
H},\quad \G_{0b}
v_0(\l)=0, \;\; \l\in\rho (A_0),  \\
\G_{1a}u(\l)=0, \quad i(\wh\G_a-\wh\G_b)u(\l)=0, \quad
\G_{0b}u(\l)=I_{\cH_b},\;\;\l\in\rho (A_0).
\end{gather*}
Assume,  in addition, that the assumption {\rm (A3)} is fulfilled
and $\Pi=\bt$ is an ordinary decomposing boundary triplet
\eqref{3.43.1}, \eqref{3.43.2} for $\Tma$.   Then the
corresponding Weyl function $M(\cd)$ admits a  block matrix
representation
\begin {equation}\label{4.26.16}
M(\l)=\begin{pmatrix} m_0(\l )& M_{2}(\l) \cr M_{3}(\l) &
M_{4}(\l)
\end{pmatrix}: H_0\oplus \cH_b\to H_0\oplus\cH_b,
\;\;\;\l\in\rho (A_0)
\end{equation}
with the entries given by
\begin{gather}
m_0(\l)=(\G_{0a}+\wh\G_a)v_0(\l)+\tfrac i 2 P_{\wh H},\qquad
M_{2}(\l)=(\G_{0a}+\wh\G_a)u(\l),\label{4.26.17}\\
M_{3}(\l)=-\G_{1b}v_0(\l), \qquad M_{4}(\l)=-\G_{1b}u(\l), \quad
\l\in\rho (A_0).\label{4.26.18}
\end{gather}
\end{corollary}

\begin{theorem}\label{th4.2}
Let the assumptions {\rm (A1)} and {\rm (A2)}  be fulfilled and
let $\pair\in \RP$ be a collection of operator pairs \eqref{2.2}.
 Then for each $\l\in\CR$ there exists a unique operator solution
$v_\tau(\cd,\l)\in\lo{H_0}$ of Eq. \eqref{3.2} satisfying the
boundary conditions
\begin {gather}
\G_{1a}v_\tau(\l)=-P_H, \quad \l\in\CR,\label {4.27}\\
i(\wh\G_a-\wh\G_b)v_\tau(\l)=P_{\wh H},\quad \l\in\CR,\label {4.27a}\\
C_0(\l)\G_{0b}v_\tau(\l)+C_1(\l)\G_{1b}v_\tau(\l)=0,
\;\;\l\in\bC_+,\label {4.28} \\
D_0(\l)\G_{0b}v_\tau(\l)+D_1(\l)\G_{1b}v_\tau(\l)=0,
\;\;\l\in\bC_-\label {4.28a}
\end{gather}
(here $P_H$ and $P_{\wh H}$ are the orthoprojectors  in  $H_0$
onto $H$ and $\wh H$ respectively and $v_\tau(\l)$ is the linear
map from Lemma \ref{lem3.2} corresponding to the solution
$v_\tau(\cd,\l)$).
 Moreover, $v_\tau(\cd,\l)$ is connected with the solutions $v_0(\cd,\l)$ and
 $u_\pm(\cd,\l)$ from Proposition \ref{pr4.1} via the equalities
\begin {gather}
v_\tau(t,\l)=v_0(t,\l)-u_+(t,\l)(\tau_+(\l)+M_{4+}(\l))^{-1}M_{3+}(\l),
\quad
\l\in\bC_+ \label{4.29}\\
v_\tau(t,\l)=v_0(t,\l)-u_-(t,\l)(\tau_+^*(\ov\l)+M_{4-}(\l))^{-1}M_{3-}(\l),
\quad \l\in\bC_- \label{4.30},
\end{gather}
in which $M_{3\pm}(\cd)$ and $M_{4\pm}(\cd)$ are the operator
functions given by \eqref{4.26.7}--\eqref{4.26.8b}.

If in addition $n_+(\Tmi)=n_-(\Tmi)$,  then $\tau\in \wt R(\cH_b)$
is given by \eqref{2.19} and the boundary conditions \eqref{4.28}
and \eqref{4.28a} take the form
\begin {equation*}
C_0(\l)\G_{0b}v_\tau(\l)+C_1(\l)\G_{1b}v_\tau(\l)=0, \;\;\l\in\CR.
\end{equation*}
\end{theorem}
\begin{proof}
To prove the theorem it is sufficient to show that the equalities
\eqref{4.29} and  \eqref{4.30} define a unique solution
$v_\tau(\cd,\l)\in\lo{H_0}$ of Eq. \eqref{3.2} satisfying
\eqref{4.27}--\eqref{4.28a}.

Let $\dot\Pi=\{\wt\cH_b\oplus\cH_b,\dot\G_0,\dot\G_1\}$ be a
boundary triplet \eqref{3.46} for $T^*$. Then by Proposition
\ref{pr2.10a}, (3) the corresponding Weyl function is $\dot
M_+(\l)=M_{4+}(\l)$ and according to \cite{Mog06.2} one has
$0\in\rho (\tau_+(\l)+M_{4+}(\l)), \;\l\in\bC_+,$ and $0\in\rho
(\tau_+^*(\ov\l)+M_{4-}(\l)), \;\l\in\bC_-$. Therefore for each
$\l\in\CR$ the equalities \eqref{4.29} and \eqref{4.30} correctly
define the solution $v_\tau(\cd,\l)\in\lo{H_0}$ of Eq.
\eqref{3.2}.

 Combining \eqref{4.29} and \eqref{4.30}
with  \eqref{4.3} and \eqref{4.5} one gets the equalities
\eqref{4.27} and \eqref{4.27a}. To prove \eqref{4.28} and
\eqref{4.28a} we let
$T_+(\l)=(\tau_+(\l)+M_{4+}(\l))^{-1},\;\l\in\bC_+,$ and
$T_-(\l)=(\tau_+^*(\ov\l)+M_{4-}(\l))^{-1},\;\;\l\in\bC_-$. Then
\begin {equation}\label{4.31.1}
\tau_+(\l)=\{\{T_+(\l)h, (I-M_{4+}(\l)T_+(\l))h\}:h\in\cH_b\}
\end{equation}
and $\tau_+^*(\ov\l)=\{\{T_-(\l)h,
h-M_{4-}(\l)T_-(\l))h\}:h\in\wt\cH_b\}$, which in view of
\eqref{2.8b} yields
\begin {gather}
\tau_-(\l)=\{\{(-T_-(\l)-iP_{\cH_b^\perp}+iP_{\cH_b^\perp}M_{4-}(\l)T_-(\l))h
,\qquad \qquad\qquad\label{4.31.2}\\
\qquad\qquad \qquad\qquad\qquad(-P_{\cH_b} +
P_{\cH_b}M_{4-}(\l)T_-(\l))h\}:h\in\wt\cH_b\}.\nonumber
\end{gather}

It follows from \eqref{4.4}, \eqref{4.6} and \eqref{4.26.8a},
\eqref{4.26.8b} that
\begin {gather}
\G_{0b}v_0(\l)=-i P_{\cH_b^\perp}M_{3-}(\l), \qquad
\G_{1b}v_0(\l)=
-P_{\cH_b}M_{3-}(\l), \;\;\l\in\bC_-\label{4.31.3}\\
\G_{0b}u_-(\l)=I_{\cH_b}-i P_{\cH_b^\perp}M_{4-}(\l), \qquad
\G_{1b}u_-(\l)= -P_{\cH_b}M_{4-}(\l), \;\;\l\in\bC_-\label{4.31.4}
\end{gather}
and  the relations \eqref{4.29} and \eqref{4.30} with taking
\eqref{4.4},\eqref{4.6}, \eqref{4.26.7} and \eqref{4.31.3},
\eqref{4.31.4} into account give
\begin {gather*}
\G_{0b}v_\tau(\l)=-T_+(\l)M_{3+}(\l), \quad \l\in\bC_+,\\
\G_{1b}v_\tau(\l)=-(I-M_{4+}(\l)T_+(\l))M_{3+}(\l), \quad \l\in\bC_+,\\
\G_{0b}v_\tau(\l)=(-iP_{\cH_b^\perp}-T_-(\l)+iP_{\cH_b^\perp}M_{4-}(\l)T_-(\l))
M_{3-}(\l),\;\;\;\l\in\bC_-,\\
\G_{1b}v_\tau(\l)=(-P_{\cH_b} +
P_{\cH_b}M_{4-}(\l)T_-(\l))M_{3-}(\l), \;\;\;\l\in\bC_-
\end{gather*}
Hence by \eqref{4.31.1} and \eqref{4.31.2} one has
\begin {equation}\label{4.31.5}
\{\G_{0b} v_\tau(\l)h_0, \G_{1b}v_\tau(\l)h_0\}\in\tau_\pm(\l),
\;\;\; h_0\in H_0, \;\;\l\in\bC_\pm,
\end{equation}
 which in view of the equalities \eqref{2.2} yields \eqref{4.28} and
\eqref{4.28a}.

Finally, one proves uniqueness  of $v_\tau(\cd,\l)$ in the same
way as uniqueness of $v_0(\cd,\l)$ in Proposition \ref{pr4.1}.
\end{proof}
\section{$m$-functions}\label{sect5}
Let the assumptions (A1) and (A2) at the beginning of Section
\ref{sect4} be fulfilled.
\begin{definition}\label{def5.1}
A boundary parameter $\tau $ (at the endpoint $b$) is a collection
$\pair$  of holomorphic  operator pairs \eqref{2.2} belonging to
the class $\wt R(\wt\cH_b,\cH_b)$.

In  the case of equal deficiency indices $n_+(\Tmi)=n_-(\Tmi)$ one
has $\wt\cH_b=\cH_b$ and, therefore, a boundary parameter is an
operator pair $\tau\in\wt R(\cH_b)$ of the form \eqref{2.19}.
\end{definition}
Let in addition to (A1) and (A2) the  assumption (A3) be
satisfied, let $\tau$ be a boundary parameter and let $v_\tau
(\cd,\l)\in\lo{H_0}$ be the corresponding operator solution of Eq.
\eqref{3.2} defined in Theorem \ref{th4.2}.
\begin{definition}\label{def5.2}
The operator function $m_\tau(\cd):\CR\to [H_0]$ defined by
\begin {equation}\label{5.0}
m_\tau(\l)=(\G_{0a}+\wh \G_a)v_\tau (\l)+\tfrac i 2 P_{\wh H},
\quad\l\in\CR,
\end{equation}
will be called the $m$-function corresponding to the boundary
parameter $\tau$ or, equivalently, to the boundary value problem
\eqref{4.0.1}--\eqref{4.0.3.2}.

If $n_+(\Tmi)=n_-(\Tmi)$, then  $m_\tau(\cd)$ corresponds to the
boundary value problem \eqref{4.0.1}, \eqref{4.0.2} and
\eqref{4.0.3a}. In this case the $m$-function $m_\tau(\cd)$ will
be called canonical if $\tau\in \wt R^0(\cH_b)$.
\end{definition}
It follows from  \eqref{4.27} that $m_\tau(\cd)$ satisfies the
equality
\begin {equation}\label{5.1}
v_{\tau,a}(\l)\left(=\begin{pmatrix} \G_{0a}+\wh \G_a\cr \G_{1a}
\end{pmatrix} v_\tau(\l)\right )=\begin{pmatrix} m_\tau(\l)-\tfrac
i 2 P_{\wh H}\cr -P_H
\end{pmatrix}:H_0\to H_0\oplus H.
\end{equation}

It turns out that for a given operator $U$ and a boundary
parameter $\tau$ the $m$-function $m_\tau(\cd)$ is defined up to a
self-adjoint constant. More precisely, the following proposition
holds.
\begin{proposition}\label{pr5.2a}
Suppose  that under the assumptions {\rm (A1)} and {\rm (A2)}
\begin {equation*}
\wt U_j=\begin{pmatrix} u_7^{(j)} & u_8^{(j)}  & u_9^{(j)} \cr
\hline  u_1 & u_2 & u_3 \cr u_4 & u_5 & u_6 \end{pmatrix}:
H\oplus\wh H\oplus H\to H\oplus \wh H\oplus H, \quad j\in\{1,2\}
\end{equation*}
are two $J$-unitary extensions of $U$ and $\G_{0a}^{(j)}:\AC\to H,
\; j\in\{1,2\},$ are the mappings \eqref{3.23}. Moreover, let
$\tau$ be a boundary parameter and let
\begin {equation*}
m_\tau^{(j)}(\l)=(\G_{0a}^{(j)}+\wh \G_a)v_\tau (\l)+\tfrac i 2
P_{\wh H},\quad \l\in\CR,\quad j\in\{1,2\}
\end{equation*}
be the corresponding $m$-functions. Then there exists an operator
$B=B^*\in [H]$ such that the equality
\begin {equation}\label{5.1a}
m_\tau^{(2)}(\l)=m_\tau^{(1)}(\l)+\wt B, \quad \l\in\CR,
\end{equation}
holds with the operator $\wt B=\wt B^*\in [H_0]$ given by $\wt B=B
P_H$.
 \end{proposition}
\begin{proof}
By using the equality $\wt U_j^*J\wt U_j=J, \;j\in\{1,2\},  $ one
can easily prove that there exists $B=B^*\in [H]$ such that $\wt
U_2=X \wt U_1$ with
\begin {equation*}
X=\begin{pmatrix} I & 0 & -B \cr 0 & I & 0\cr 0 &  0 &
I\end{pmatrix}: H\oplus\wh H\oplus H\to H\oplus \wh H\oplus H.
\end{equation*}
Therefore the mappings $\G_a^{(j)}y=\wt U_j y(a), y\in \AC, \;
j\in\{1,2\},$ are connected by $\G_a^{(2)}=X \G_a^{(1)}$, which in
view of \eqref{3.22} gives
\begin {equation*}
\G_{0a}^{(2)}=\G_{0a}^{(1)}-B \G_{1a}.
\end{equation*}
Now by using \eqref{4.27} one obtains
\begin {gather*}
m_\tau^{(2)}(\l)= (\G_{0a}^{(2)}+\wh\G_a)v_\tau(\l)+ \tfrac i 2
P_{\wh
H}=\\
(\G_{0a}^{(1)}+\wh\G_a )v_\tau(\l)+\tfrac i 2 P_{\wh H}-B
\G_{1a}v_\tau(\l)=m_\tau^{(1)}(\l)+ BP_H,
\end{gather*}
which proves \eqref{5.1a}.
\end{proof}
In the following proposition we show that the $m$-function
$m_\tau(\cd)$ can be defined in a somewhat different way.
\begin{proposition}\label{pr5.3}
Let under the assumptions {\rm (A1)--(A3)} $\tau$ be a boundary
parameter at the endpoint $b$,  let $\f_U(\cd,\l)(\in [H_0,\bH])$
be the operator solution defined by \eqref{3.32.1} and let
$\psi(\cd,\l)(\in [H_0,\bH]), \;\l\in\bC,$ be the operator
solutions of Eq. \eqref{3.2} with the initial data
\begin {equation}\label{5.2}
\psi_a(\l)(=\wt U \psi (a,\l))=\begin{pmatrix} -\tfrac i 2 P_{\wh
H}\cr -P_H\end{pmatrix}:H_0\to H_0\oplus H.
\end{equation}
Then there exists a unique operator function $m(\cd):\CR\to [H_0]$
such that for every $\l\in\CR$ the operator solution $v(\cd,\l)$
of Eq. \eqref{3.2} given by
\begin {equation}\label{5.3}
v(t,\l):=\f_U (t,\l)m(\l)+\psi (t,\l)
\end{equation}
belongs to $\lo{H_0}$ and satisfies the  boundary conditions
\eqref{4.27a}--\eqref{4.28a}.  Moreover, the equalities $v(t,\l)=
v_\tau(t,\l)$ and $m(\l)=m_\tau(\l)$ are valid.
\end{proposition}
\begin{proof}
Let $m_\tau(\cd)$ be the  $m$-function in the sense of Definition
\ref{def5.2} and let $v(\cd,\l), \;\l\in\CR,$ be the solution of
Eq. \eqref{3.2} given by \eqref{5.3} with $m(\l)=m_\tau(\l)$. Then
in view of \eqref{3.32.2},\eqref{5.2} and \eqref{5.1} one has
$v_a(\l)=v_{\tau,a}(\l)$ and, consequently, $v(t,\l)=
v_\tau(t,\l)$. Therefore by Theorem \ref{th4.2} $v(\cd,\l)$
belongs to $\lo{H_0}$ and satisfies the   boundary conditions
\eqref{4.27a}--\eqref{4.28a} . Hence there exists an operator
function $m(\l)(=m_\tau(\l))$ with the desired properties.

Assume now that the solution $v(\cd,\l)$ of Eq. \eqref{3.2} given
by \eqref{5.3} with some $m(\l)$ belongs to $\lo{H_0}$ and
satisfies the boundary conditions \eqref{4.27a}--\eqref{4.28a}.
Then in view of \eqref{3.32.2}  and \eqref{5.2}
$\G_{1a}v(\l)=-P_H$ and according to Theorem \ref{th4.2} $v(t,\l)=
v_\tau(t,\l)$. Therefore $m(\l)=m_\tau(\l)$, which proves
uniqueness of $m(\l)$.
\end{proof}
Description of all $m$-functions immediately in terms of the
boundary parameter $\tau$ is contained in the following theorem.
\begin{theorem}\label{th5.4}
Let the assumptions  {\rm (A1)--(A3)} be satisfied and let
$M_\pm(\cd)$ be the operator functions given by \eqref{4.26.3} -
\eqref{4.26.8b} (that is, $M_\pm(\cd)$ are the Weyl functions of
the decomposing boundary triplet $\Pi=\bta$). Moreover, let
$\tau_0=\{\tau_+,\tau_- \}$ be a boundary parameter defined by
\eqref{2.2} and the equality \eqref{4.0.6}. Then
$m_0(\l)=m_{\tau_0}(\l)$ and for every boundary parameter $\pair$
defined by \eqref{2.2} the corresponding $m$-function
$m_\tau(\cd)$ is of the form
\begin {equation}\label{5.5}
m_\tau(\l)=m_0(\l)+M_{2+}(\l)(C_0(\l)-C_1(\l)M_{4+}(\l))^{-1}C_1(\l)M_{3+}(\l),
\quad\l\in\bC_+
\end{equation}
\end{theorem}
\begin{proof}
One can easily verify that  $v_0(t,\l)=v_ {\tau_0}(t,\l)$, where
$v_0(\cd,\l)\in\lo{H_0}$ is the solution of Eq. \eqref{3.2}
defined in Proposition \ref{pr4.1}. This and the equality
\eqref{4.26.5} imply that $m_0(\l)=m_{\tau_0}(\l)$. Next, applying
the operator $\G_{0a}+\wh\G_a$ to the equalities \eqref{4.29} and
\eqref{4.30} with taking \eqref{4.26.5} and \eqref{4.26.6} into
account one obtains
\begin {gather}
m_\tau(\l)=m_0(\l)-M_{2+}(\l)(\tau_+(\l)+M_{4+}(\l))^{-1}M_{3+}(\l),
\quad\l\in\bC_+, \label{5.6}\\
m_\tau(\l)=m_0(\l)-M_{2-}(\l)(\tau_+^*(\ov\l)+M_{4-}(\l))^{-1}M_{3-}(\l),
\quad\l\in\bC_-. \label{5.7}
\end{gather}
Moreover, according to \cite[Lemma 2.1]{MalMog02}
$0\in\rho(C_0(\l)-C_1(\l)M_{4+}(\l)), \; \l\in\bC_+,$ and
\begin {gather*}
-(\tau_+(\l)+M_{4+}(\l))^{-1}=(C_0(\l)-C_1(\l)M_{4+}(\l))^{-1}C_1(\l),
\quad\l\in\bC_+ ,
\end{gather*}
which together with \eqref{5.6} yields \eqref{5.5}.
\end{proof}
The following corollary is immediate from Theorem \ref{th5.4}.
\begin{corollary}\label{cor5.4a}
Let under the assumptions {\rm (A1)--(A3)} $n_+(\Tmi)=n_-(\Tmi)$
and let $M(\cd)$ be the operator function given by
\eqref{4.26.16}--\eqref{4.26.18} (so that $M(\cd)$ is the Weyl
function of the ordinary decomposing boundary triplet $\Pi=\bt$
for $\Tma$). Moreover, let $\tau_0=\{(I_{\cH_b},0);\cH_b\} \in\wt
R^0(\cH_b)$. Then $m_0(\l)=m_{\tau_0}(\l)$ and for every boundary
parameter $\tau$ defined by \eqref{2.19}  the corresponding
$m$-function $m_\tau(\cd)$ is
\begin {equation}\label{5.8a}
m_\tau(\l)=m_0(\l)+M_2(\l)(C_0(\l)-C_1(\l)M_4(\l))^{-1}C_1(\l)M_3(\l),
\quad\l\in\CR.
\end{equation}
\end{corollary}
\begin{proposition}\label{pr5.6}
The $m$-function $m_\tau(\cd)$ is a Nevanlinna operator function
such that the relation
\begin {equation}\label{5.21}
(\im \,\l)^{-1}\cd \im\, m_\tau(\l)\geq \int_\cI
v_\tau^*(t,\l)\D(t) v_\tau(t,\l)\, dt
\end{equation}
holds for all $\l\in\bC_+$. If in addition $n_+(\Tmi)=n_-(\Tmi)$,
then \eqref{5.21} holds for all $\l\in\CR$.
\end{proposition}
\begin{proof}
It follows from \eqref{5.6} and  \eqref{5.7}  that the operator
function $m_\tau(\cd)$ is holomorphic in $\CR$. Next,   the
equality $M_+^*(\ov\l)=M_-(\l)$ for the Weyl functions
\eqref{4.26.3} and \eqref{4.26.4} implies that
$m_0^*(\ov\l)=m_0(\l), \;M_{2+}^*(\ov\l)=M_{3-}(\l), \;
M_{3+}^*(\ov\l)=M_{2-}(\l) $ and $M_{4+}^*(\ov\l)=M_{4-}(\l)$.
This and \eqref{5.6}, \eqref{5.7} yield the equality
$m_\tau^*(\ov\l)=m_\tau(\l), \; \l\in\CR$.  Now it remains to show
that $m_\tau(\cd)$ satisfies  \eqref{5.21}.

Let $\pair\in\RP$  be a boundary parameter defined by \eqref{2.2}.
Assume that $\l\in\bC_+, \; h_0\in H_0$ and let
$y:=v_\tau(\l)h_0$, so that $y=y(t)=v_\tau(t,\l)h_0,\; t\in\cI$.
It follows from \eqref{3.21} that
\begin {equation}\label{5.21a}
(J y(a),y(a))=(J\G_a y, \G_a y)=-2i \,\im (\G_{1a}y,\G_{0a}y
)+i\,||\wh \G_a y||^2.
\end{equation}
Applying now the Lagrange's identity \eqref{3.6} to $\{y,\l
y\}\in\tma$ and taking the equalities \eqref{5.21a} and
\eqref{3.32.5} into account one obtains
\begin {gather}
\im\, \l \cd (y,y)_\D=\tfrac 1 2 (||\wh \G_b y||^2-||\wh \G_a
y||^2 )+ \im\,
(\G_{1a}y,\G_{0a}y )-\label{5.22}\\
(\im\, (\G_{1b}y,\G_{0b}y )-\tfrac 1 2 ||P_{\cH_b^\perp} \G_{0b}
y||^2).\nonumber
\end{gather}
It follows from \eqref{4.27a} that $\wh \G_b y=\wh\G_a y+i P_{\wh
H} h_0$ and, therefore,
\begin {equation}\label{5.25}
|| \wh \G_b y||^2 -||\wh\G_a y||^2=||P_{\wh H} h_0||^2 + 2\im
(\wh\G_a y,P_{\wh H} h_0 ).
\end{equation}
According to \eqref{5.1}
\begin {gather}
\G_{0a}y=P_H m_\tau(\l)h_0, \qquad \G_{1a}y=-P_H h_0, \label{5.26}\\
\wh\G_a y=P_{\wh H}m_\tau(\l)h_0 - \tfrac i 2 P_{\wh
H}h_0\label{5.27}
\end{gather}
and substitution of \eqref{5.27} to the right hand part of
\eqref{5.25} yields
\begin {equation}\label{5.28}
||\wh \G_b y||^2-||\wh \G_a y||^2 =2\im\, (P_{\wh H}m_\tau(\l)h_0,
P_{\wh H}h_0).
\end{equation}
Moreover, by \eqref{5.26} one has
\begin {equation}\label{5.29}
\im\, (\G_{1a}y,\G_{0a}y )=\im\, (P_H m_\tau(\l)h_0, P_H h_0).
\end{equation}
Substituting now \eqref{5.28} and \eqref{5.29} to \eqref{5.22}  we
obtain
\begin {equation}\label{5.30}
\im\, \l \cd (y,y)_\D=\im\, (m_\tau(\l)h_0, h_0)- (\im\,(\G_{1b}y,
\G_{0b}y)-\tfrac 1 2  ||P_{\cH_b^\perp} \G_{0b}y||^2).
\end{equation}
It follows from  \eqref{4.28} that $\{\G_{0b}y,
\G_{1b}y\}\in\tau_+(\l)$. Therefore according to \cite[Proposition
4.3]{Mog06.1}
\begin {equation}\label{5.31}
\im\,(\G_{1b}y, \G_{0b}y)-\tfrac 1 2  ||P_{\cH_b^\perp}
\G_{0b}y||^2\geq 0.
\end{equation}
Moreover, in view of \eqref{3.0} one has
\begin {equation}\label{5.32}
(y,y)_\D=((\int_\cI v_\tau^*(t,\l)\D(t) v_\tau(t,\l)\, dt)h_0,
h_0).
\end{equation}
Combining now  \eqref{5.31} and \eqref{5.32} with \eqref{5.30} we
arrive at the relation \eqref{5.21}.
\end{proof}
\begin{corollary}\label{cor5.6a}
For each boundary parameter $\tau$ the following equality holds:
\begin {equation}\label{5.33}
\f_U(x,\l)v_\tau^*(x,\ov\l)-v_\tau(x,\l)\f_U^*(x,\ov\l)=J, \quad
x\in\cI, \quad \l\in\CR.
\end{equation}
 \end{corollary}
\begin{proof}
Let $\wt U$ be a $J$-unitary extension \eqref{3.17.5} of $U$ and
let $Y_0(x,\l)(\in [\bH])$ be the operator solution of Eq.
\eqref{3.2} with $Y_{0,a}(\l)(=\wt UY_0(a,\l))=I_\bH$. Then by the
Lagrange's identity \eqref{3.6} one has
\begin {equation*}
Y_0^*(x,\ov\l)JY_0(x,\l)=Y_0^*(a,\ov\l)JY_0(a,\l)=\wt U^{-1*}J\wt
U^{-1}=J
\end{equation*}
and, consequently,
\begin {equation}\label{5.34}
Y_0(x,\l)JY_0^*(x,\ov\l)=J,\quad x\in\cI, \quad \l\in\CR.
\end{equation}
Since by Proposition \ref{pr5.6} $m_\tau^*(\ov\l)=m_\tau(\l)$, it
follows from \eqref{5.1} that
\begin {equation*}
v_{\tau,a}^*(\ov\l)=(m_\tau(\l)+\tfrac i 2 P_{\wh H}, \;
-I_H):H_0\oplus H\to H_0.
\end{equation*}
Combining of this equality with \eqref{3.32.2} yields
\begin{gather*}
\f_{U,a}(\l)v_{\tau,a}^*(\ov\l)-v_{\tau,a}(\l)\f_{U,a}^*(\ov\l)=
\begin{pmatrix}I_{H_0}\cr 0 \end{pmatrix}(m_\tau(\l)+\tfrac i 2 P_{\wh H},
 \; -I_H)-\\
\begin{pmatrix} m_\tau(\l)-\tfrac i 2 P_{\wh H} \cr -P_H\end{pmatrix} (I_{H_0},
\; 0)=\begin{pmatrix} i P_{\wh H} & -I_H \cr P_H  &
0\end{pmatrix}=J.
\end{gather*}
Now by using \eqref{5.34} one obtains
\begin{gather*}
\f_U(x,\l)v_\tau^*(x,\ov\l)-v_\tau(x,\l)\f_U^*(x,\ov\l)=\\
(Y_0(x,\l)\f_{U,a}(\l))(Y_0(x,\ov\l)v_{\tau,a} (\ov\l) )^*-
(Y_0(x,\l)v_{\tau,a}(\l))(Y_0(x,\ov\l)\f_{U,a} (\ov\l))^*=\\
Y_0(x,\l)(\f_{U,a}(\l)v_{\tau,a} (\ov\l) )^*-
v_{\tau,a}(\l)\f_{U,a}^* (\ov\l))Y_0^*(x,\ov\l)=
Y_0(x,\l)JY_0^*(x,\ov\l)=J.
\end{gather*}
\end{proof}

In the following proposition we show that a canonical $m$-function
$m_\tau(\cd)$ is the Weyl function of some symmetric extension of
$\Tmi$ (in the sense of Definition \ref{def2.10}).
\begin{proposition}\label{pr5.7}
Let the assumptions {\rm (A1)--(A3)} be satisfied and let
$n_+(\Tmi)=n_-(\Tmi)$. Moreover, let $\tau\in\wt R^0(\cH_b)$ be a
boundary parameter \eqref{2.22}, let $v_\tau(\cd,\l)\in\lo{H_0}$
be the operator solution of Eq. \eqref{3.2} defined in Theorem
\ref{th4.2} and let $m_\tau(\cd)$ be the corresponding
$m$-function. Then:

 {\rm (1)} The equalities
\begin{gather*}
\wt T=\{\{\wt y,\wt f\}\in\Tma : y(a)=0, \;\wh \G_b y=0, \;
C_0\G_{0b}y+ C_1 \G_{1b}y =0 \},\\
\wt T^*=\{\{\wt y,\wt f\}\in\Tma : C_0\G_{0b}y + C_1 \G_{1b}y =0
\}
\end{gather*}
define a symmetric extension $\wt T$ of $\Tmi$ and its adjoint
$\wt T^*$;

{\rm (2)} The collection $\wt\Pi=\{H_0, \wt\G_0, \wt\G_1\}$ with
the operators
\begin {equation}\label{5.37}
\wt\G_0 \{\wt y,\wt f\}=-\G_{1a}y+i (\wh \G_a -\wh \G_b) y, \;\;
\wt\G_1
 \{\wt
y,\wt f\}=\G_{0a}y+ \tfrac 1 2 (\wh \G_a +\wh \G_b) y, \;\; \{\wt
y,\wt f\}\in\wt T^*,
\end{equation}
is a boundary triplet for $\wt T^*$. Moreover, the $\g$-field
$\wt\g(\cd)$ and Weyl function $\wt M(\cd)$ of $\wt\Pi$ are
\begin {equation}\label{5.38}
\wt\g(\l)=\pi v_\tau(\l), \qquad \wt M(\l)=m_\tau(\l), \qquad
\l\in\CR.
\end{equation}

{\rm (3)}The following identity holds
\begin {equation}\label{5.39}
m_\tau(\mu)-m_\tau^*(\l)= (\mu-\ov\l)\int_\cI v_\tau^*(t,\l)\D(t)
v_\tau(t,\mu)\, dt, \quad \mu,\l\in\CR.
\end{equation}
This implies that for the canonical $m$-function $m_\tau(\cd)$ the
inequality \eqref{5.21} turns into the eq1uality, which holds for
all $\l\in\CR$.
\end{proposition}
\begin{proof}
Clearly, we may assume that $\tau $ is given in the normalized
form \eqref{2.23}, in which case the operators
\begin{gather}
\ov\G_0\{\wt y,\wt f\}=\{-\G_{1a}y+i (\wh \G_a -\wh \G_b)
y,\;C_0\G_{0b}y+ C_1
\G_{1b}y  \}(\in H_0\oplus\cH_b), \label{5.40}\\
\ov\G_1\{\wt y,\wt f\}=\{\G_{0a}y+ \tfrac 1 2 (\wh \G_a +\wh \G_b)
y,\; C_1\G_{0b}y-C_0\G_{1b}y  \}(\in H_0\oplus\cH_b)\label{5.40a}
\end{gather}
($\{\wt y,\wt f\}\in\Tma $) form a decomposing boundary triplet
$\ov\Pi=\{\cH, \ov\G_0, \ov\G_1\}$ for $\Tma$ with
$\cH=H_0\oplus\cH_b$. Let $\ov\g(\l)$ be the $\g$-field and
\begin {equation}\label{5.41}
\ov M(\l)=\begin{pmatrix} \ov m_0(\l) & \ov M_2(\l) \cr \ov
M_3(\l) & \ov M_4(\l)
\end{pmatrix}:H_0\oplus\cH_b\to
H_0\oplus\cH_b, \quad \l\in\CR,
\end{equation}
be the Weyl function of the triplet $\ov\Pi$. Assume also that
$\ov v_0(\cd,\l)\in \lo{H_0}$ is the operator solution of Eq.
\eqref{3.2} defined in Proposition \ref{pr4.1} (for the triplet
$\ov\Pi$). Then $\ov v_0(t,\l)= v_\tau(t,\l) $ and \eqref{4.26.1}
yields $\ov\g(\l)\up H_0=\pi v_\tau (\l)$. Moreover, in view of
\eqref{4.26.17} one has $\ov m_0(\l)=m_\tau(\l), \; \l\in\CR$.
Applying now Proposition \ref{pr2.10a} to the triplet $\ov\Pi$
(with $\dot\cH_0=\dot\cH_1=H_0$) we obtain  statements  (1) and
(2). Finally, \eqref{5.39} follows from the identity \eqref{2.34b}
for the triplet $\wt\Pi$ and Lemma \ref{lem3.2}, (2) applied to
the solution $v_\tau(\cd,\l)$.
\end{proof}
\begin{remark}\label{rem5.8}
Let under the assumptions (A1)--(A3) $\Pi=\bt$ be an ordinary
decomposing boundary triplet for $\Tma$, let $\tau\in\wt
R^0(\cH_b)$ be a boundary parameter given in the normalized form
\eqref{2.22}, \eqref{2.23} and let $\ov\Pi=\{\cH, \ov\G_0,
\ov\G_1\}$ be another  decomposing boundary triplet for $\Tma$
defined by \eqref{5.40} and \eqref{5.40a}. The triplets $\Pi$ and
$\ov\Pi$ are connected by
\begin {equation*}
\begin{pmatrix}\ov\G_0 \cr\ov\G_1 \end{pmatrix}=\begin{pmatrix} X_1 & X_2
\cr X_3 & X_4 \end{pmatrix}
\begin{pmatrix} \G_0 \cr \G_1 \end{pmatrix},
\end{equation*}
where $X_j\in [H_0\oplus\cH_b] $ are  defined as follows:
\begin {equation*}
X_1=\begin{pmatrix} I & 0 \cr 0 &C_0 \end{pmatrix}, \;\;\;
X_2=\begin{pmatrix} 0 & 0 \cr 0 & -C_1 \end{pmatrix} , \;\;\;
X_3=\begin{pmatrix} 0 & 0 \cr 0 & C_1\end{pmatrix} , \;\;\;
X_4=\begin{pmatrix} I & 0 \cr 0 &C_0 \end{pmatrix}.
\end{equation*}
Therefore according to \cite{DM95} the Weyl functions $M(\cd)$ and
$\ov M(\cd)$ of the triplets $\Pi$ and $\ov\Pi$ respectively are
connected by means of linear fractional transformation,
\begin {equation} \label{5.45}
\ov M(\l)=(X_3+X_4 M(\l))(X_1+X_2 M(\l))^{-1}.
 \end{equation}
By using the block representation \eqref{4.26.16} of $M(\l)$ one
obtains
\begin{gather*}
(X_1+X_2 M(\l))^{-1}  =\begin{pmatrix}I & 0 \cr -C_1M_3 &
C_0-C_1M_4
\end{pmatrix}^{-1}=\\
\begin{pmatrix}I & 0 \cr  (C_0-C_1M_4)^{-1}C_1M_3
& (C_0-C_1M_4)^{-1}\end{pmatrix}
\end{gather*}
\noindent and \eqref{5.45}, \eqref{5.41}  imply that $\ov m_0(\l)$
coincides with the right-hand side of \eqref{5.8a}. This and the
equality $m_\tau(\l)=\ov m_0(\l)$ obtained in the proof of
Proposition \ref{pr5.7} yield \eqref{5.8a}. Thus, for canonical
$m$-functions $m_\tau(\cd)$ formula \eqref{5.8a} is a simple
consequence of the relation \eqref{5.45} for Weyl functions.

Note that in this proof we follow the reasonings of  \cite[Remark
86]{DM95}, where the Krein formula for canonical resolvents was
proved in a similar  way.
\end{remark}

\section{Spectral functions}
\subsection{Green's function}
In the sequel we put $\gH:=\LI$ and denote by $\gH_b$ the set of
all $\wt f\in\gH$ with the following property: there exists
$\b\in\cI$ (depending on $\wt f$) such that for some (and  hence
for all) function $f\in\wt f$ the equality $\D(t)f(t)=0 $ holds
a.e. on $(\b, b)$.

Assume hypothesis  (A1) and (A2). Let  $\f_U(\cd,\l)$ be the
operator-valued solution \eqref{3.32.1}, let $\tau$ be a boundary
parameter and let $v_\tau(\cd,\l)\in\lo{H_0}$ be the
operator-valued  solution of Eq. \eqref{3.2} defined in Theorem
\ref{th4.2}.
%

%
\begin{definition}
The operator function $G_\tau (\cd,\cd,\l):\cI\times\cI\to [\bH]$
given by
\begin {equation} \label{6.1}
G_\tau (x,t,\l)=\begin{cases} v_\tau(x,\l)\, \f_U^*(t,\ov\l), \;\;
x>t \cr \f_U(x,\l)\, v_\tau^* (t,\ov\l), \;\; x<t \end{cases},
\quad \l\in \CR
\end{equation}
will be called the Green's function corresponding to the boundary
parameter $\tau$.
\end{definition}

Next we compute  the generalized resolvent of $T$ in terms of the
Green's function.
\begin{theorem}\label{th6.2}
Let $\tau$ be a boundary parameter and let $R_\tau(\cd)$ be the
corresponding generalized resolvent of the relation $T$ (see
Theorem \ref{th4.0.1}). Then
\begin {equation}\label{6.2}
R_\tau(\l)\wt f=\pi\left(\int_\cI  G_\tau (\cd,t,\l)\D(t)f(t)\, dt
\right ), \quad \wt f\in\gH, \quad  f\in\wt f.
\end{equation}
\end{theorem}
\begin{proof}
Since $v_\tau(\cd, \ov\l)\in\lo{H_0}$, it follows from \eqref{6.1}
that
$$\int_\cI||G_\tau(x,t,\l)\D^{\frac 1 2}(t)||^2 \, dt<\infty, \; x\in\cI.$$ Hence
for each $f\in\lI$ and $x\in\cI$ one has
\begin {equation*}
\int_\cI ||G_\tau(x,t,\l)\D(t)f(t)||\, dt \leq \int_\cI
||G_\tau(x,t,\l)\D^{\frac 1 2}(t)||\cd || \D^{\frac 1
2}(t)f(t)||\, dt <\infty
\end{equation*}
and, therefore, the equality
\begin {equation} \label{6.3}
y_f=y_f(x,\l):=\int_\cI  G_\tau (x,t,\l)\D(t)f(t)\, dt, \quad
f\in\lI, \quad \l\in\CR
\end{equation}
correctly defines the function $y_f(\cd,\cd):\cI\times \CR\to
\bH$. This implies that  \eqref{6.2} is equivalent to the
following statement: for each $\wt f\in\gH$
\begin {equation} \label{6.4}
y_f(\cd,\l)\in\lI \;\;\text{and}\;\; R_\tau(\l)\wt f = \pi
(y_f(\cd,\l)), \quad f\in\wt f, \;\;\;\l\in\CR.
\end{equation}
To prove \eqref{6.4} first assume that $\wt f\in\gH_b$. We show
that in this case the function $y_f(\cd,\l)$ given by \eqref{6.3}
is a solution of the boundary problem
\eqref{4.0.1}--\eqref{4.0.3.2}. It follows from \eqref{6.1} that
\begin {equation} \label{6.5}
y_f=y_f(x,\l)=\f_u(x,\l)C_1(x,\l)+v_\tau(x,\l)C_2(x,\l)=Y(x,\l)C(x,\l),
\end{equation}
where
\begin {gather*}
C_1(x,\l)=\int_x^b  v_\tau^*(t,\ov\l)\D(t) f(t)\, dt, \qquad
C_2(x,\l)=\int_a^x \f_U^*(t,\ov\l)\D(t) f(t)\, dt,\\
Y(x,\l)=(\f_u(x,\l),\,v_\tau(x,\l) ), \quad C(x,\l)=\{C_1(x,\l),
C_2(x,\l)\}(\in H_0\oplus H_0).
\end{gather*}
Moreover, by \eqref{6.5} and the equality $\D(t)f(t)=0$ (a.e. on
$(\b,b)$) one has
\begin {equation} \label{6.6}
y_f(x,\l)=v_\tau(x,\l)\int_\cI \f_U^*(t,\ov\l)\D(t) f(t)\, dt,
\quad x\in (\b,b).
\end{equation}
This implies that $y_f\in \AC\cap \lI$.  Next, in view of
\eqref{5.33} one has
\begin {gather*}
Y(x,\l)C'(x,\l)=(-\f_u(x,\l)v_\tau^*(x,\ov\l)+v_\tau(x,\l)\f_U^*(x,\ov\l))
\D(x)f(x)=\\
-J \D(x)f(x).
\end{gather*}
By using this equality we obtain
\begin {gather*}
J y_f'(x,\l) - B(x)y_f(x,\l)=(J Y'(x,\l)-B(x)Y(x,\l))C(x,\l)+\\
J Y(x,\l)C'(x,\l)= \l\D(x)Y(x,\l)C(x,\l)-J^2
\D(x)f(x)=\\
\l\D(x)y_f(x,\l)+\D(x)f(x).
\end{gather*}
Thus, for each $\l\in\CR$ the function $y_f(\cd,\l)$ satisfies
\eqref{4.0.1} a.e. on $\cI$.

Next we show that $y_f(\cd,\l)$ satisfies the boundary conditions
\eqref{4.0.2}
 \eqref{4.0.3.2}. Let $\wt U$ be a $J$-unitary extension \eqref{3.17.5} of $U$
and let $\G_a$ be the mapping \eqref{3.21}. Since by \eqref{6.5}
\begin {equation*}
\G_a y_f=\f_{U,a}(\l)\int_\cI  v_\tau^*(t,\ov\l)\D(t)f(t)\, dt,
\end{equation*}
it follows from \eqref{3.32.2} that
\begin {gather} \label{6.7}
\wh\G_a y_f=P_{\wh H} \int _\cI v_\tau^*(t,\ov\l)\D(t)f(t)\,
dt=\int _\cI
(v_\tau (t,\ov\l)\up\wh H)^*\D(t)f(t)\, dt,\\
 \G_{1a}y_f=0.\label{6.8}
\end{gather}
Moreover, according to  Remark \ref{rem3.2a}, (1) the equality
\eqref{6.6} yields
\begin {gather}
\wh\G_b y_f=\wh\G_b v_\tau(\l)\int _\cI \f_U^*(t,\ov\l)\D(t)f(t)\,
dt,
\label{6.10}\\
\G_{0b} y_f=\G_{0b} v_\tau(\l)\int _\cI \f_U^*(t,\ov\l)\D(t)f(t)\,
dt,
\label{6.11a}\\
\quad \G_{1b} y_f=\G_{1b} v_\tau(\l)\int _\cI
\f_U^*(t,\ov\l)\D(t)f(t)\, dt .\label{6.11b}
\end{gather}
In view of \eqref{6.8} the first condition in \eqref{4.0.2} is
fulfilled. Next, by \eqref{4.27a} and \eqref{5.1} one has
\begin {equation} \label{6.12}
\wh\G_b v_\tau(\l)=\wh\G_a v_\tau(\l)+i P_{\wh H}=(P_{\wh
H}m_\tau(\l)-\tfrac i 2 P_{\wh H} )+iP_{\wh H}=P_{\wh
H}(m_\tau(\l)+\tfrac i 2 I_{H_0}).
\end{equation}
Moreover, in view of \eqref{5.1}
$$
v_{\tau,a}(\l)\up \wh H= \begin{pmatrix} (m_\tau(\l)-\tfrac i 2
I_{H_0})\up \wh H \cr 0 \end{pmatrix}:\wh H\to H_0\oplus H, \quad
\l\in\CR,
$$
and \eqref{3.32.2} gives $v_\tau(t,\l)\up\wh H=
\f_U(t,\l)(m_\tau(\l)-\tfrac i 2 I_{H_0})\up \wh H$. Combining
this equality with \eqref{6.10}, \eqref{6.12} and  \eqref{6.7} one
obtains
\begin {gather*}
\wh\G_b y_f =P_{\wh H}(m_\tau(\l)+\tfrac i 2 I_{H_0})\int _\cI
\f_U^*(t,\ov\l)\D(t)f(t)\, dt =\\\int _\cI
(\f_U(t,\ov\l)(m_\tau(\ov\l)-\tfrac
i 2 I_{H_0})\up\wh H)^*\D(t)f(t)\, dt= \\
\int _\cI (v_\tau (t,\ov\l)\up\wh H)^*\D(t) f(t)\, dt=\wh \G_a y_f
\end{gather*}
(here we make use of the relation $m_\tau^*(\ov\l)=m_\tau (\l)$).
Hence the second condition in \eqref{4.0.2} is fulfilled. Finally
combining \eqref{6.11a} and \eqref{6.11b} with  \eqref{4.28} and
\eqref{4.28a} one obtains the relations \eqref{4.0.3.1} and
\eqref{4.0.3.2} for $y_f$. Thus $y_f(\cd,\l)$ is a solution of the
boundary problem \eqref{4.0.1}-- \eqref{4.0.3.2} and by Theorem
\ref{th4.0.1} the relations \eqref{6.4} hold.

Now assume that $\wt f\in \gH$ is arbitrary,   $f\in\wt f$, and
$y_f=y_f(x,\l)$ is given by  \eqref{6.3}. Assume also that $f_n=f
\chi_{[a,b-\frac 1 n ]}, \; \wt f_n =\pi f_n(\in \gH_b)$ and let
$y_{f_n}=y_{f_n}(x,\l)$ be given by \eqref{6.3} with $f_n(t)$ in
place of $f(t)$. Moreover, let a function $y_R\in\lI$ be such that
$\pi y_R=R_\tau (\l) \wt f $. Since $\wt f_n \to \wt f$ and $\pi
y_{f_n}=R_\tau(\l)\wt f_n$, it follows that $||y_R-y_{f_n}||_\D\to
0$. On the other hand, $y_{f_n}(x,\l)\to y_f(x,\l), \; x\in\cI,$
and, consequently, $\D (x)(y_f(x,\l)-y_R(x,\l))=0$ a.e. on $\cI$.
Hence $y_f\in\lI$ and $\pi y_f=\pi y_R=R_\tau (\l)\wt f$, which
gives the relations \eqref{6.4} for $\wt f$.
 \end{proof}
\begin{remark}\label{rem6.2a}
Theorem \ref{6.2} generalizes several results in this direction.
More precisely, in the case of  Hamiltonian system \eqref{3.1}
($\wh H=\{0\}$) and separated boundary conditions formulas
\eqref{6.1} and  \eqref{6.2}  for \emph{canonical resolvents} of
$\Tmi$ were proved in \cite{HinSch93,Kra89}. Moreover, assuming
that the minimal operator $\Tmi$ is generated by Hamiltonian
system  with  the minimal deficiency indices  $n_\pm(\Tmi)=\dim
H$, formulas \eqref{6.1} and  \eqref{6.2} for \emph{generalized
resolvents}  of $\Tmi$ have been obtained in \cite{DLS88,DLS93}.
Note also that formulas for canonical and generalized resolvents
of even order ordinary differential equations subject to separated
boundary conditions are known  as late as the middle of nineteenth
(see e.g. \cite{DunSch,Nai,Sht57}).
\end{remark}

\subsection{The space $\LSH$}\label{sect6.2}
Let $\cH$ be a finite dimensional Hilbert space.
\begin{definition}\label{def6.3}
A non-decreasing operator function $\Si(\cd): \bR\to [\cH]$ is
called a distribution function if it is left continuous and
satisfies the equality $\Si(0)=0$.
\end{definition}

Next recall the definition of the space $\LSH$ (see e.g.
\cite[Section 20.5]{Nai}, \cite[Section 7.2.3]{Ber}). Denote by
$C_{0}(\cH)$ the set of continuous vector functions $f: \bR\to
\cH$ having  compact supports. Introduce the semi-scalar product
on $C_{0}(\cH)$  by setting
\begin {equation} \label{6.13}
(f,g)_{\LSH}=\int_\bR(d \Si(t)f(t),g(t)_\cH)=\lim_{d(\pi_n)\to
0}\sum_{k=1}^n (\Si(\D_k)f(\xi_k),g(\xi_k)).
\end{equation}
Here  $\pi_n=\{a=t_0< t_1<\dots <t_n=b\}$ denotes  a partition of
a segment $[a,b]$ containing the supports of functions $f$ and
$g$, $d(\pi_n)$ is the diameter of the partition $\pi_n$,
$\Si(\D_k) := \Si(t_k)-\Si(t_{k-1}),$ and $\xi_k\in
[t_{k-1},t_k]$. The limit in \eqref{6.13} is understood in the
same sense  as in the definition of the Riemann-Syieltjes
integral, i.e., a particular choice of $\pi_n$ with a given
diameter and of $\xi_k\in [t_{k-1},t_k]$ is irrelevant.

The completion of $C_{0}(H)$ with respect to the semi-norm
$p(f):=(f,f))_{\LSH}^{\frac1 2}$ gives rise to a  semi-Hilbert
space $\wt L^2(\Si,H)$ (i.e., to a complete space with a semi-norm
in place of norm). Denoting by $\ker p:=\{f\in \wt
L^2(\Si,H):p(f)=0\}$ the kernel of the semi-norm, we introduce the
quotient space  $\LSH := \wt L^2(\Si,H)/\ker p$ which is already
Hilbert space.

Let $\Si=(\s_{ij})_{i,j=1}^n$ be a matrix valued  measure
generated by a distribution function $\Si(\cd)$ and let
$\s=\sum_{j} \s_{jj}$. Clearly,  the measure $\Si(\cdot)$ is
absolutely continuous with respect to $\s$ (in fact both measures
are equivalent). Therefore, by the Radon-Nykodim theorem, there
exists a $\s$-measurable matrix density
$\Psi(\cdot)=(\psi_{ij}(\cdot))_{i,j=1}^n$ such that
\begin {equation*}
\Si(\delta)=\int_\delta \Psi (t)d\s (t), \qquad \Psi
(t):=(\psi_{ij}(t))_{i,j=1}^n=(d \s_{ij}/d\s)_{i,j=1}^n,\quad
\delta\in \mathcal B_b(\Bbb R).
\end{equation*}
Let $\wt L_0^2(\Si,\bC^n)$ be the set of $\s$-measurable
vector-valued functions $f:\bR\to\bC^n$ satisfying
  \begin {equation}\label{6.14a}
||f||_{\wt L_0^2(\Si,\bC^n)}^2:=\int_\bR(\Psi(t)f(t),
f(t))\,d\s(t)<\infty.
  \end{equation}
%
\begin{theorem}\cite{Kac50}\label{thKac}
The spaces $\wt L^2(\Si,\bC^n)$ and $L^2(\Si,\bC^n)$ are
identified isometrically with the spaces $\wt L_0^2(\Si,\bC^n)$
and $L_0^2(\Si,\bC^n):=\wt L_0^2(\Si,\bC^n)/N_0$, respectively,
where $N_0=\{f\in \wt L_0^2(\Si,\bC^n):||f||_{\wt
L_0^2(\Si,\bC^n)}\}=0$ is the kernel of the semi-norm. Therefore,
$f\in \wt L^2(\Si,\bC^n)$ if and only if $f$ is $\s$-measurable
and the norm \eqref{6.14a} is finite.
  \end{theorem}

It was shown in \cite{MalMal03} that the spaces $\wt
L^2(\Si,\bC^n)$ and $L^2(\Si,\bC^n)$ admit the representation in
the form of direct integrals
\begin {equation}\label{6.15}
\wt L^2(\Si,\bC^n)=\int_\bR \oplus \wt G(t) \, d\s(t), \quad
L^2(\Si,\bC^n)=\int_\bR \oplus  G(t) \, d\s(t),
\end{equation}
where $\wt G(t)$is the $n$-dimensional Euclidian space with the
semi-scalar product $\langle f,g \rangle:=(\Psi (t)f,g)$ and
$G(t)=\wt G(t)/\{f\in \wt G(t): (\Psi (t)f,f)=0\}$. In particular,
representation  \eqref{6.15} gives a simple proof of Theorem
\ref{thKac} (as distinguished from the known  proofs in
\cite{Kac50} and \cite{DunSch}).

\subsection{Spectral functions and the Fourier transform}
Let the assumptions (A1) and (A2) at the beginning of Section
\ref{sect4} be satisfied and let $\f_U(\cd,\l)(\in [H_0,\bH])$  be
the operator solution of Eq. \eqref{3.2} with the initial data
\eqref{3.32.1}. For each $\wt f\in\gH_b$ introduce the Fourier
transform $\wh f(\cd):\bR\to H_0$ by setting
\begin {equation} \label{6.18}
\wh f(s)=\int_\cI \f_U^*(t,s)\D(t)f(t)\, dt, \quad f\in \wt f.
\end{equation}
Note that  $\wh f(\cd)$ is uniquely defined by $\wt f$, i.e., it
does not depend on the choice of $f\in\wt f$.

Next assume that $\pair \in \wt R(\wt\cH_b, \cH_b)$ is a boundary
parameter given by \eqref{2.2} (with $\cH_0=\wt\cH_b$ and
$\cH_1=\cH_b$). Then according to Theorem \ref{th4.0.1} the
corresponding boundary problem \eqref{4.0.1}--\eqref{4.0.3.2}
generates the generalized resolvent
\begin {equation} \label{6.18.1}
R_\tau(\l)=P_\gH (\wt T^\tau-\l)^{-1}\up \gH, \quad \l\in\CR,
\end{equation}
of the symmetric relation $T\in \C (\gH)$. The equality
\eqref{6.18.1} uniquely (up to the unitary equivalence) defines a
self-adjoint $\gH$-minimal relation $\wt T^\tau$ in $\wt
\gH\supset  \gH$ such that $T\subset \wt T^\tau$. Denote also by
$F_\tau(\cd)$ the corresponding spectral function of $T$, so that
in view of \eqref{2.38}
\begin {equation} \label{6.18.2}
R_\tau(\l)=\int_\bR \frac {dF_\tau(t)}{t-\l}, \quad \l\in\CR.
\end{equation}
In the following we fix some $J$-unitary extension $\wt U$ of $U$
(see \eqref{3.17.5}) and denote by $m_\tau(\cd)$  the $m$-function
of the boundary problem \eqref{4.0.1}--\eqref{4.0.3.2}. Note that
in view of Proposition \ref{pr5.2a} a choice of $\wt U$ does not
matter in our further considerations.
\begin{definition}\label{def6.4}
A distribution function $\Si(\cd)=\Si_\tau(\cd):\bR\to [H_0]$ is
called a spectral function of the boundary problem
\eqref{4.0.1}--\eqref{4.0.3.2} if, for each $\wt f\in\gH_b$ and
for each finite interval $[\a,\b) \subset \bR$, the Fourier
transform \eqref{6.18} satisfies the equality
\begin {equation} \label{6.19}
((F_\tau(\b)-F_\tau(\a))\wt f,\wt
f)_\gH=\int_{[\a,\b)}(d\Si_\tau(s)\wh f(s),\wh f(s)).
\end{equation}
\end{definition}
Note that the integral on the right-hand side of \eqref{6.19}
exists, since the function $\wh f(\cd)$ is continuous (and even
holomorphic) on $\bR$; moreover, by \eqref{6.19} $\wh f\in\LS$.

Let $\wt\gH_0:=\wt\gH\ominus \mul \Tt$, so that
\begin {equation} \label{6.20}
\wt\gH=\wt\gH_0\oplus\mul \Tt
\end{equation}
Then by \eqref{2.39} and \eqref{6.19} one has
\begin {equation} \label{6.21}
||P_{\wt\gH_0}\wt f||_{\wt\gH}=||\wh f||_{\LS}, \quad \wt
f\in\gH_b
\end{equation}
and, consequently, $||\wh f||\leq ||\wt f||$. Therefore for each
$\wt f\in\gH$ there exists a function $\wh f\in\LS$ (the Fourier
transform of $\wt f$) such that
\begin {equation*}
\lim\limits_{\b\uparrow b}||\wh f - \int_a^\b
\f_U^*(t,\cd)\D(t)f(t)\, dt
 ||_{\LS}=0, \quad f\in\wt f,
\end{equation*}
and the equality $V\wt f=\wh f, \; \wt f\in\gH,$ defines the
linear operator $V:\gH\to \LS$ such that
\begin {equation} \label{6.23}
||V\wt f||_{\LS}=||P_{\wt\gH_0}\wt f||_{\wt\gH}, \quad \wt
f\in\gH.
\end{equation}
This implies that $V$ is a contraction from $\gH$ to $\LS$.
\begin{theorem}\label{th6.5}
For each boundary parameter $\tau$ there exists a unique spectral
function $\Si_\tau(\cd)$ of the boundary problem
\eqref{4.0.1}--\eqref{4.0.3.2}. This function is defined by the
Stieltjes inversion formula
\begin {equation} \label{6.24}
\Si_\tau(s)=\lim\limits_{\d\to+0}\lim\limits_{\varepsilon\to +0}
\frac 1 \pi \int_{-\d}^{s-\d}\im \,m_\tau(\s+i\varepsilon)\, d\s.
\end{equation}
\end{theorem}
\begin{proof}
It follows from Proposition \ref{pr5.3} that the Green function
\eqref{6.1} admits the representation
\begin {gather*}
G_\tau(x,t,\l)=\f_U(x,\l) m_\tau(\l)\f_U^*(t,\ov\l)+G_0(x,t,\l),\\
\text{where} \quad
 G_0 (x,t,\l)=\begin{cases} \psi (x,\l)\, \f_U^*(t,\ov\l), \;\; x>t
\cr \f_U(x,\l)\, \psi^* (t,\ov\l), \;\; x<t \end{cases}
\qquad\qquad\qquad\qquad
\end{gather*}
Now by using  \eqref{6.2} and the Stieltjes-Liv\u{s}ic inversion
formula one proves the theorem in the same way as Theorem 4 in
\cite{Sht57}.
\end{proof}
Next, similarly to \cite{DunSch,Nai,Sht57} one can prove the
following theorem.
\begin{theorem}\label{th6.5a}
Let $V:\gH\to\LS$ be the Fourier transform corresponding to the
spectral function $\St$ and let $V^*$ be the operator adjoint to
$V$. Then for each function $g=g(s)\in\LS$ with the compact
support the function
\begin {equation*}
f_g(t):=\int_\bR \f_U(t,s)\, d\Si_\tau(s) g(s)
\end{equation*}
belongs to  $\lI$ and $V^*g=\pi f_g$. Therefore
\begin {equation} \label{6.25}
V^* g=\pi\left(\int_\bR \f_U(\cd,s)\,d\Si_\tau(s) g(s)\right),
\quad g=g(s)\in \LS,
\end{equation}
where the integral converges in the semi-norm \eqref{3.0}.
\end{theorem}

For a spectral function $\Si_\tau(\cd)$ denote by $\L$ the
multiplication operator in $\LS$ given by the relations
\begin {equation} \label{6.26}
\begin{array}{c}
\dom \L=\{f\in\LS:tf(t)\in\LS\} , \\
(\L f)(t)=tf(t), \;\; f\in\dom\L.
\end{array}
\end{equation}
As is known $\L$ is a self-adjoint operator and  the spectral
measure $E_\L(\cd)$ of $\L$ is
\begin {equation} \label{6.27}
(E_\L (\d)f)(t)=\chi_{\d}(t)f(t), \quad f\in\LS, \quad \d\in\cB,
\end{equation}
where $\chi_\d(\cd)$ is the indicator of the Borel set $\d$.
Moreover, in view of \eqref{6.19} one has
\begin {equation} \label{6.27a}
F_\tau(\b)-F_\tau(\a)=V^* E_\L ([\a,\b))V, \quad [\a,\b)\subset
\bR.
\end{equation}
\begin{proposition}\label{pr6.6}
Let $\St$ be a spectral function of the boundary problem
\eqref{4.0.1}-- \eqref{4.0.3.2}, let $V:\gH\to \LS$ be the
corresponding Fourier transform and let
\begin {equation} \label{6.28}
L_0:={\rm clos}\, (V\gH)
\end{equation}
Then the operator $\L$ is $L_0$-minimal (in the sense of
Definition \ref{def2.10c}).
\end{proposition}
\begin{proof}
Let $L_1:=\ker V^*(=\LS \ominus L_0)$ and let $g\in L_1$ be a
vector such that $E_\L ([\a,\b))g\in L_1$ for each bounded
interval $[\a,\b)\subset \bR$. Then $\L E_\L ([\a,\b))g\in L_1$
and, consequently, $V^* E_\L ([\a,\b))g=0$ and $V^* \L E_\L
([\a,\b))g=0$. This and \eqref{6.25} imply that the functions
\begin {equation}\label{6.30}
y(t)=\int_{[\a,\b)}\f_U(t,s)\,d\Si_\tau(s)g(s), \;\; f(t)=
\int_{[\a,\b)}\f_U(t,s)\,d\Si_\tau(s) s g(s)
\end{equation}
satisfy the equalities
\begin {equation} \label{6.31}
\D(t)y(t)=0 \;\;\text{and} \;\;\D(t)f(t)=0 \;\;\text{a.e. on}
\;\;\cI.
\end{equation}
On the other hand, in view of \eqref{6.30} one has
\begin{gather*}
J y'(t)-B(t)y(t)=\int_{[\a,\b)}(J\f_U'(t,s)-B(t)\f_U(t,s))\, d
\Si_\tau(s)g(s)=\\
\int_{[\a,\b)}(s\D(t)\f_U(t,s))\, d \Si_\tau(s)g(s)=\D(t)f(t).
\end{gather*}
Combining this equality with \eqref{6.31} and taking definiteness
of the system \eqref{3.1} into account one gets
\begin {equation}\label{6.32}
y(t)=\int_{[\a,\b)}\f_U(t,s)\,d\Si_\tau(s)g(s)=0, \quad t\in\cI,
\quad [\a,\b)\subset \bR.
\end{equation}
It follows from \eqref{3.32.1} and \eqref{3.32.2} that the
operator $\f_U(a,s)$ dos not depend on $s$ and $\ker
\f_u(a,s)=\{0\}$. This and \eqref{6.32} yield
\begin{gather*}
\int_{[\a,\b)}d\Si_\tau(s)g(s)=0, \quad [\a,\b)\subset \bR,
\end{gather*}
which gives the equality $g=0$. Thus the condition (2) of
Definition \ref{def2.10c} is satisfied.
\end{proof}
Let $\wt\gH$ be  decomposed as in  \eqref{6.20} and let
\begin {equation}\label{6.33}
\gH_V:=\wt\gH_0\cap\gH, \quad \gH_k:=\mul \Tt\cap\gH, \quad
\gH_c=\gH\ominus(\gH_V\oplus\gH_k).
\end{equation}
Then
\begin {equation}\label{6.34}
\gH=\gH_V\oplus\gH_k\oplus\gH_c
\end{equation}
and by \eqref{6.21} the operator $V$ (the Fourier transform) is
isometric on $\gH_V$, strictly contractive on $\gH_c$ and has
$\gH_k$ as a kernel. Observe also that $\mul T\subset \gH_k$, so
that $V\up \mul T=0$.

Next assume that $\gH_0:=\gH\ominus\mul T$, so that $\gH$ can be
represented as
\begin {equation}\label{6.35}
\gH=\gH_0\oplus \mul T.
\end{equation}
It follows from \eqref{6.34} that $\gH_0$ is the maximally
possible subspace of $\gH$ on which the Fourier transform $V$ may
be isometric.
\begin{definition}\label{def6.7}
A spectral function $\St$ of the boundary problem
\eqref{4.0.1}--\eqref{4.0.3.2} will be referred to the class
$SF_0$ if the operator
\begin {equation}\label{6.36}
V_0:=V\up\gH_0
\end{equation}
is an isometry from $\gH_0$ to $\LS$.
\end{definition}
By using $\gH$-minimality of $\Tt$ one can easily show that
\begin {equation}\label{6.37}
\St\in SF_0\iff \mul\Tt=\mul T.
\end{equation}
Therefore all spectral functions belong to $SF_0$ if and only if
$\mul T=\mul T^*$.

If $\St\in SF_0$, then by   \eqref{6.25} for each $\wt f\in \gH_0$
the inverse Fourier transform is
\begin {equation}\label{6.38}
\wt f=\pi\left(\int_\bR \f_U(\cd,s)\,d\Si_\tau(s) \wh f(s)\right)
\end{equation}
\begin{theorem}\label{th6.9}
Let $\tau$ be a boundary parameter, let $\St$ be the spectral
function of the boundary problem \eqref{4.0.1}--\eqref{4.0.3.2}
and let $V$ be the corresponding Fourier transform. Assume also
that $\Tt\in\C (\wt\gH)$ is the (exit space) self-adjoint
extension of $T$ defined by \eqref{6.18.1}, $\wt\gH_0\subset
\wt\gH$ and $\gH_0\subset \gH$ are the subspaces from
decompositions \eqref{6.20} and  \eqref{6.35} respectively  and
$T^\tau$ is the operator part of $\Tt$ (so that $T^\tau$ is a
self-adjoint operator in $\wt\gH_0$). If $\St\in SF_0$, then
$\gH_0\subset \wt\gH_0$ and there exists a unitary operator $\wt
V\in [\wt\gH_0, \LS]$ such that $\wt V\up\gH_0=V_0(=V\up \gH_0)$
and the operators $T^\tau$ and $\L$ are unitarily equivalent by
means of $\wt V$.

Moreover, if $\mul T=\mul T^* $ (that is, the condition {\rm (C1)}
in Remark \ref{rem3.7} is fulfilled), then the statements of the
theorem hold for each spectral function $\St$.
 \end{theorem}
\begin{proof}
Since in view of \eqref{6.37} $\mul \Tt=\mul T$, it follows that
$\gH_0\subset\wt\gH_0$ and the decomposition \eqref{6.20} takes
the form
\begin {equation}\label{6.42}
\wt\gH=\wt\gH_0\oplus\mul T.
\end{equation}
It follows from \eqref{6.27a} and \eqref{2.37} that for each
finite interval $[\a,\b)\subset \bR$ the spectral function
$E_\tau(\cd)$ of $T^\tau$ satisfies the equality
\begin {equation*}
P_{\gH_0}E_\tau ([\a,\b))\up\gH_0=V_0^* E_\L
([\a,\b))V_0=V_0^*(P_{L_0}E_\L ([\a,\b))\up L_0)V_0,
\end{equation*}
where $L_0=V_0\gH_0(=V\gH_0)$. This implies that
\begin {equation}\label{6.43}
P_{\gH_0}(T^\tau -\l)^{-1}\up \gH_0= V_0^*(P_{L_0}(\L-\l)^{-1}\up
L_0)V_0, \quad \l\in\CR.
\end{equation}
Since  $\Tt$ is $\gH$-minimal, it follows from \eqref{6.35} and
\eqref{6.42} that the operator $T^\tau$ is $\gH_0$-minimal.
Moreover, according to Proposition \ref{pr6.6} the  operator $\L$
is $L_0$-minimal. Now, applying Proposition \ref{pr2.11.1} to
operators $T^\tau$ and $\L$ we arrive at the desired statements
for $\St\in SF_0$.

The last statement of the theorem follows from the fact that in
the case $\mul T=\mul T^* $ the inclusion $\St\in SF_0$ holds for
each spectral function $\St$.
\end{proof}
Combining of Theorems \ref{th6.9} and \ref{th4.0.1} yields the
following corollary.
\begin{corollary}\label{cor6.10}
Let $\tau$ be a boundary parameter and let $\St$ be the spectral
function of the boundary problem \eqref{4.0.1}--\eqref{4.0.3.2}.
Then the following statements are equivalent:

{\rm (1)} $n_+(\Tmi)=n_-(\Tmi), \; \tau \in \wt R^0(\cH_b)$ and
the canonical self-adjoint extension $\Tt$ of $T$ given by
\eqref{4.0.5} satisfies the equality $\mul \Tt=\mul T$

{\rm (2)} The Fourier transform $V$ isometrically maps $\gH_0$
onto $\LS$ (that is, $V\up \gH_0$ is a unitary operator).

If the statement {\rm (1)} (and hence {\rm (2)}) is valid, then
the operator $T^\tau$ (the self-adjoint part of $\Tt$) and the
multiplication operator $\L$ are unitarily equivalent by means of
$V$.
\end{corollary}
\begin{theorem}\label{th6.11}
Assume that $T$ is a densely defined operator, that is, the
condition {\rm (C2)} in Remark \ref{rem3.7} is fulfilled. Then for
each boundary parameter $\tau$ and the corresponding spectral
function $\St$ the following hold: {\rm (i)} $\Tt$ is an operator,
that is, $\Tt= T^\tau$;   {\rm (ii)} the  Fourier transform $V$ is
an isometry; {\rm (iii) } there exists a unitary operator $\wt
V\in [\wt\gH, \LS]$ such that $\wt V\up \gH= V$ and the operators
$\Tt$ and $\L$ are unitarily equivalent by means of $\wt V$.

Moreover, the following statements are equivalent:

{\rm (1)} $n_+(\Tmi)=n_-(\Tmi)$ and $ \tau \in \wt R^0(\cH_b)$,
so that $\Tt$ is the canonical self-adjoint extension of $T$
given by the boundary conditions \eqref{4.0.5};

{\rm (2)} $V\gH =\LS$, that is the fourier transform $V$ is a
unitary operator.

If the statement {\rm (1)} (and hence {\rm (2)}) is valid, then
the operators $\Tt$ and $\L$ are unitarily equivalent by means of
$V$.
\end{theorem}
\begin{proof}
Since $\mul T=\mul T^*=\{0\}$, the required statements are implied
by Theorem \ref{th6.9} and Corollary \ref{cor6.10}.
\end{proof}
It follows from Theorem \ref{th6.9} that the operators $T^\tau$
and $\L$ have the same spectral properties. This implies, in
particular, the following corollary.
\begin{corollary}\label{cor6.12}
{\rm (1)} If $\tau $ is a boundary parameter such that $\St \in
SF_0$, then the spectral multiplicity of the operator $T^\tau$
does not exceed $\nu_-(=\dim H_0)$.

{\rm (2)} If the condition  {\rm (C1)} in Remark \ref{rem3.7} is
fulfilled, then the above statement on the spectral multiplicity
of  $T^\tau$ holds for each boundary parameter $\tau$.
\end{corollary}
In the next theorem we give a parametrization of all spectral
functions $\St$ in terms of a boundary parameter $\tau$.
\begin{theorem}\label{th6.13}
Let $n_+(\Tmi)=n_-(\Tmi)$ and let $M(\cd)$ be the operator
function defined  by \eqref{4.26.16}--\eqref{4.26.18}. Then, for
each boundary parameter $\tau \in \wt R(\cH_b)$ given by
\eqref{2.19} the equality
\begin {equation}\label{6.45}
m_\tau(\l)=m_0(\l)+M_2(\l)(C_0(\l)-C_1(\l)M_4(\l))^{-1}C_1(\l)M_3(\l),
\quad\l\in\CR,
\end{equation}
together with \eqref{6.24} defines a (unique) spectral function
$\St$ of the boundary problem \eqref{4.0.1}--\eqref{4.0.3.2}.
Moreover, the following hold:

{\rm (1)} $\St \in SF_0$ if and only if the following two
conditions are satisfied:
\begin{gather}
\lim\limits_{y\to \infty} \frac 1 y (C_0(i y)-C_1(i y)M_4(i
y))^{-1}C_1(i
y)=0,\label{6.46}\\
\lim\limits_{y\to \infty} \frac 1 y M_4(i y) (C_0(i y)-C_1(i
y)M_4(i y))^{-1}C_0(i y)=0.\label{6.47}
\end{gather}

{\rm (2)} Each spectral function $\St$ belongs to the class $SF_0$
if and only if
\begin {equation*}
\lim\limits_{y\to \infty} \frac 1 y M_4(i y) =0
\;\;\;\text{and}\;\;\; \lim\limits_{y\to \infty}y\,\im (M_4(i
y)h,h)=+\infty, \;\; h\in\cH_b, \;\;h\neq 0.
\end{equation*}
\end{theorem}
\begin{proof}
The main statement of the theorem directly follows from Corollary
\ref{cor5.4a} and Theorem \ref{th6.5}.

Next, consider the boundary triplet $\dot\Pi=\{\cH_b,\dot \G_0,
\dot\G_1\}$ for $T^*$ defined in Proposition \ref{pr3.6}. Since
$M(\cd) $ is the Weyl function of the decomposing boundary triplet
\eqref{3.43.1}, \eqref{3.43.2}for $\Tma$, it follows from
Proposition \ref{pr2.10a}, (3) that the Weyl function of the
triplet $\dot\Pi$ coincides with $M_4(\l)$. Now applying to the
boundary triplet $\dot\Pi$ the results of \cite{DM95, DM00} we
obtain statements (1) and (2).
\end{proof}
\subsection{The case of minimal equal deficiency indices}\label{sub6.4}
In this subsection we reformulate the above results for the
simplest case of minimally possible equal deficiency indices of
$\Tmi$, which in view of \eqref{3.17c} are
\begin {equation}\label{6.50}
n_+(\Tmi)= n_-(\Tmi)=\nu_-.
\end{equation}
In this case $\nu_{b-}=0, \; \nu_{b+}=\wh\nu$ and according to
Lemma \ref{lem3.2.1} there exists a surjective linear mapping
$\wh\G_b:\dom\tma\to \wh H$ such that
\begin {equation}\label{6.51}
[y,z]_b=i(\wh\G_b y, \wh\G_b z), \quad y,z\in\dom\tma.
\end{equation}
Below we suppose that the assumption (A1) at the beginning of
Section \ref{sect4} is fulfilled and that $\wh \G_b$ is a
surjective operator satisfying \eqref{6.51}.

It follows from Proposition \ref{pr3.6} and Theorem \ref{th4.0.1}
that the equality
\begin {equation*}
T=\{\{\wt y,\wt f\}\in\Tma: \G_{1a}y=0, \; \wh\G_a y = \wh\G_b y
\}
\end{equation*}
defines a self-adjoint relation $T$ in $\gH(=\LI)$ and the
(canonical) resolvent of $T$ is given by the boundary value
problem
\begin{gather}
J y'-B(t)y=\l\D (t)y +\D(t) f(t), \quad t\in\cI,\label{6.53}\\
\G_{1a}y =0,  \qquad \wh\G_a y=\wh\G_b y,
\;\;\;\l\in\CR.\label{6.54}
\end{gather}

Next, in view of Theorem \ref{th4.2} for each $\l\in\CR$ there
exists a unique operator solution $v(\cd,\l)\in\lo{H_0}$ of Eq.
\eqref{3.2} such that
\begin {equation}\label{6.55}
\G_{1a} v(\l)=-P_H, \qquad i(\wh\G_a -\wh\G_b)v(\l)=P_{\wh H},
\quad \l\in\CR.
\end{equation}
Moreover, if $\wt U$ is a $J$-unitary extension \eqref{3.17.5} of
$U$ and $\G_{0a}$ is the mapping \eqref{3.23}, then the
(canonical) $m$-function $m(\cd)$ of the  problem \eqref{6.53},
\eqref{6.54} is given by the equality
\begin {equation}\label{6.56}
m(\l)=(\G_{0a}+\wh \G_a)v (\l)+\tfrac i 2 P_{\wh H},
\quad\l\in\CR,
\end{equation}
or, equivalently, by the relations
\begin {equation*}
v(t,\l):=\f_U (t,\l)m(\l)+\psi (t,\l)\in\lo{H_0}, \;\;
i(\wh\G_a-\wh\G_b) v(\l)=P_{\wh H}, \;\;\l\in\CR.
\end{equation*}

The boundary problem \eqref{6.53}, \eqref{6.54} has a unique
spectral function $\Si(\cd)$, which is defined by the Stieltjes
formula \eqref{6.24} with $m_\tau(\cd)=m(\cd)$. Moreover
,Corollary \ref{cor6.10} implies that the corresponding Fourier
transform $V$ isometrically maps $\gH_0(=\gH\ominus \mul T)$ onto
$L^2(\Si;H_0)$.
\subsection{Example}
In this subsection we provide an example illustrating the results
of the paper.

Let $\cI=[0,\infty)$ and let $\d(\cd)$ be a Borel  function on
$\cI$ such that $\d(t)>0$ (a.e. in $\cI$) and
\begin {equation*}
C:=\int_0^\infty \d(t) \, dt<\infty.
\end{equation*}
Assume also that in formulas \eqref{3.16} and \eqref{3.17a} $H=\wh
H=\bC$, so that $\bH=\bC^3$ and $H_0=\bC^2$. Consider the
symmetric system
\begin {equation}\label{6.57}
Jy'=\D(t) f(t), \quad t\in\cI, \quad f\in\lI,
\end{equation}
where $J$ and $\D(t)$ are given by the matrices
\begin {equation*}
J=\begin{pmatrix} 0& 0 & -1 \cr 0 & i & 0 \cr 1& 0 & 0
\end{pmatrix}, \qquad \D(t)=\begin{pmatrix}\tfrac 1 2 (\d(t)+1) &
0 & \tfrac i 2 (\d(t)-1)\cr 0 & 1 & 0 \cr -\tfrac i 2 (\d(t)-1) &
0 & \tfrac 1 2 (\d(t)+1)
\end{pmatrix}.
\end{equation*}
Clearly, $\D(t)$ is a nonnegative  invertible matrix (a.e. on
$\cI$); therefore the system \eqref{6.57} is definite and $\Tmi$
is a densely defined operator in $\LI$. The immediate checking
shows that  the homogeneous system
\begin {equation}\label{6.59}
Jy'=\l\D(t) y, \quad t\in\cI, \quad \l\in\bC,
\end{equation}
has a fundamental solution
\begin {equation}\label{6.60}
Y(t,\l)=\begin{pmatrix} e^{-i \l \Phi(t)} & 0 & e^{i \l t} \cr 0 &
e^{-i \l t}& 0 \cr -i e^{-i \l \Phi(t)} & 0 & ie^{i \l t}
\end{pmatrix},
\end{equation}
where
\begin {equation*}
\Phi (t):=\int_0^t \d(s)\, ds.
\end{equation*}

Denote by $y^{(1)}(\cd,\l), \; y^{(2)}(\cd,\l) $ and
$y^{(3)}(\cd,\l)$ vector solutions of Eq. \eqref{6.59} formed by
the first, second and third columns of the matrix \eqref{6.60}
respectively. It is easily seen that $y^{(1)}(\cd,\l), \;
y^{(3)}(\cd,\l)\in\lI, \; y^{(2)}(\cd,\l)\notin\lI$ for all
$\l\in\bC_+$ and $y^{(1)}(\cd,\l), \; y^{(2)}(\cd,\l)\in\lI, \;
y^{(3)}(\cd,\l)\notin\lI$ for all $\l\in\bC_-$. Therefore the
operator $\Tmi$ has minimally possible equal deficiency indices
$n_+(\Tmi)=n_-(\Tmi)=2$.

Let $\t(\cd)\in\lI$ be the solution of Eq. \eqref{6.59} given by
\begin {equation*}
\t(t)=\tfrac i {\sqrt 2} e^{-C}y^{(1)}(t,i)= \tfrac i {\sqrt 2}
e^{-C}\{e^{\Phi(t)}, 0, -i e^{\Phi(t)}\}.
\end{equation*}
Since $[\t,\t]_\infty=i$, it follows from Remark \ref{rem3.2a},
(2) that the equality $ \wh\G_b y=[y,\t]_\infty, \; y\in\dom\tma,$
defines the surjective linear mapping $\wh \G_b:\dom\tma\to \bC$
satisfying \eqref{6.51}.

We assume that $\wt U=I$ (see \eqref{3.17.5}). Then for each
function $y\in\dom\tma$ decomposed as
\begin {equation*}
y(t)=\{y_0(t), \wh y(t), y_1(t)\}(\in \bC\oplus\bC\oplus\bC),
\quad t\in\cI,
\end{equation*}
one has $\G_{0a}y=y_0(0), \; \wh\G_a y=\wh y(0), \;
\G_{1a}y=y_1(0) $  and the boundary problem \eqref{6.53},
\eqref{6.54} can be written as
\begin{gather}
J y'=\l\D (t)y +\D(t) f(t), \quad t\in\cI,\label{6.63}\\
y_1(0) =0,  \qquad \wh y(0)=[y,\t]_b,
\;\;\;\;\;\;\l\in\CR.\label{6.64}
\end{gather}
According to Subsection \ref{sub6.4} there exists a unique
operator solution
\begin {equation}\label{6.65}
v(t,\l)=\begin{pmatrix} r_0(t,\l) & q_0(t,\l) \cr \wh r(t,\l) &
\wh q(t,\l) \cr r_1(t,\l) &
q_1(t,\l)\end{pmatrix}:\underbrace{\bC\oplus\bC}_{H_0}\to
\underbrace{\bC\oplus\bC\oplus\bC}_{\bH}, \quad \l\in\CR,
\end{equation}
of Eq. \eqref{6.59} belonging to $\lo{H_0}$ and satisfying the
boundary conditions \eqref{6.55}, which in our case take the form
\begin {equation*}
r_1(0,\l)=-1, \quad \wh r(0,\l)- [r,\t]_\infty=0, \quad
q_1(0,\l)=0, \quad \wh q(0,\l)- [q,\t]_\infty=-i.
\end{equation*}
The immediate checking shows that for $\l\in\bC_+$ such a solution
is
\begin {equation}\label{6.67}
v(t,\l)=\begin{pmatrix} ie^{i \l t} & \tfrac i {\sqrt 2}  e^{i \l
C} (e^{-i \l \Phi(t)}+ e^{i \l t}) \cr 0 & 0 \cr -e^{i \l t} &
\tfrac i {\sqrt 2}  e^{i \l C} (-ie^{-i \l \Phi(t)}+ ie^{i \l
t})\end{pmatrix}.
\end{equation}
Combining of \eqref{6.56} with \eqref{6.65} implies  that the
$m$-function of the  problem \eqref{6.63}, \eqref{6.64} is
\begin {equation*}
m(\l)=\begin{pmatrix} r_0(0,\l) & q_0(0,\l) \cr \wh r(0,\l) & \wh
q(0,\l)+\tfrac i 2 \end{pmatrix}, \quad \l \in\CR.
\end{equation*}
Therefore by \eqref{6.67} one has
\begin {equation*}
m(\l)=\begin{pmatrix}i & i\sqrt 2 e^{i\l C}\cr 0 & \frac i 2
\end{pmatrix}, \quad \l \in\bC_+.
\end{equation*}
Applying the Stieltjes formula \eqref{6.24} to $m(\cd)$  one
obtains the spectral function of the boundary problem
\eqref{6.63}, \eqref{6.64}:
\begin {equation}\label{6.69}
\Si(s)=\frac 1 \pi\begin{pmatrix} s & - \tfrac i {\sqrt 2 C}
(e^{isC}-1)\cr \tfrac i {\sqrt 2 C} (e^{-isC}-1 ) & \tfrac 1 2
s\end{pmatrix}
\end{equation}
Since $\Si (s)$ has the continuous derivative
\begin {equation}\label{6.70}
\Si'(s)=\frac 1 \pi\begin{pmatrix} 1 &  \tfrac 1 {\sqrt 2 }
e^{isC}\cr \tfrac 1 {\sqrt 2 } e^{-isC} & \tfrac 1 2
\end{pmatrix},
\end{equation}
it follows that $L^2(\Si; \bC^2)$ is the set of all functions
$g(\cd)$ such that
$$\int_\bR (\Si'(s)g (s), g(s))\, ds <\infty$$.

To simplify further considerations we pass to the new orthonormal
basis $\{\dot e_1,\dot e_2, \dot e_3 \}$ in $\bC^3$ with $\dot
e_1=\{\tfrac 1 {\sqrt 2}, 0, -\tfrac i {\sqrt 2} \}, \; \dot
e_2=\{ 0,1,0\}$ and $\dot e_3=\{\tfrac 1 {\sqrt 2}, 0, \tfrac i
{\sqrt 2} \}$. Then the Hilbert space $\LI$ can be identified with
the set of all Borel functions $f(\cd):\cI\to \bC^3$ of the form
\begin {equation*}
f(t)=\dot f_1(t)\dot e_1+\dot f_2(t)\dot e_2 +\dot f_3(t)\dot
e_3=:\{\dot f_1(t),\dot f_2(t), \dot f_3(t) \},
\end{equation*}
where $\d^{\frac 1 2}\dot f_1\in L^2(\cI)$ and $\dot f_2, \; \dot
f_3 \in \in L^2(\cI)$.

Next, the equality
\begin {equation*}
\f(t,\l)=\begin{pmatrix} \tfrac 1 2 (e^{-i\l\Phi (t)}+e^{i\l t}) &
0 \cr 0 & e^{-i\l t} \cr \tfrac i 2 (-e^{-i\l\Phi (t)}+e^{i\l t})
& 0)\end{pmatrix}:\underbrace{\bC\oplus\bC}_{H_0}\to
\underbrace{\bC\oplus\bC\oplus\bC}_{\bH}, \quad \l\in\bC,
\end{equation*}
defines the operator solution of Eq. \eqref{6.59} with $\f (0,\l)=
\begin{pmatrix} I_{H_0}\cr 0  \end{pmatrix}$. This and formula \eqref{6.18}
(with $\f_U(t,\l)=\f (t,\l)$) imply that for each function
$f(\cd)=\{\dot f_1 (\cd), \dot f_2 (\cd),\dot f_3 (\cd)\}\in \LI$
the Fourier  transform $\wh f(\cd)=\{\wh f_1 (\cd), \wh f_2
(\cd)\}\in L^2(\Si;\bC^2)$ is given by
\begin {equation*}
\wh f_1(s)= \tfrac 1 {\sqrt 2} \int_0^\infty (e^{i s
\Phi(t)}\d(t)\dot f_1(t)+ e^{-i s t} \dot f_3(t))\, dt, \qquad \wh
f_2(s)=\int_0^\infty e^{i s t}\dot f_2(t)\, dt.
\end{equation*}
According to Theorem \ref{th6.11} $Vf=\wh f$ is a unitary operator
from $\LI$ onto $L^2(\Si;\bC^2)$ and by using \eqref{6.38} one can
easily prove that the inverse Fourier  transform is
\begin{gather*}
\dot f_1(t)= \tfrac 1 {\pi\sqrt 2} \int_\bR (e^{-i s \Phi(t)}\wh
f_1(s)+ \tfrac
1 {\sqrt 2} e^{-is(\Phi(t)-C)}\wh f_2(s))\, d s ,\\
\dot f_2(t)= \tfrac 1 {\pi\sqrt 2} \int_\bR (e^{-i s (t+C)}\wh
f_1(s)+ \tfrac
1 {\sqrt 2} e^{-ist}\wh f_2(s))\, d s ,\\
\dot f_3(t)= \tfrac 1 {\pi\sqrt 2} \int_\bR (e^{i s t}\wh f_1(s)+
\tfrac 1 {\sqrt 2} e^{is(t+C)}\wh f_2(s))\, d s
\end{gather*}

\vskip 2mm

\centerline{\large Acknowledgments}

\vskip 2mm

The authors express their gratitude to DFG for financial support
within the Project No. 436 UKR 113/854.


\begin{thebibliography}{DHS}
\bibitem {AkhGla}
N.I.~Akhiezer, I.M. Glazman, \textit {Theory of linear operators
in Hilbert space}, Dover Publications, New York, 1993.

\bibitem {AlbMalMog12}
S. Albeverio,M.M. Malamud, V.I.Mogilevskii,\textit{On
Titchmarsh-Weyl functions of first-order symmetric systems with
arbitrary deficiency indices}, arXiv:1206.0479v1 [math.FA] 3 Jun
2012.


\bibitem{Atk}
F.V. Atkinson, \textit{Discrete and continuous boundary problems},
Academic Press, New York, 1963.

\bibitem{BHSW10}
J. Behrndt, S. Hassi, H. de Snoo,  R. Wiestma,
\textit{Square-integrable solutions and Weyl functions for
singular canonical systems}, Math. Nachr. \textbf{284} (2011), no
11-12, 1334--1383.

\bibitem{Ber}
Yu.M.~Berezanskii, \textit{Expansions in eigenfunctions of
selfadjoint operators}, Amer. Math. Soc., Providence, 1968.
(Russian edition: Naukova Dumka, Kiev, 1965).


\bibitem{Bru76}
V.M. ~Bruk, \textit{On a class of boundary value problems with
spectral parameter in the boundary condition}, Mathematics of the
USSR-Sbornik. \textbf{29} (1976), no2,  186--192.



\bibitem{DM00}
V.A. Derkach, S. Hassi, M.M. Malamud, H.S.V. de Snoo,\textit {
Generalized resolvents of symmetric operators and admissibility},
Methods of Functional Analysis and Topology \textbf{6} (2000),
no~3 , 24--55.


\bibitem{DM09}
V.A. Derkach, S. Hassi, M.M. Malamud,  H.S.V. de Snoo, \textit
{Boundary relations and generalized resolvents of symmetric
operators}, Russian J. Math. Ph. \textbf{16} (2009), no~1, 17--60.

\bibitem{DM91}
V.A.~Derkach,  M.M.~Malamud, \textit {Generalized resolvents and
the boundary value problems for Hermitian operators with gaps}, J.
Funct. Anal. \textbf{95} (1991),1--95.


\bibitem {DM95}
V.A. ~Derkach, M.M. ~Malamud, \textit {The extension theory of
Hermitian operators and the moment problem}, J. Math. Sciences
\textbf{73} (1995), no 2, 141-242.

\bibitem{DLS88}
A. Dijksma, H. Langer,  H.S.V. de Snoo, \textit{Hamiltonian
systems with eigenvalue depending boundary conditions}, Oper.
Theory Adv. Appl. \textbf{35} (1988), 37--83.

\bibitem{DLS93}
A. Dijksma, H. Langer,  H.S.V. de Snoo, \textit{Eigenvalues and
pole functions of Hamiltonian systems with eigenvalue depending
boundary conditions}, Math. Nachr.  \textbf{161} (1993), 107--153.


\bibitem{DunSch}
N. Dunford and J.T. Schwartz, \textit{Linear operators. Part2.
Spectral theory}, Interscience Publishers, New York-London, 1963.





\bibitem{Ful77}
Ch. T. Fulton, \textit{Parametrizations of Titchmarsh's
$m(\l)$-functions in the limit circle case}, Trans. Amer. Math.
Soc. \textbf{229} (1977), 51--63.

\bibitem{GK}
I. Gohberg, M.G. Krein, \textit{Theory and applications of
Volterra operators in Hilbert space}, Transl. Math. Monographs,
24, Amer. Math. Soc., Providence, R.I., 1970

\bibitem{Gor66}
M.L.~Gorbachuk,\, \textit{On spectral functios of a differential
equation of the second order with operator-valued coefficients},
Ukrain. Mat. Zh. \textbf{18} (1966), no~2 , 3--21.

\bibitem{GorGor}
V.I.~Gorbachuk,  M.L.~Gorbachuk, \textit{Boundary problems for
differential-operator equations}, Kluver Acad. Publ.,
Dordrecht-Boston-London, 1991. (Russian edition: Naukova Dumka,
Kiev, 1984).

\bibitem{HSW00}
S. Hassi, H.S.V. de Snoo, H. Winkler, \textit{Boundary-value
problems for two-dimensional canonical systems}, Integral
Equations Operator Theory \textbf{36} (2000), 445--479.

\bibitem{HinSha82}
D.B. Hinton, J.K. Shaw, \textit{Parameterization of the $M(\l)$
function for a Hamiltonian system of limit circle type}, Proc.
Roy. Soc. Edinburgh Sect. A \textbf{93} (1982/83), no 3-4,
349-360.

\bibitem{HinSch93}
D.B. Hinton, A. Schneider, \textit{On the Titchmarsh-Weyl
coefficients for singular S-Hermitian systems I}, Math. Nachr.
\textbf{163} (1993), 323--342.

\bibitem{HinSch06}
D.B. Hinton, A. Schneider, \textit{Titchmarsh-Weyl coefficients
for odd order linear Hamiltonian systems}, J. Spectral Mathematics
\textbf{1}(2006),1-36.

\bibitem{HinSch98}
D.B. Hinton, A. Schneider, \textit{On the spectral representation
for singular selfadjoint boundary eigenfunction problems}, Oper.
Theory: Advances and Applications, \textbf{106}, Birkhauser,
Basel, 1998.

\bibitem{Hol85}
A.M.~Khol'kin,\, \textit{Description of selfadjoint extensions of
differential operators of an arbitrary order on the infinite
interval in the absolutely indefinite case},  Teor. Funkcii
Funkcional. Anal. Prilozhen. \textbf{44} (1985), 112--122.

\bibitem{Kac50}
I.S. Kats, \textit{On Hilbert spaces generated by monotone
Hermitian matrix-functions}, Khar'kov. Gos. Univ. Uchen. Zap.
\textbf{34} (1950), 95-113=Zap.Mat.Otdel.Fiz.-Mat. Fak. i
Khar'kov. Mat. Obshch. (4)\textbf{22} (1950), 95-113.

\bibitem{Kac03}
I.S. Kats, Linear relations generated by the canonical
differential equation of phase dimension 2, and eigenfunction
expansion, St. Petersburg Math. J. \textbf{14}, 429--452 (2003).

\bibitem{KacKre}
I.S. Kac, M.G. Krein, \textit{On spectral functions of a string},
Supplement to the Russian edition of F.V. Atkinson ,
\textit{Discrete and continuous boundary problems}, Mir, Moscow,
1968.

\bibitem{Khr06}
V.I. Khrabustovsky, \textit{On the characteristic operators and
projections and on the solutions of Weyl type of dissipative and
accumulative operator systems. 3. Separated boundary conditions},
J. Math. Phis. Anal. Geom., \textbf{2} (2006), no 4, 449--473.


\bibitem{KogRof75}
 V.I.~Kogan, F.S.~Rofe-Beketov,\, \textit{On square-integrable solutions of
 symmetric systems of  differential equations of arbitrary order}, Proc. Roy.
Soc. Edinburgh Sect. A \textbf{74} (1974/75), 5--40.

\bibitem{Kov83}
I.V. Kovalishina, \textit {Analytic theory of a class of
interpolation problems}, Izv. Akad. Nauk SSSR Ser. Mat.,
\textbf{47} (1983), no 3,
 455–-497


\bibitem{Kra89}
A.M. Krall, \textit{$M(\l)$-theory for singular Hamiltonian
systems with one singular endpoint}, SIAM J. Math. Anal.
\textbf{20} (1989), 664--700.


\bibitem{LanTex77}
H.~Langer, B.~Textorious, \textit{On generalized resolvents and
$Q$-functions of symmetric linear relations (subspaces) in Hilbert
space}, Pacif. J. Math. \textbf{72}(1977), No.1 , 135-165.


\bibitem{LT82}
H. Langer, B. Textorius, \textit{$L$-resolvent matrices of
symmetric linear relations with equal defect numbers;
appliccations to canonical differential relations}, Integral
Equations Operator Theory  \textbf{5} (1982), 208--243.


\bibitem{LesMal03}
M. Lesch, M.M. Malamud, \textit{On the deficiency indices and
self-adjointness of symmetric Hamiltonian systems}, J.
Differential Equations \textbf{189} (2003), 556--615.
%
\bibitem {Mal92}
 M. M. ~Malamud, \textit{On the formula of generalized resolvents of a
nondensely defined Hermitian operator}, Ukr. Math. Zh.
\textbf{44}(1992), no 12, 1658-1688.
%
\bibitem {MalMal03}
M.M. Malamud, S.M. Malamud. \textit{Spectral theory of operator
measures in {H}ilbert space}
\newblock {St. Petersbg. Math. J.}, 15(3):323--373, 2003.


\bibitem {MalMog02}
M. M. Malamud, V. I. Mogilevskii, \textit{Krein type formula for
canonical resolvents of dual pairs of linear relations}, Methods
of Funct. Anal. and Topology \textbf{8}, No.4 (2002), 72-100.

\bibitem {MalNei12}
M. Malamud, H. Neidhardt,  \textit{Sturm-Liouville boundary value
problems with operator potentials and  unitary equivalence}, J.
Differential Equations, v.252 (2012), 5875-5922.


\bibitem{Mog06.1}
 V.I.Mogilevskii,\, \textit{Nevanlinna type families of linear relations and
the dilation theorem}, Methods  Funct. Anal.  Topology \textbf{12}
(2006), no~1, 38--56.

\bibitem{Mog06.2}
 V.I.Mogilevskii,\, \textit{Boundary triplets and Krein type resolvent formula
 for symmetric operators
with unequal defect numbers}, Methods  Funct. Anal.  Topology
\textbf{12} (2006), no~3, 258--280.

\bibitem{Mog07}
 V.I.Mogilevskii,\, \textit{Description of spectral functions of differential
 operators  with arbitrary deficiency indices}, Math. Notes \textbf{81}
(2007), no~4, 553--559.

\bibitem {Mog09.1}
V.I.Mogilevskii,\, \textit{Boundary triplets and Titchmarsh - Weyl
functions of differential operators with arbitrary  deficiency
indices },Methods  Funct. Anal.  Topology \textbf{15} (2009),
no~3, 280--300.



\bibitem{Mog11}
V.I.Mogilevskii, \textit{Boundary pairs and boundary conditions
for general (not necessarily definite) first-order symmetric
systems with arbitrary deficiency indices}, Math.
Nachr.\textbf{285} (2012), no14--15, 1895--1931

\bibitem{Nai}
M.A.~Naimark, \textit{Linear differential operators}, vol. 1 and 2
(Harrap, London, 1967).



\bibitem{Orc}
B.C. Orcutt, \textit{Canonical differential equations},
Dissertation, University of Virginia, 1969.

\bibitem{Rof69}
F.S.~Rofe-Beketov,\, \textit{Self-adjoint extensions of
differential operators in the space of vector-valued functions},
Teor. Funkcii Funkcional. Anal. Prilozhen. \textbf{8} (1969),
3--23.

\bibitem{Sht57}
A.V. \u{S}traus, \textit{On generalized resolvents and spectral
functions of differential operators of an even order}, Izv. Akad.
Nauk. SSSR, Ser.Mat., \textbf{21}, (1957), 785--808.





\end{thebibliography}
\end{document}